\documentclass[12pt,oneside]{book}
\usepackage{amsmath,amsfonts,amsfonts,amssymb}
\usepackage{amsbsy,theorem}
\usepackage{mathrsfs}
\usepackage{amscd}
\usepackage{pstricks}
\usepackage{color}
\usepackage{graphicx}

\theoremstyle{break}
\newtheorem{lemma}{Lemma}[section]

\newtheorem{proposition}[lemma]{Proposition}
\newtheorem{theorem}[lemma]{Theorem}
\newtheorem{corollary}[lemma]{Corollary}

\theorembodyfont{\upshape}
\newtheorem{remark}[lemma]{Remark}
\newtheorem{definition}[lemma]{Definition}
\newtheorem{example}[lemma]{Example}

\newcommand\beginproof[1]{%
   \trivlist\item[\hskip\labelsep{\bf #1.}]}
\newcommand\proof{\beginproof{Proof}}

\def\endproof{\newline \hspace*{\fill}\endproofsymbol\endtrivlist}
\def\endproofsymbol{\frame{\rule[0pt]{0pt}{8pt}\rule[0pt]{8pt}{0pt}}}

\newcommand \QQ {{\mathbb Q}}
\newcommand \PP {{\mathbb P}}

\newcommand{\CC}{\ensuremath{\mathbb{C}}}
\newcommand \GG {{\mathbb G}}
\newcommand \AAA {{\mathbb A}}
\newcommand \RR {{\mathbb R}}
\newcommand \ZZ {{\mathbb Z}}
\newcommand \cV {{\mathcal V}}
\newcommand \cM {{\mathcal M}}
\newcommand \cN {{\mathcal N}}
\newcommand \cl {{l}}

\newcommand{\sF}{\mathscr{F}}
\newcommand{\sign}{\mathrm{sign}}
\newcommand{\rank}{\mathrm{rank\,}}
\newcommand{\coker}{\mathrm{coker\,}}
\newcommand{\Spec}{\mathrm{Spec}\,}
\newcommand{\gr}{\mathrm{gr}\,}

\newcommand{\diag}{\mathrm{diag}}
\newcommand{\cd}{\mathrm{cd}}
\newcommand{\scs}{\scriptstyle}

\newcommand{\s}[1]{\ensuremath{\scs{#1}}}
\newcommand{\so}{\ensuremath{\scs{0}}}
\newcommand{\sk}{$\scs{\!\!-1}$}
\newcommand{\sj}{$\scs{1}$}

\begin{document}

\begin{titlepage}
  \begin{center}
    \Large
    {\bf
      Cohomology of graph  hypersurfaces
      associated to certain Feynman graphs}     \\
    \end{center}
    \vspace*{0.5cm}
    \begin{center}
    \textsc{Dzmitry Doryn}\\
    \medskip
    1 September 2008
    \end{center}
  \vspace{0.5cm}
  \begin{center}
    \textbf{Abstract}\\
  \end{center}

\noindent To any Feynman graph (with $2n$ edges) we can associate a
hyper\-surface \hbox{$X\subset\PP^{2n-1}$}. We study the middle
cohomology $H^{2n-2}(X)$ of such hyper\-surfaces. S. Bloch, H.
Esnault, and D. Kreimer (Commun. Math. Phys. 267, 2006) have
computed this cohomology for the first series of examples, the wheel
with spokes graphs $WS_n$, $n\geq 3$. Using the same technique, we
introduce the generalized zigzag graphs and prove that
$W_5(H^{2n-2}(X))=\QQ(-2)$ for all of them (with $W_{\ast}$ the
weight filtration). Next, we study prim\/itively log divergent
graphs with small number of edges and the beh\/avior of graph
hypersurfaces under the gluing of graphs.

This paper is my thesis at the university of Duisburg-Essen.

\end{titlepage}

%
%

\tableofcontents

\chapter*{Introduction} \addcontentsline{toc}{chapter}{Introduction}
There are interesting zeta and multi zeta values appearing in the
calculation of Feynman integrals in physics. One hopes that there
exist Tate mixed Hodge structures with periods given by Feynman
integrals at least for some identifiable subset of graphs.

This paper is a natural continuation of the work started in
\cite{BEK}. For technical reasons we restrict our attention to
primitively log divergent graphs. In \cite{BEK}, Sections 11,12 the
series $WS_n$ was work out in all details. Let
$X_n\subset\PP^{2n-1}$ be the graph hypersurface for the graph
"wheel with $n$ spokes" $WS_n$, it was proved that (as a Hodge
structure)
\begin{equation*}
    H^{2n-1}_c(\PP^{2n-1}\backslash X)\cong \QQ(-2)
\end{equation*}
and that the de Rham cohomology $H^{2n-1}_{DR}(\PP^{2n-1}\backslash
X)$ is generated by the integrand of graph period (\ref{h66}). In
this paper we succeeded to do the same computation for the graph
$ZZ_5$, the zigzag graph with Betti's number equals 5. We define a
big series of graphs for which the minimal nontrivial weight piece
of Hodge structure is of Tate type: $\gr^W_{min}
H^{2n-1}_c(\PP^{2n-1}\backslash X)\cong\QQ(-2)$ for graph
hypersurfaces $X$ in this situation. We study gluings of primitively
log divergent graphs and compute $\gr^W_{min}
H^{2n-1}_c(\PP^{2n-1}\backslash X)=\QQ(-3)$ for the case $WS_3\times
WS_n$.

This paper is organized as follows. Section 1.1 contains some
theorems on determinants, this is a key ingredient of our
computation. The second section is a remainder of the construction
of graph polynomials and Feynman integrals. The cohomological tools
are presented in Section 1.3.

In Section 2.1 we compute the middle dimensional cohomology of the
graph $ZZ_5$. We define \emph{generalized zigzag graphs} GZZ and
prove that they are primitively log divergent in Section 2.2. Then
we present the main result that the minimal nontrivial weight pieces
of mixed Hodge structures of such graphs are Tate. In Section 2.3 it
is proved that the integrand (\ref{h66}) is nonzero in the de Rham
cohomology $H^{2n-2}_{DR}(\PP^{2n-1}\backslash X)$ for $GZZ(n,2)$
and generates it in the case of $ZZ_5$.

Section 3.1 contains the classification of primitively log divergent
graphs with small number of edges. In the next section we compute
the cohomology for the new found graph with 10 edges, the graph
$XX_5$. This graph is obtained from two copies of the graph $WS_3$
by the operation of gluing. We study gluings of primitively log
divergent graphs in Section 3.3 and try to compute the middle
dimensional cohomology for the series $WS_3\times WS_n$.
\medskip\\

I would like to thank  Prof. Dr. H\'el\`ene Esnault and Prof. Dr.
Eckart Viehweg for the possibility to come to Germany, to join their
research group and to study algebraic geometry. I am very grateful
to my advisor  Prof.\,Dr. \,H\'el\`ene Esnault for her constant
support and guidance throughout the preparation of this thesis. I
would like to thank Dr. Kay R\"ulling for reading this thesis and
for help and Dr. Georg Hein for useful comments. Finally, I want to
thank all my colleagues at the university of Essen for the nice
friendly atmosphere during this three wonderful years.

\chapter{Preliminaries}
\section{Determinants}
\noindent Fix some commutative ring $R$ with 1 and let
$\cM=(a_{ij})_{0\leq i,j\leq n}$ be an
\hbox{${(n{+}1)}{\times}{(n{+}1)}$}-matrix with entries in $R$. The
numeration of rows and columns goes $0$ through $n$. Let
$\cM(i_0,\ldots,i_k;j_0,\ldots,j_t)$ be the submatrix which we get
from the matrix after removing rows $i_0$ to $i_k$ and the columns
$j_0$ to $j_k$. It is very convenient to denote the determinant of
$\mathcal{M}$ just by $M$. We assume that the determinant of
zero-dimensional matrix is 1. For example, $M(0,n;0,n)=1$ for the
matix in the definition above with $t=n=1$.

\begin{theorem}\label{T1.1}
Let $n\geq 1$. For any ${(n+1)}\times{(n+1)}$-matrix $\cM$ and any
integers $0\leq i$, $j,k,t\leq n$, satisfying $i\neq k$ and $j\neq
t$, we have
\begin{equation}
    M(i;j)M(k;t)-M(k;j)M(i;t)=M\cdot M(i,k;j,t).
    \label{b14}
\end{equation}
\end{theorem}
\proof First, we show that it is enough to prove the statement for
the case $i=j=0$ and $k=t=n$. Fix a matrix $\cM$ and some $i$, $j$,
$k$ and $t$. Let $\widetilde{\cM}$ be a matrix that we get from
$\cM$ after interchanging the pairs of rows $0\leftrightarrow i$,
$k\leftrightarrow n$ and columns $0\leftrightarrow j$,
$t\leftrightarrow n$. Notice that the operation of interchanging
rows commutes with the operation of interchanging columns. Suppose
that the statement of the theorem is true for $\widetilde{\cM}$:
\begin{equation}
    \widetilde{M}(0;0)\widetilde{M}(n;n) - \widetilde{M}(0;n)\widetilde{M}(n;0)=
    \widetilde{M} \widetilde{M}(0,n;0,n).
    \label{b15}
\end{equation}
Both matrices $\cM(k;t)$ and $\widetilde{\cM}(n;n)$ have the same
missing column and row, namely the $t$-th column and $k$-th row of
$\cM$, thus the matrices differ only by the order of rows and
columns. This means that the determinants $M(k;t)$ and
$\widetilde{M}(n;n)$ are the same up to sign. But each interchange
of two rows or two columns changes the sign of a determinant, hence
$$
    M(k;t)=\widetilde{M}(n;n).
$$
Similarly,
\begin{eqnarray}
    M(k;j)& =&\widetilde{M}(n;0)\\
    M(i;t)& =&\widetilde{M}(0;n)\\
    M(i;j)& =&\widetilde{M}(0;0)\\
    M(i,k;j,t) &=&\widetilde{M}(0,n;0,n)\\
    M &=&\widetilde{M}
\end{eqnarray}
Thus, we can assume that $i=j=0$ and $k=t=n$.\pagebreak[3]
\par\noindent
We prove the statement by induction on the dimension of $\cM$.
Suppose that for all $r\times r$-matrices with $r\leq n$ and all
$i$, $j$, $k$ and $t$ the statement is true. The strategy is to
present the polynomials of both sides of (\ref{b15}) as polynomials
of the variables which are entries of the first and last row and
column. For simplicity, we denote by $\cN$ the matrix $\cM(0,n;0,n)$
and define $I$ to be the set $\{1, 2,\ldots, n-1 \}$. We start with
$M(n,n)$ and, using the Laplace expansion along the zero column and
then along the zero row, we get
\begin{equation}
\begin{split}
    M(n,n)&=a_{0 0}M(0,n;0,n)+\sum_{i\in I} (-1)^i a_{i 0}M(i,n;0,n)\\
    &=a_{0 0}N + \sum_{i\in I}(-1)^i a_{i 0}
    \sum_{j\in I}(-1)^{j-1}a_{0 j}M(0,i,n;0,j,n)\\
    &=a_{0 0}N + \sum_{i,j\in I}(-1)^{i+j-1} a_{i 0}a_{0 j}N(i;j).
\end{split}
     \label{b19}
\end{equation}
Hence the left hand side of (\ref{b14}) equals
\begin{equation*}
\begin{aligned}
    LHS &=&&
     \Bigl(a_{0 0}N+\sum_{i,j\in I} (-1)^{i+j-1}a_{i 0}a_{0 j}N(i;j)\Bigr)\\
    & &\cdot&\Bigl(a_{n n}N +\sum_{k,t\in I} (-1)^{k+t-1}a_{k n}a_{n t}N(k;t)\Bigr)- \\
\end{aligned}
\end{equation*}
\begin{equation}
\begin{aligned}
    &&-& \Bigl((-1)^{n-1} a_{n 0}N + \sum_{i,t\in I} (-1)^{i+t+n}a_{i 0}a_{n
    t}N(i;t)\Bigr)\\
    & &\cdot&\Bigl((-1)^{n-1} a_{0 n}N + \sum_{k,j\in I} (-1)^{k+j+n}a_{0 j}a_{k n}N(k;j)\Bigr)\\
    &=& &a_{00}a_{nn}N^2 + a_{n n}N\sum_{i,j\in I} (-1)^{i+j-1}a_{i 0}a_{0
    j}N(i;j)\\
    && + & a_{0 0}N\sum_{k,t\in I} (-1)^{k+t-1}a_{k n}a_{n t}N(k; t) \\
    && + &\sum_{i,j,k,t\in I} (-1)^{i+j+k+t}a_{i 0}a_{0 j}a_{k n}a_{n
        t}N(i;j)N(k;t)\\
    &&-& a_{n 0}a_{0 n}N^2
    -  a_{0 n} N \sum_{i,t\in I}(-1)^{i+t-1} a_{i 0}a_{n t}N(i;t)\\
    & &-&  a_{n 0} N \sum_{k,j\in I}(-1)^{k+j-1} a_{0 j}a_{k n}N(k;j)\\
    &&-& \sum_{i,j,k,t\in I}(-1)^{k+j+i+t} a_{i 0}a_{n t}a_{0 j}a_{k
    n}N(i;t)N(k;j).
    \label{b20}
\end{aligned}                   
\end{equation}
By the assumption, for $i\neq k$ and $j\neq t$ one has
\begin{equation}
    N(i;j)N(k;t)-N(i;t)N(k;j)=NN(i,k;j,t).
    \label{b21}
\end{equation}

Note that, in the case $i=k$ or $j=t$, the products $N(i;j)N(k;t)$
and $N(i;t)N(k;j)$ are the same, and then all summands of
(\ref{b20}) which are multiples of $N(i;j)N(k;t)$ or $N(i;t)N(k;j)$
cancel. Thus we can rewrite the big sum (\ref{b20}) in the following
way

\begin{equation}
\begin{aligned}
    LHS &=&(&a_{00}a_{nn}-a_{0 n}a_{n 0})N^2\\
    &&+& a_{n n}N\sum_{i,j\in I} (-1)^{i+j-1}a_{i 0}a_{0 j}N(i;j)\\
    &&+& a_{0 0}N\sum_{k,t\in I} (-1)^{k+t-1}a_{k n}a_{n t}N(k; t)\\
    &&+& a_{0 n} N \sum_{i,t\in I}(-1)^{i+t} a_{i 0}a_{n t}N(i;t)\\
    &&+& a_{n 0} N \sum_{k,j\in I}(-1)^{k+j} a_{0 j}a_{k n}N(k;j)\\
    &&+&N\sum_{\substack{i\neq k\in I\\j\neq t\in I}} (-1)^{i+j+k+t}a_{i 0}a_{0 j}a_{k n}a_{n
        t}N(i,k;j,t).
    \label{b22}
\end{aligned}
\end{equation}
The last sum is equal to $NM$. Indeed, define $I'$ to be
$I\cup\{n+1\}$. We expand $M$ along the zero row and the zero column
using a formula like (\ref{b19}) and then expand it further along
the $n$-th row and the $n$-th column:
\begin{equation}
\begin{aligned}\label{b24}
    M&=&&a_{0 0}M(0;0)+\sum_{i,j\in I'}(-1)^{i+j-1} a_{i 0}a_{0 j}M(0,i;0,j)\\
     &=&&a_{0 0}M(0;0)-a_{n 0}a_{0 n}M(0,n;0,n)\\
     &&+&a_{n 0}\sum_{j\in I}(-1)^{n+j-1} a_{0 j}M(0,n;0,j)\\
     &&+&a_{0 n}\sum_{i\in I}(-1)^{i+n-1} a_{i
     0}M(0,i;0,n)\\
     &&+&\sum_{i,j\in I}(-1)^{i+j-1} a_{i 0}a_{0 j}M(0,i;0,j)\\
     &=&&a_{0 0}\Bigl(a_{n n}M(0,n;0,n)
            +\sum_{k,t\in I}(-1)^{k+t-1} a_{k n}a_{n t}M(0,i,n;0,j,n)\Bigr)\\
     &&-&a_{n 0}a_{0 n}M(0,n;0,n)
            +\sum_{j,k\in I}(-1)^{k+j} a_{n 0}a_{0 j}a_{k n}M(0,k,n;0,j,n)\\
     &&+& \sum_{i,t\in I}(-1)^{i+t} a_{i 0}a_{0 n}a_{n
     t}M(0,i,n;0,t,n)\\
     &&+&\sum_{i,j\in I}(-1)^{i+j-1}
     a_{i 0}a_{0 j}\Bigl( a_{n n}M(0,i,n;0,j,n)\\
     &&+&\sum_{\substack{k,t\in I\\i\neq k,j\neq t}}(-1)^{k+t-1} a_{k n}a_{n
     t}M(0,i,k,n;0,j,t,n)\Bigr)
\end{aligned}
\end{equation}
Hence the sum (\ref{b22}) equals $MN$.
\endproof
\noindent For a matrix $\cM=(a_{ij})_{0\leq i,j\leq n}$ we define
the minors
\begin{equation}
\begin{aligned}
    I^i_k:=M(0,1,\ldots,&i-1,i+k,i+k+1,\ldots,n;\\
    &0,1,\ldots,i-1,i+k,i+k+1,\ldots,n).
    \label{b25}
\end{aligned}
\end{equation}
and
\begin{equation}
    S_t:=M(t,t+1,\ldots,n; 0,t+1,t+2,\ldots,n),
    \label{b26}
\end{equation}
where $1\leq k\leq n+1$ and $1\leq t\leq n$. We usually write $I_n$
for $I^0_n$. For example, $I_{n+1}=M$, $I^1_n=M(0;0)$ and
$I_n=M(n;n)$.

\begin{corollary} \label{C1.2}
For a symmetric matrix $\cM=(a_{ij})_{0\leq i,j\leq n}$ one has the
following equality
    \begin{equation}
        I_n I^1_n-I^1_{n-1}I_{n+1}=(S_n)^2.
    \end{equation}
\end{corollary}
\begin{proof}
Since $\cM$ is symmetric, $M(0;n)=M(n;0)=S_n$ and the statement
follows immediately from Theorem \ref{T1.1}.
\end{proof}
Take now an $(n+1)\times(n+1)$-matrix $\cM=(a_{ij})$ with entries in
$R$ and suppose that the transpose of the last row equals the last
column with elements, numerated by single lower indices.
\begin{equation}
\cM= \left(
  \begin{array}{ccccccc}
    a_{0 0}     & a_{0 1}    &  \vdots  & a_{0\, n-2}  &  a_{0\, n-1}   & a_0    \\
    a_{1 0}     & a_{1 1}    &  \vdots  & a_{1\,n-2}   &  a_{1\, n-1}   & a_1    \\
    \cdots      &\cdots      &\ddots    &\cdots        & \cdots         & \cdots \\
    a_{n-2\, 0} & a_{n-2\,1} & \vdots   & a_{n-2\,n-2} &  a_{n-2\, n-1} & a_{n-2}\\
    a_{n-1\, 0} & a_{n-1\,1} & \vdots   & a_{n-1\,n-2} &  a_{n-1\, n-1} & a_{n-1}\\
    a_0         & a_1        & \vdots   & a_{n-2}      &  a_{n-1}       & a_n    \\
  \end{array}
\right). \label{b28}
\end{equation}
The determinant of $\cM$ is thought of as an element in
$R[a_0,\ldots,a_n]$. \noindent It can be written as
\begin{equation}
    M=I_{n+1}=a_nI_n-G_n,
    \label{b29}
\end{equation}
defining $G_n$. Then $G_n$ is computed as
\begin{equation}
    G_n:=\!\!\!\sum_{0\leq i, j\leq n-1}\!\!\!(-1)^{i+j}
    a_ia_jI_n(i;j).
    \label{b30}
\end{equation}
The entries $a_i$ play the role of variables while the other entries
and minors are coefficients. The element $G_n\in R[a_0,\ldots,a_n]$
is of degree 2 as a polynomial of the variables. We claim
\begin{theorem}\label{T1.3}
    Let $I_{n-1}\not\equiv 0 \!\mod I_n$. Then
    \begin{equation}
        I_{n-1}G_n\equiv Li_n Li'_n \mod I_n
        \label{b31}
    \end{equation}
    for some $Li_n$ and $Li'_n$, linear as polynomials of the
    "variables".
\end{theorem}
\begin{proof}
By Theorem \ref{T1.1}, we have
\begin{equation}
     I_n(i;j)I_n(n-1;n-1)\equiv I_n(i;n-1)I_n(n-1;j) \mod I_n
\end{equation}
for all $1\leq i,j\leq n-2$. We multiply $G_n$ by
$I_{n-1}=I_n(n-1,n-1)$ and get
\begin{equation}
    \begin{aligned}
        I_{n-1}G_n&=&&a_{n-1}^2(I_{n-1})^2+\!\!\!\sum_{0\leq i, j\leq {n-2}}\!\!\!
    (-1)^{i+j}a_ia_jI_{n-1}I_n(i;j)\\
    &&+& a_{n-1}\sum_{0\leq i\leq n-2}(-1)^{i+n-1}a_iI_{n-1}\Bigl(I_n(i;n-1)+I_n(n-1;i)\Bigr)\\
    &\equiv&& \biggl(a_{n-1}I_{n-1}+ \sum_{0\leq i\leq
    n-2}(-1)^{i+n-1}a_iI_n(i;n-1)\biggl)\\
    &&\cdot&\biggl(a_{n-1}I_{n-1}+ \sum_{0\leq j\leq
    n-2}(-1)^{j+n-1}a_jI_n(n-1;j) \biggr) \mod I_n.
    \end{aligned}
    \label{b33}
\end{equation}
We set
  \begin{equation}
    Li_n=a_{n-1}I_{n-1}+ \sum_{0\leq i\leq
    n-2}(-1)^{i+n-1}a_iI_n(i;n-1)
    \label{b34}
  \end{equation}
and
  \begin{equation}
    Li'_n=a_{n-1}I_{n-1}+ \sum_{0\leq j\leq
    n-2}(-1)^{j+n-1}a_iI_n(n-1;j).
  \end{equation}
\end{proof}
In the next chapters we deal only with symmetric matrices, thus we
make a
\begin{corollary}\label{C1.4}
    Let $\cM=(a_{ij})$ be a symmetric $(n+1)\times(n+1)$-matrix
    with entries in a ring $R$
    (see (\ref{b28})). If $I_{n-1}\not\equiv 0 \!\mod I_n$, the
    congruence
    \begin{equation}
        I_{n-1}G_n\equiv (Li_n)^2 \mod I_n
        \label{b36}
    \end{equation}
    holds, where $G_n$ and $Li_n$ are given
    by (\ref{b30}) and (\ref{b34}) respectively.
\end{corollary}
One more fact about $G_n$ will be used frequently in the next
chapters.
\begin{theorem}\label{T1.5}
    Let $\cM=(a_{ij})$ be a symmetric $(n+1)\times(n+1)$-matrix
    with entries in a ring $R$ and assume that the quotient ring
    $R/(I_n)$ is a domain.
    If $I_{n-1}\equiv 0 \!\mod I_n$, then
    \begin{equation}
        G_n\equiv \sum_{0\leq i\leq n-2} a_i^2I_n(i,i)+
        2\!\!\!\sum_{0\leq i< j\leq n-2}\!\!\!(-1)^{i+j}
    a_ia_jI_n(i;j) \mod I_n.
    \label{b37}
    \end{equation}
\end{theorem}
\begin{proof}
    We prove that $G_n$ forgets the "variable" $a_{n-1}$.
    Since $\cM$ is symmetric, we can rewrite $G_n$ (see (\ref{b30}))
    as
    \begin{equation}
    \begin{aligned}
        G_n&=&&a_{n-1}^2I_{n-1}+\sum_{0\leq i\leq n-2}(-1)^{i+n-1}a_ia_{n-1}\Bigl(I_n(i,n-1)+I_n(n-1,i)\Bigr)\\
        &&+&\!\!\!\sum_{0\leq i\leq j\leq n-2}\!\!\!(-1)^{i+j}a_ia_jI_n(i;j)\\
        &=&& a_{n-1}^2I_{n-1} + 2\sum_{0\leq i\leq n-2}(-1)^{i+n-1}a_ia_{n-1}I_n(i,n-1)\\
        &&+&\sum_{0\leq i\leq n-2}a_i^2I_n(i,i)+2\!\!\!\sum_{0\leq i< j\leq
        n-2}\!\!\!(-1)^{i+j}a_ia_jI_n(i;j).
    \end{aligned}
    \label{b38}
    \end{equation}
    Theorem \ref{T1.1} implies
    \begin{equation}
        (I_n(i,n-1))^2\equiv I_n(i,i)I_n(n-1,n-1) \mod I_n
    \end{equation}
    for all $0\leq i\leq n-2$. 
    Because $R/(I_n)$ is a domain and $I_{n-1}=I_n(n-1,$ $n-1)\equiv 0 \!\mod I_n$, we get
    \begin{equation}
        I_n(i,n-1)\equiv 0 \mod I_n.
    \end{equation}
    Hence, (\ref{b38}) implies the congruence
    \begin{equation}
        G_n\equiv \sum_{0\leq i\leq n-2}a_i^2I_n(i,i)+2\!\!\!\sum_{0\leq i< j\leq
        n-2}\!\!\!(-1)^{i+j}a_ia_jI_n(i;j)\mod I_n.
        \label{b41}
    \end{equation}
\end{proof}
\begin{remark}
    In our computations we will apply Corollary \ref{C1.4} and Theorem
    \ref{T1.3} only for the case where $R$ in a polynomial ring over
    an algebraically closed field of characteristic zero. Moreover, entries of matrices will be only linear
    polynomials in $R$.
\end{remark}
\noindent Take a symmetric $(n+1)\times(n+1)$-matrix $\cM=(a_{ij})$
and numerate the variables in the last column and row by $a_i$'s
with single lower indices as in (\ref{b28}). We assume that the
entries are in  the ring $R=K[x_0,...,x_m]$, for some field
$K=\bar{K}$ of char 0, for some set of variables $X=\{x_0,\ldots,
x_m\}$ and for some $m$ (in the matrices associated to graphs we
have always $m = 2n+1$). Assume that the $a_i$ are in $X\cup\{0\}$
for $0\leq i\leq n-2$ and all nonzero $a_i$ for $0\leq i\leq n$ are
mutually different. Without loss of generality, we can assume
$x_n=a_n$ and $x_{n-1}=a_{n-1}$.\medskip
\newline \noindent Consider the projective space
$\PP^m=\PP^m(x_0:\ldots:x_m)$. The vanishing of polynomials in $R$,
or determinants of submatrices of $\cM$, or $\cM$ itself define
hypersurfaces in this projective space. Throughout the whole paper,
for a finite set $f_1, f_2, \dots$ of homogeneous polynomials we
denote by $\cV(f_1,f_2,...)$ the corresponding reduced projective
scheme. We consider the following situation. Let
\begin{equation}
    V:=\cV(I_n,I_{n+1})\in\PP^m,
\end{equation}
and define an open $U$ in $V$ by
\begin{equation}
    \left\{
        \begin{aligned}
        I_n=0\\
        I_{n+1}=0\\
        I_{n-1}\neq 0.\!\!
        \end{aligned}
    \right.
    \label{b43}
\end{equation}
This is equivalent to say $U:=V\backslash V\cap\cV(I_{n-1})$; we
usually write down such systems in order to see the way of
stratifying the schemes further. We write
\begin{equation}
    I_{n+1}=a_nI_n - G_n
\end{equation}
and see that $U$ is defined by the system
\begin{equation}
    \left\{
        \begin{aligned}
        I_n=0\\
        G_n=0\\
        I_{n-1}\neq 0.\!\!\\
        \end{aligned}
    \right.
    \label{b44}
\end{equation}
These three polynomials are independent of $a_n$. Let $P_1\in\PP^m$
be the point where all variables but $a_n$ vanish. Consider the
natural projection $\pi_1:\PP^m\backslash Pt_1\longrightarrow
\PP^{m-1}$ and denote by $V_1$ resp. $U_1$ the images of $V$ resp.
$U$ under $\pi_1$, i.e.
\begin{equation}
V_1=\cV(I_n,G_n)\subset\PP^{m-1}
\end{equation}
and $U_1\subset\PP^{m-1}$ defined by the system (\ref{b44}). By the
Corollary \ref{C1.4}, on $U$ and $U_1$ we have
\begin{equation}
    I_{n-1}G_n=(Li_n)^2,
\end{equation}
thus the vanishing of $G_n$ is equivalent to the vanishing of
$Li_n$. This means that $U_1$ is defined by
    \begin{equation}
    \left\{
        \begin{aligned}
        I_n=0\\
        Li_n=0\\
        I_{n-1}\neq 0.\!\!\\
        \end{aligned}
    \right.
    \label{b46}
\end{equation}
On $U_1$ we can express $a_{n-1}$ from the equation
\begin{equation}
    Li_n=a_{n-1}I_{n-1}+ \sum_{1\leq i\leq
    n-2}(-1)^{i+n-1}a_iI_n(i;n-1)=0
    \label{b47}
\end{equation}
and get
\begin{equation}
    a_{n-1}=\sum_{0\leq i\leq n-2}(-1)^{i+n-1}\frac{a_iI_n(i;n-1)}{I_{n-1}}.
    \label{b48}
\end{equation}
\par
\noindent Of course, the right hand side of the last equality can
depend on $a_{n-1}$. If this is not the case, we consider the point
$P_2\subset\PP^{m-1}$ where all variables but $a_{n-1}$ are zeros
and claim
\begin{theorem}\label{T1.7}
    Suppose that all entries $a_{i j}$ of $\cM$ are independent of
    $a_{n-1}$.
    The natural projection $$\pi_2:\PP^{m-1}\backslash P_2\longrightarrow
    \PP^{m-2}$$ induces an isomorphism between $U_1$ and an open
    $U_2\subset\PP^{m-2}$ defined by
    \begin{equation}
    U_2:=\cV(I_n)\backslash\cV(I_n,I_{n-1}).
    \end{equation}
\end{theorem}
\begin{proof} $\pi_2(U_1)=U_2$ and the expression (\ref{b48}) for $a_{n-1}$ gives the map
    \begin{equation}
    \phi:U_2\longrightarrow\PP^{m-1}
    \end{equation}
    inverse to $\pi_2$.
\end{proof}

\newpage

\section{Graph polynomials}
%
%
Let $\Gamma$ be a finite graph with edges $E$ and vertices $V$. We
choose an orientation of edges. For a given vertex $v$ and a given
edge $e$ we define $\sign(e,v)$ to be $-1$ if $e$ enters $v$ and
$+1$ if $e$ exits $v$. Denote by $\ZZ[E]$ (resp. $\ZZ[V]$) the free
$\ZZ$-module generated by the elements of $E$ (resp. $V$). Consider
the homology sequence
\begin{equation}
    0\longrightarrow H_1(\Gamma,\ZZ)\xrightarrow{\;\;\iota\;\;}
    \ZZ[E]\xrightarrow{\;\;\partial\;\;} \ZZ[V]
    \longrightarrow H_0(\Gamma,\ZZ) \longrightarrow 0,
    \label{h44}
\end{equation}
where the $\ZZ$-linear map $\partial$ is defined by
$\partial(e)=\sum_{v\in V} sign(v,e)$. The elements $e^{\vee}$ of a
dual basis of $\ZZ[E]$ define linear forms $e^{\vee}\circ\iota$ on
$H=H_1(\Gamma,\ZZ)$. We view the squares of this functions
$(e^{\vee}\circ\iota)^2:H\rightarrow \ZZ$ as rank 1 quadratic forms.
For a fixed basis of $H$ we can associate a rank 1 symmetric matrix
$M_e$ to each such form.
\begin{definition}
We define the \emph{graph polynomial} of $\Gamma$
\begin{equation}
    \Psi_{\Gamma}:=\det (\sum_{e\in E}A_eM_e)
    \label{h46}
\end{equation}
in some variables $A_e$.
\end{definition}
The polynomial $\Psi$ is homogeneous of degree $\rank H$. A change
of the basis of $H$ only changes $\Psi_{\Gamma}$ by $+1$ or $-1$.
\begin{definition}
    The \emph{Betti number} of a graph $\Gamma$ is defined to be
    \begin{equation}
        h_1(\Gamma)=\textrm{rank}\; H_1(\Gamma,\ZZ).
    \end{equation}
\end{definition}
Recall that a tree $T\subset\Gamma$ is a \emph{spanning tree} for
the connected graph $\Gamma$ if every vertex of $\Gamma$ lies in
$T$. We can extend this notion to a disconnected $\Gamma$ by simply
requiring $T\cap\Gamma_i$ be a spanning tree for each connected
component $\Gamma_i\subset\Gamma$. The following proposition (see
\cite{BEK}, Proposition 2.2) is often used as a definition of graph
polynomial.
\begin{proposition}\label{Pr2.3}
    With notation as above, we have
    \begin{equation}
        \Psi_{\Gamma}(A)= \sum_{T\;\text{span}\;tr.}\prod_{e\not\in T} A_e.
    \end{equation}
\end{proposition}
\begin{corollary}\label{Cor1.2.4}
    The coefficients of $\Psi_{\Gamma}$ are all either $0$ or $+1$.
\end{corollary}

For the graph $\Gamma$ we build the table $Tab(\Gamma)$ with
$h(\Gamma)$ rows and $|E(\Gamma)|$ columns. Each row corresponds to
a loop of $\Gamma$, and these loops form a basis of
$H_1(\Gamma,\ZZ)$. For each such loop we choose some direction of
loop tracing. The entry $Tab(\Gamma)_{i j}$ equals $1$ if the edge
$e_j$ in the $i$'s loop is in the tracing direction of the loop and
equals $-1$ if this edge is in the opposite direction; if the edge
$e_j$ does not appear in the $i$'s loop, then $Tab(\Gamma)_{i j}=0$.
We take $N:=|E(\Gamma)|$ variables $T_1,\ldots,T_N$ and build a
matrix
\begin{equation}
    M_{\Gamma}(T)=\sum_{k=1}^N T_k M^k
\end{equation}
where $M^k$ is a $h_1(\Gamma)\times h_1(\Gamma)$ matrix with entries
\begin{equation}
M^d_{i j}=Tab(\Gamma)_{i d}\cdot Tab(\Gamma)_{j d}.
\end{equation}
By definition, the graph polynomial for $\Gamma$ is
\begin{equation}
    \Psi_{\Gamma}(T)=\det M_{\Gamma}(T).
\end{equation}
Consider the following example which will appear in section $2.1$.
\begin{example}\label{Ex2.4}
Let $\Gamma$ be the graph $ZZ_5$ (see the drawing). This graph has
10 edges and the Betti number equals 5.\\
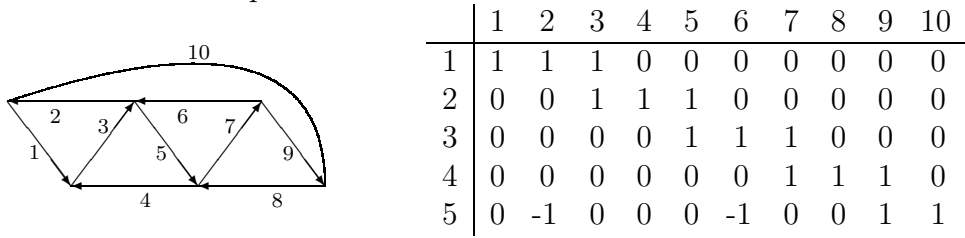
\begin{picture}(0,0)(0,40)
\put(10,50){\vector(3,-4){24}} \put(18,28){$\scriptstyle 1$}
\put(58,50){\vector(-1,0){48}} \put(26,42){$\scriptstyle 2$}
\put(34,18){\vector(3,4){24}} \put(44,38){$\scriptstyle 3$}
\put(82,18){\vector(-1,0){48}} \put(60,10){$\scriptstyle 4$}
\put(58,50){\vector(3,-4){24}} \put(66,28){$\scriptstyle 5$}
\put(106,50){\vector(-1,0){48}} \put(74,42){$\scriptstyle 6$}
\put(82,18){\vector(3,4){24}} \put(92,38){$\scriptstyle 7$}
\put(130,18){\vector(-1,0){48}} \put(110,10){$\scriptstyle 8$}
\put(106,50){\vector(3,-4){24}} \put(114,28){$\scriptstyle 9$}
\qbezier(10,50)(130,90)(130,18) \put(78,66){$\scriptstyle 10$}
\end{picture}
\mathstrut\qquad\qquad\qquad\qquad\qquad\qquad\qquad
\begin{tabular}{c|c c c c c c c c c c|}
      & 1 & 2 & 3 & 4 & 5 & 6 & 7 & 8 & 9 &\!10\\ \hline
    1 & 1 & 1 & 1 & 0 & 0 & 0 & 0 & 0 & 0 & 0\\
    2 & 0 & 0 & 1 & 1 & 1 & 0 & 0 & 0 & 0 & 0\\
    3 & 0 & 0 & 0 & 0 & 1 & 1 & 1 & 0 & 0 & 0\\
    4 & 0 & 0 & 0 & 0 & 0 & 0 & 1 & 1 & 1 & 0\\
    5 & 0 &\!-1 & 0 & 0 & 0 &\!-1 & 0 & 0 & 1 & 1\\
\end{tabular}
\bigskip\\ We choose the orientation and the numbering of edges as on
the drawing to the left. Following the construction above, we build
the table $Tab(ZZ_5)$ to the right and get the following matrix
\begin{equation}
M_{ZZ_5}(T)=
 \left(
  \begin{array}{ccccc}
  \scs{\!T_1+T_2+T_3\!}&\scs{T_3}      &\scs{0}        & \scs{0}    & \scs{-T_2}  \\
      \scs{T_3}  &\scs{\!T_3+T_4+T_5\!}&\scs{T_5}      & \scs{0}    & \scs{0}      \\
      \scs{0}      &\scs{T_5}  &\scs{\!T_5+T_6+T_7\!}&\scs{T_7}     & \scs{-T_6}    \\
      \scs{0}      &\scs{0}        &\scs{T_7}    &\scs{\!T_7+T_8+T_9\!}&\scs{T_9}    \\
      \scs{-T_2}    &\scs{0}        &\scs{-T_6}    &\scs{T_9}  &\scs{\!T_2+T_6+T_9+T_{10}\!} \\
  \end{array}
 \right)
\end{equation}
\end{example}
Obviously, using a different numbering of edges we get the same
matrix and polynomial up to reindexing of $T$'s. The different
numeration of loops which we take will give us the same graph
polynomial. Indeed, the interchange of the $i$'s and the $j$'s loop
gives us the matrix $\tilde{M}_{\Gamma}(T)$ which can be gotten from
$M_{\Gamma}(T)$ by interchanging of the $i$'s and the $j$'s row and
then the $i$'s and the $j$'s column. The change of the orientation
of one edge $e_i$ gives us the same matrix $M_{\Gamma}$. If we
change the direction of tracing of the $i$'s loop, then all entries
of the $i$ row change signs and then all the entries of the $i$
column, thus the polynomial remains the same. Finally, if we choose
another basis of $H_1(\Gamma,\ZZ)$ then $\Psi_{\Gamma}$ remains the
same by Proposition \ref{Pr2.3}, but the matrix changes. We are
interested not only on the graph polynomial itself, but also in the
matrix $M_{\Gamma}$.  We will always choose loops of small length to
get as much zeros in $M_{\Gamma}$ as possible.
\begin{definition}
    The graph hypersurface $X_{\Gamma}\subset\PP^{N-1}$ is the
    hypersurface cut out by $\Psi_{\Gamma}=0$.
\end{definition}
Throughout the whole paper we deal with such graph hypersurfaces.
Sometimes it is convenient to make a linear change of coordinate in
$\PP^{N-1}$ to simplify the matrix. Clearly, this new matrix
$\widetilde{M}_{\Gamma}$ will define a hypersurface isomorphic to
$X$, which we denote again by $X$. For the graph in Example
\ref{Ex2.4}, we note that $T_0$, $T_4$, $T_8$ and $T_{10}$ appear
only in the diagonal of $M_{ZZ_5}(T)$. Changing the coordinates and
redefining the variables by $A_0,\ldots,A_5$, $B_0$, $B_1$, $B_3$
and $B_4$, we get the matrix
\begin{equation}
    M=M_{ZZ_5}(A,B)=
  \left(
  \begin{array}{ccccc}
     B_0  & A_0 &  0   &  0   & A_5  \\
     A_0  & B_1 & A_1  &  0   &  0   \\
      0   & A_1 & C_2  & A_2  & A_4  \\
      0   &  0  & A_2  & B_3  &  A_3 \\
    A_5   &  0  & A_4  & A_3  & B_4  \\
  \end{array}
\right) \label{h52},
\end{equation}
where $C_2=A_1+A_2-A_4$. If we change the direction of tracing of
the fifth loop for $ZZ_5$, then we come the the matrix of the same
shape as (\ref{h52}) but with $C_2=A_1+A_2+A_4$.
\begin{remark}
Proposition 2.3 shows that for a variable $T_i$ the determinant of
the matrix  $M_{\Gamma}(T)$ is linear as a polynomial in $T_i$. The
statement does not hold for the matrix $\tilde{M}_{\Gamma}$ which we
get form $M_{\Gamma}(T)$ by linear change of variables.
\end{remark}
\bigskip
Now we are going to define \emph{Feynman quadrics} and explain where
the graph polynomials are coming from. We follow \cite{BEK}, section
5. Let $K\subset\RR$ be a real field. (We use $K=\QQ$ for the
applications to Feynman quadrics.) Take some homogeneous quadrics
\begin{equation}
Q_i: q_i(Z_1,\ldots, Z_{2r})=0,\quad\text{for}\;\; 1\leq i\leq r,
\end{equation}
in the space $\PP^{2r-1}(Z_1:\ldots:Z_{2r})$. The union $\cup_{1}^i
Q_i$ of this quadrics has degree $2r$. Thus
$\Gamma(\PP^{2r-1},\omega(\sum_1^r Q_i))=K[\eta]$ where $\eta$ is
given by
\begin{equation}
    \eta\lvert_{Z_{2r-1}\neq 0}=\frac{dz_1\wedge\ldots\wedge dz_{2r-1}}
    {\tilde q_1\dots\tilde q_r}
\end{equation}
on the affine open $Z_{2r}\neq 0$ with coordinates
$z_i=\frac{Z_i}{Z_{2r}}$ and $\tilde{q}_i=\frac{q_i}{Z_{2r}^2}$. We
write
\begin{equation}
    \eta=\frac{\Omega_{2r-1}}{q_1\dots q_r},\quad\text{where}\;\;
    \Omega_{2r-1}=\sum_{i=1}^{2r}(-1)^i Z_idZ_1\wedge\dots
    \widehat{dZ_i}\dots\wedge dZ_{2r}.
    \label{h55}
\end{equation}
The transcendental quantity of interest is the \emph{period}
\begin{equation}
    P(Q):=\int_{\PP^{2r-1}(\RR)} \eta =
    \int_{z_1,\ldots,z_{2r\!-\!1}=-\infty}^{\infty} \frac{dz_1\wedge\dots\wedge
    dz_{2r-1}}{\tilde{q}_1\dots\tilde{q}_r}.
    \label{h56}
\end{equation}
(when the integral is convergent).

Suppose now $r=2n$ above and let $H\cong K^n$ be an $n$-dimensional
vector space; we identify $\PP^{4n-1}=\PP(H^4)$. For a linear
functional $\cl:H\rightarrow K$, $\cl^2$ gives a rank 1 quadratic
form of $H$. A \emph{Feynman quadric} is a rank 4 positive
semi-definite form on $\PP^{4n-1}$ of the form
$q=q_{\cl}=(\cl^2,\cl^2,\cl^2,\cl^2)$. We are interested in the
quadrics $Q_i$ of this form. So, we suppose that the linear forms
$\cl_i$ are given, $1\leq i\leq 2n$ and consider the period $P(Q)$
for $q_i=(q_{\cl_i},q_{\cl_i},q_{\cl_i},q_{\cl_i})$.

For a linear form $\cl:H\rightarrow K$, we define
$\lambda=\ker(\cl)$ and
$\Lambda=\PP(\lambda,\lambda,\lambda,\lambda)\subset \PP^{4n-1}$.
The Feynman quadric $q_{\cl}$ is then a cone over the codimension 4
linear space $\Lambda$. We have $q_{\cl}=Z_1^2+\dots +Z_4^2$ for a
suitable choice of homogeneous coordinates $Z_1,\dots,Z_{2n}$.
\medskip\\
Consider now a graph $\Gamma$ with $N$ edges and the Betti number
$n=h_1(\Gamma)$. By (\ref{h44}), we have the configuration of $N$
hyperplanes in the $n$-dimensional vector space $H=H_1(\Gamma)$. As
above, we consider the Feynman quadrics
$q_i=(\cl^2,\cl^2,\cl^2,\cl^2)$ on $\PP^{4n-1}$, $1\leq i\leq N$.
\begin{definition}
The graph $\Gamma$ is said to be \emph{convergent} (resp.
\emph{logarithmically divergent} ) if $N>2h_1(\Gamma)$ (resp.
$N=2h_1(\Gamma)$). The logarithmically divergent graph $\Gamma$ is
\emph{primitively log divergent} if any connected proper subgraph
$\Gamma'\subset\Gamma$ is convergent.
\end{definition}
When $\Gamma$ is logarithmically divergent, the form
\begin{equation}
    \omega_{\Gamma}:=\frac{d^{4n-1}x}{q_1\ldots q_{2n}}
\end{equation}
has poles only along $\bigcup Q_i$, and we define the period
\begin{equation}
    P(\Gamma):=\int_{\PP^{4n-1}(\RR)} \omega_{\Gamma}
\end{equation}
as in (\ref{h56}).
\begin{proposition}
Let $\Gamma$ be a logarithmically divergent graph with $2n$ edges
and $h_1(\Gamma)$. The period $P(\Gamma)$ converges if and only if
\;$\Gamma$ is primitively log divergent.
\end{proposition}
\begin{proof}  See \cite{BEK}, Proposition 5.2.
\end{proof}
\medskip
Now we explain the Schwinger trick. Let
$Q_i:q_i(Z_1,\ldots,Z_{4n})=0$, $1\leq i\leq 2n$ be quadrics in
$\PP^{4n-1}$, and assume that the period integral (\ref{h56})
converges. Let $M_i$ be the $4n\times 4n$ symmetric matrix
corresponding to $q_i$, we define
\begin{equation}
    \Phi(A_1,\ldots,A_{2n}):=\det (A_1M_1+\ldots+A_{2n}M_{2n}).
\end{equation}
The Schwinger trick relates the period integral $P(Q)$ (see
(\ref{h56})) to an integral on $\PP^{2n-1}$.
\begin{equation}
    \int_{\PP^{4n-1}(\RR)}\frac{\Omega_{4n-1}(Z)}{q_1\ldots q_{2n}}
= C\int_{\sigma^{2n-1}(\RR)}\frac{\Omega_{2n-1}(A)}{\sqrt{\Phi}}.
\end{equation}
Here by $\sigma^{2n-1}(\RR)\subset\PP^{2n-1}(\RR)$ we denote the
locus of all points $s=[s_1,\ldots,s_{2n}]$ such that the projective
coordinates $s_i\geq 0$. $C$ is an elementary constant, and
$\Omega$'s are as in (\ref{h55}). More precisely, we have the
following
\begin{proposition}
Assuming that the integral $P(Q)$ is convergent, we have
\begin{equation}
    P(Q):=\int_{\PP^{4n-1}(\RR)}\frac{\Omega_{4n-1}(Z)}{q_1\ldots
    q_{2n}}=\frac{c}{\pi^{2n}}\int_{\sigma^{2n-1}(\RR)}\frac{\Omega_{2n-1}(A)}{\sqrt{\Phi}},
\end{equation}
where $c\in\overline{\mathstrut\QQ}^{\times}$, $[\QQ(c):\QQ]\leq 2$.
If $\Phi=\Xi$ for some $\Xi\in\QQ[A_1,\ldots,A_{2n}]$, then
$c\in\QQ^{\times}$.
\end{proposition}

\begin{proof} See \cite{BEK}, Proposition 6.2.
\end{proof}
\begin{corollary}
    Let $\Gamma$ be a primitively log divergent graph with $2n$ edges and
    $q_1,\ldots,q_{2n}$ are the Feynman quadrics associated to
    $\Gamma$. The symmetric matrices $M_i$ in this case are block
    diagonal
    \begin{equation}
    M= \left(
  \begin{array}{cccc}
     N_i& 0 & 0 & 0 \\
      0 &N_i& 0 & 0 \\
      0 & 0 &N_i& 0 \\
      0 & 0 & 0 &N_i\\
  \end{array}
        \right) \label{h62},
    \end{equation}
and we can write $\Phi=\Psi_{\Gamma}^4$, where $\Psi_{\Gamma}=\det
(A_1N_1+\ldots+A_{2n}N_{2n})$ is a graph polynomial (\ref{h46}).
Applying Schwinger trick, we get
\begin{equation}
    P(Q):=\int_{\PP^{4n-1}(\RR)}\frac{\Omega_{4n-1}(Z)}{q_1\ldots
    q_{2n}} = \frac{c}{\pi^{2n}}\int_{\sigma^{2n-1}(\RR)}
    \frac{\Omega_{2n-1}(A)}{\Psi^2_{\Gamma}} \label{h63}
\end{equation}
for $c\in\QQ^{\times}$.
\end{corollary}
\medskip
In Section 7 of \cite{BEK}, the following construction was defined
for a primitively log divergent graph $\Gamma$. Consider
$\PP^{2n-1}$ with homogeneous coordinates $A_1,\ldots, A_{2n}$
associated with edges of $\Gamma$. We refer to linear spaces
$L\subset\PP^{2n-1}$ defined by vanishing of the $A_i$ as coordinate
linear spaces. For such an $L$, we write $L(\RR^{\geq 0})$ for the
subset of real points with non-negative coordinates. We know that
\begin{equation}
    X_{\Gamma}(\CC)\cap \sigma^{2n-1}(\RR)=\bigcup_{L\subset
    X_{\Gamma}}L(\RR^{\geq 0}),
\end{equation}
where the union goes over all coordinate linear spaces $L\subset
X_{\Gamma}$ (see \cite{BEK}, Lemma 7.1).
\begin{proposition}
For $\Gamma$ a primitively log divergent graph, define
\begin{equation}
    \eta=\eta_{\Gamma}=  \frac{\Omega_{2n-1}(A)}{\Psi^2_{\Gamma}}.
    \label{h66}
\end{equation}
There exist a tower
\begin{equation}
\begin{aligned}
    &P=P_r\xrightarrow{\;\pi_r\;} P_{r-1}\xrightarrow{\pi_{r-1}}
    \ldots\xrightarrow{\;\pi_{2}\;} P_1
    \xrightarrow{\;\pi_1\;}\PP^{2n-1},\\
    &\pi=\pi_r\circ\dots\circ\pi_1,
\end{aligned}
\end{equation}
where $P_i$ is obtained from $P_{i-1}$ by blowing up the strict
transform of a coordinate linear space $L_i\subset X_{\Gamma}$ and
such that
\begin{description}
  \item[(i)] $\pi^*\eta_{\Gamma}$ has no poles along the exceptional
  divisors associated to the blowups.
  \item[(ii)] Let $B\subset P$ be a total transform of coordinate
  hyperplanes $\Delta^{2n-2}: A_1A_2\ldots A_{2k}=0$. Then $B$ is a
  normal crossings divisor in $P$. No face (on-empty intersection
  of components) of $B$ is contained in the strict transform of $Y$
  of $X_{\Gamma}$ in $P$.
  \item[(iii)] the strict transform of $\sigma^{2n-1}(\RR)$ does not
  meet $Y$.
\end{description}
\end{proposition}
\begin{proof} See $\cite{BEK}$, Proposition 7.3.
\end{proof}
So we define a relative cohomology
\begin{equation}
    H:=H^{2n-1}(P\backslash Y, B\backslash B\cap Y).
    \label{h68}
\end{equation}
The period of this relative cohomology (i.e. the integration along a
homology of $H$ with a de Rham cohomology of $H$) is exactly
\begin{equation}
    \int_{\sigma^{2n-1}(\RR)}
    \frac{\Omega_{2n-1}(A)}{\Psi^2_{\Gamma}},
    \label{h64}
\end{equation}
that one appeared in (\ref{h63}). For more explanation see
$\cite{BEK}$ and $\cite{Bl}$ section 7 and 8. We can consider the
system of realizations of $H$ (see, for example, $\cite{Hub}$).
There is a hope (see \cite{BEK}, 7.25) that for all primitively log
divergent graph, or for an identifiable subset of them, the maximal
weight piece of the Betti realization $H_B$ is Tate,
\begin{equation}
    \gr^W_{max} H_B = \QQ(-p)^{\oplus r}.
\end{equation}
One would like to that there should be a rank 1 sub-Hodge structure
$\iota: \QQ(-p)\hookrightarrow\gr^W_{max} H_B$ such that the image
of $\eta_{\Gamma}\in H_{DR}$ in $\gr^W_{max} H_{DR}$ spans
$\iota(\QQ(-p))_{DR}$.

Unfortunately, we cannot compute this even in very simple cases, but
something can be done here. Note that by the construction the blow
up above, we have natural inclusion $\PP^{2n-1}\backslash
X\hookrightarrow P\backslash Y$. This implies a morphism
\begin{equation}
    H^{2n-1}(P\backslash Y) \xrightarrow{\;j\;} H^{2n-1}(\PP^{2n-1}\backslash X).
    \label{h70}
\end{equation}
Furthermore, the relative cohomology in (\ref{h68}) fits into an
exact sequence
\begin{equation}
    \longrightarrow H^{2n-2}(B\backslash B\cap Y)\longrightarrow H \longrightarrow
    H^{2n-1}(P\backslash Y)\longrightarrow
\end{equation}
The idea (and the only thing we can do) is to compute
$H^{2n-1}(\PP^{2n-1}\backslash X)$. We hope that the map $j$ in
(\ref{h70}) is nonzero, otherwise our computations give no
information about $H$. In the paper \cite{BEK}, Section 11, there
was computed $H^{2n-1}(\PP^{2n-1}\backslash X_n)$ for $X_n$ a graph
hypersurface of $WS_n$, $n\geq 3$ (for Betti or \cl-adic cohomology)
\begin{equation}
H^{2n-1}(\PP^{2n-1}\backslash X_n)\cong \QQ(-2n+3).
\end{equation}
Moreover, motivated by discussion above about the weights of
realizations of $H$, for the de Rham cohomology there was proved
(see Section 12) that the class of
\begin{equation}
    \eta_n:=\frac{\Omega_{2r-1}}{\Psi_n^2}\in\Gamma(\PP^{2n-1},\omega(2X_n))
\end{equation}
lies in the second level of the Hodge filtration (and generates the
whole cohomology because $H^{2n-1}_{DR}(\PP^{2n-1}\backslash X_n))$
is one dimensional).

In the next chapters we compute $H^{2n-1}(\PP^{2n-1}\backslash X)$
(or the maximal graduate piece of weight filtration) for new
examples of primitively divergent graphs. For $ZZ_5$ it also
succeeded to do the computation for
$H^{2n-1}_{DR}(\PP^{2n-1}\backslash X)$.

\newpage
\section{Cohomology}
In this section we explain the cohomological tools we will use. We
start with  \'etale cohomology theory and at some place we proceed
both with \'etale and Betti cases.
\begin{definition}
Let $X$ be a separated scheme over some field $K$ of char. 0. We and
$\sF$ be a torsion  constructible sheaf of abelian groups on X. We
can consider $\sF$ as a contravariant functor $\sF:Et/X\rightarrow
Ab$. The category of sheaves for the \'etale topology on $X$ is
abelian with enough injectives. Now we can define \'etale cohomology
groups $H^r(X_{et},\sF)$ exactly as in the classical case using the
derived functors of $\sF\mapsto \sF(X)$.

Suppose now that there exists an embedding $j:X\hookrightarrow
\overline{X}$ into some complete scheme $\overline{X}$ as an open
subscheme. Then the \emph{cohomology with compact support} of $X$ is
defined to be
\begin{equation*}
        H^r_c(X,\sF)=H^r(\overline{X},j_{!}\sF).
        \label{cc1}
\end{equation*}
\end{definition}
For any closed subscheme $Z\subset X$ in the assumptions above, one
has the following exact sequence
\begin{equation}
    \longrightarrow H^r_c(X-Z,\sF)\longrightarrow H^r_c(X,\sF)
    \longrightarrow H^r_c(Z,\sF)\longrightarrow
    \label{cc2}
\end{equation}
induced by the exact sequence of abelian sheaves
\begin{equation}
    0\longrightarrow j'_!(\sF|_{X-Z})\longrightarrow \sF
    \longrightarrow i_*(\sF|_Z)\longrightarrow 0,
\end{equation}
where  $j':X-Z\hookrightarrow X$ and $i:Z\hookrightarrow X$ denote
the inclusions (see \cite{Milne1}, ch. 3, Remark 1.30).
\begin{definition}
    A variety $X$ is said to have \emph{cohomological dimension} $c$
    if $c$ is the least integer such that
    \begin{equation}
    H^r(X,\sF)=0
    \end{equation}
    for $r>c$ and all torsion sheaves $\sF$ on $X$.
\end{definition}
From now on we suppose $K$ to be algebraically closed because in our
computations the following theorem is used frequently.
\begin{theorem}\label{T1.3.3}
    For a variety $X$ over algebraically closed field $K$,
    \begin{equation}
        \cd(X)\leq 2\dim(X).
        \label{cc5}
    \end{equation}
    If $X$ is affine, then
    \begin{equation}
        \cd(X)\leq \dim(X).
        \label{cc6}
    \end{equation}
\end{theorem}
\begin{proof}
    For the proof of the first statement see \cite{Milne2}, Theorem
    15.1. The proof of the second statement is given in \cite{SGA7},
    XIV, Theorem 3.1, this statement is usually called the \emph{Artin vanishing}.
\end{proof}
We return to the exact sequence (\ref{cc2}); applying it for the
constant sheaves $\sF_n$ determined by $\Lambda=\ZZ/\cl^{n}\ZZ$ for
each $n$, and taking the inverse limit, we get the
\emph{localization} sequence:
\begin{equation}
    \longrightarrow H^r_c(X - Z,\QQ_{\cl})\longrightarrow H^r(X,\QQ_{\cl})
    \longrightarrow H^r(Z,\QQ_{\cl})\longrightarrow
    \label{cc7}
\end{equation}
Because the operation of taking the inverse limit is an exact
functor on modules of finite length, this sequence is exact. Theorem
\ref{T1.3.3} has the following corollary.
\begin{corollary}\label{C1.3.4}
    For an affine smooth $X$ and a locally constant sheaf $\sF$ (in particular,\newline $\sF=\QQ_{\ell}$) on $X$
    \begin{equation}
        H^r_c(X,\sF)=0
    \end{equation}
    for $r<dim(X)$.
\end{corollary}
\begin{proof}
Since $X$ is smooth and $\sF$ is locally constant, we can apply the
Poincar\'e duality (\cite{Milne2}, Theorem 24.1) to $X$. This
implies the first statement. Considering the inverse system of
constant sheaves determined by $\Lambda=\ZZ/\cl^n\ZZ$, for each
$n\geq 1$, and using the same argument as that one for the sequence
(\ref{cc7}), we get $H^r_c(X,\QQ_{\cl})=0$ for $r<dim(X)$.
\end{proof}
Consider the following situation: $X\subset\PP^m$ is defined by the
vanishing of one homogeneous polynomial $f\in K[x_0,...,x_m]$,
$m\geq 2$, we write $X=\cV(f)$ in such situation. Applying
(\ref{cc7}) for the inclusion $X\hookrightarrow \PP^m$, we get an
exact sequence
\begin{equation}
    \longrightarrow H^r_c(\PP^m\backslash X,\QQ_{\cl})\longrightarrow H^r(\PP^m,\QQ_{\cl})
    \longrightarrow H^r(X,\QQ_{\cl})\longrightarrow.
\end{equation}
Note that $\PP^m\backslash X$ is affine (and smooth) of dimension
$m$, thus, by Artin's vanishing or Corollary \ref{C1.3.4},
\begin{equation}
    H^r_c(\PP^m\backslash X,\QQ_{\cl}))=0
\end{equation}
for $0\leq r\leq m-1$. This implies
\begin{equation}
    H^r(X,\QQ_{\cl})\cong H^r(\PP^m,\QQ_{\cl})
    \label{cc9}
\end{equation}
for $0\leq r\leq m-2$ and $H^{m-1}(\PP^m,\QQ_{\cl})\hookrightarrow
H^{m-1}(X,\QQ_{\cl})$. So, the first interesting cohomology of a
hypersurface in $\PP^m$ is in degree $m-1$, we call it sometimes a
\emph{middle dimensional} cohomology $H^{mid}(X)$. Now we formulate
some statements both for \'etale and Betti cohomology and write
$H^r(X)$ to unify the notation.
\begin{definition}
    Define
    \begin{equation}
        H^r_{prim}(X):=coker(H^r(\PP^m)\longrightarrow H^r(X))
    \end{equation}
    for all $r$.
\end{definition}
We have no good reference for the following statement and we will
prove in here.
\begin{theorem}
    Let $X\subset\PP^n$ be a variety over algebraically closed field of characteristic 0.
    Then the morphism
    \begin{equation}
    \phi_r:H^r(\PP^n)\longrightarrow H^r(X)
    \end{equation}
    is injective for $0\leq r\leq 2\dim X$.
\end{theorem}
\begin{proof}
First consider the case of $X$ being a hypersurface, so $\dim X =
n-1$. Since $H(X)=H(X_{{\rm red}})$, we can assume that $X$ is
reduced. For odd $r$ the cohomology $H^r(\PP^n)$ vanishes and there
is nothing to prove. We start with top cohomology $H^{2n-2}(X)$,
$r=2n-2$. The singular locus $\Sigma$ of the reduced hypersurface
$X$ is of dimension at most $n-2$. Define the complement $U:=X -
\Sigma$. Consider the localization sequence
\begin{equation}
\longrightarrow H^{2n-3}(\Sigma) \longrightarrow H^{2n-2}_c(U)
\longrightarrow H^{2n-2}(X) \longrightarrow H^{2n-2}(\Sigma)
\longrightarrow.
\end{equation}
Both the leftmost and the rightmost terms vanish for dimensional
reasons, and we get an isomorphism
\begin{equation}
    H^{2n-2}(X)\cong H^{2n-2}_c(U).
\end{equation}
Let $X$ be a union of irreducible components $X=\bigcup_{i=1}^{j}
X_i$. We resolve singularities and get some $\hat{X}=\coprod_{i=1}^j
\hat{X_i}$ with the inclusion $U\hookrightarrow \hat{X}$. This gives
us a localization sequence
\begin{equation}
\begin{aligned}
\longrightarrow H^{2n-3}(\hat{X}\backslash U ) &\longrightarrow
H^{2n-2}_c(U) \longrightarrow\\
&H^{2n-2}(\hat{X}) \longrightarrow H^{2n-2}(\hat{X}\backslash U)
\longrightarrow
\end{aligned}
\end{equation}
Again, the term to the left and the term to the right are zero for
reason of dimension, and we get an isomorphism
\begin{equation}
    H^{2n-2}_c(U)\cong H^{2n-2}(\hat{X}).
\end{equation}
Each $\hat{X}_i$ is a smooth projective scheme of dimension $n-1$,
and we can compute
\begin{equation}
    H^{2n-2}(X)\cong H^{2n-2}(\hat{X})=\bigoplus_{i=1}^j\QQ(-n+1)
\end{equation}
For a general line $\ell\in\PP^n$ we have $\ell\cap X\subset U$.
Now,
\begin{equation}
    H^{2n-2}(\PP^n)\cong H^{2n-2}_{\ell}(\PP^n)\cong\QQ(-n+1).
\end{equation}
The intersection with $X$ induces a map
\begin{equation}
    H^{2n-2}_{\ell}(\PP^n)\xrightarrow{\;\;\alpha\;\;} H^{2n-2}_{\ell\,\cap
    X}(X).
\end{equation}
By excision, we have an isomorphism
\begin{equation}
\beta: H^{2n-2}_{\ell\,\cap X}(X)\cong H^{2n-2}_{\ell\,\cap
X}(U)\cong H^{2n-2}_{\ell\,\cap X}(\hat{X}).
\end{equation}
Note that $\ell\,\cap X$ is a union of $\deg X$ points (lying on
$U$).  We have a natural morphism
\begin{equation}
    H^{2n-2}_{\ell\,\cap X}(\hat{X})\xrightarrow{\;\;\gamma\;\;} H^{2n-2}(\hat{X})\cong H^{2n-2}(\coprod_{i=1}^j
\hat{X_i}).
\end{equation}
Here $\gamma$ maps the class of a point $p\in\ell\,\cap X$ to the
class of this point in $H^{2n-2}(\hat{X}_i)$ when $p\in X_i$. All
the $\hat{X}_i$ are smooth projective of dimension $n-1$. Thus
$H^{2n-2}(\hat{X}_i)$ is one-dimensional and generated by the class
of a point. Then the composition $\gamma\beta\alpha$ is a nonzero
map. Since $H^{2n-2}(\PP^n)$ is one-dimensional, this proves that
$\phi_{2n-2}$ is injective.

Now we consider maps
\begin{equation}
    H^{2i}(\PP^n)\xrightarrow{\;\;\phi_{2i}\;\;} H^{2i}(X),
\end{equation}
$i\leq n-1$, and take $n-1$ general hyperplanes $H_1,\ldots,
H_{n-1}\subset\PP^n$. The cohomology to the left is generated by the
class $[D_i]$ with $D_i:=H_1\cap\ldots\cap H_i$. For injectivity it
is enough to show that $\phi_{2i}([D_i])\ne 0$. Using the
cup-product on $H^*(X)$, we obtain
\begin{equation}
\phi_{2i}([D_i])=\phi_2([D_1])^i\in H^{2i}(X).
\end{equation}
We see that $D_{n-1}=H_1\cap\ldots\cap H_{n-1}$ is a general line
$\ell\in\PP^n$ and it was proved above that
$\phi_{2n-2}([\ell])=\phi_2([D_1])^{n-1}\neq 0$ in $H^{2n-2}(X)$.
Thus $\phi_{2i}([D_i])\neq 0$ and $\phi_{2i}$ in injective for all
$i\leq n-1$.

Suppose now that $X\subset \PP^n$ is defined by $m$ homogeneous
polynomials, $X=\cV(f_1,\ldots,f_m)$, and is of dimension $d$. We
can play the same game for $X_{red}$ to show that $\phi_i$ are
injective for $i\leq 2d$. Indeed, take general hyperplanes $H_i$,
$1\leq i\leq d$ and define $D_{d}:=H_1\cap\ldots\cap H_{d}$. Then
$D_{d}\cap X\subset X_{smooth}$. Denote by $\hat{X}$ the Hironaka
resolution of singularities of $X_{res}$ (exists since $K=\bar{K}$).
By the same argument as above,
\begin{equation}
H^{2d}(X)\cong H^{2d}(\hat{X})\cong \bigoplus_{i=1}^j\QQ(-d).
\end{equation}
Note that in $H^{2d}(\hat{X})$ only the summands which correspond to
the resolutions of the irreducible components of maximal dimension
($=\dim(X)$) may survive, all other die for reason of dimension.

The intersection $D_{d}\cap X$ is a union of points. We explain the
map $\phi_{2d}$ as above and conclude that $\phi_{2d}([D_d])\neq 0$
in $H^{2d}(X)$. Now it follows that $\phi_{2i}([D_i])\neq 0$ and
$\phi_{2i}$ is injective for all $i\leq d$.
\end{proof}

Let $X$ be a proper scheme and $Y\subset X$ be a closed subscheme.
By the theorem above, the localization sequence for $Y\subset X$
implies that the sequence
\begin{equation}
    \longrightarrow H^i_c(X\backslash Y)\longrightarrow H^i_{prim}(X)
    \longrightarrow H^i_{prim}(Y)\longrightarrow
\end{equation}
is exact in all terms up to $H^i_{prim}(Y)$ for $i=2\dim Y$.

 The Mayer-Vietoris sequence for the closed covering
$X=X_1\cup X_2$ yields the sequence
\begin{equation}
    \longrightarrow H^i_{prim}(X)\longrightarrow H^i_{prim}(X_1)\oplus H^i_{prim}(X_2)
    \longrightarrow H^i_{prim}(X_1\cap X_2)\longrightarrow
\end{equation}
which is exact in terms up to $H^i_{prim}(X_1\cap X_2)$ for $i=2\dim
X_1\cap X_2$.

 For our computations we need some vanishing theorems. First,
Artin's vanishing holds in the analytic category.
\begin{theorem}
    Let $X$ be an affine variety defined over the field of complex
    numbers, and $\sF$ be a constructible sheaf. Then
    $H^m(X_{an},\sF)=0$ for $m>dim(X)$.
\end{theorem}
\begin{proof} The direct analytic proof can be found in \cite{Es}.
\end{proof}
The next two theorems are often referred to in the next chapters.
\begin{theorem}[Vanishing Theorem A] Let Y be a variety
$\cV(f_1,f_2,\ldots,f_k)\subset\PP^N(a_0:a_1:\ldots:a_N)$ for some
homogeneous polynomials $f_1,\ldots,f_k\in K[a_0,\ldots,a_N]$, and
suppose that $f_i$ are independent of the first $t$ variables
$a_0,\ldots,a_{t-1}$ for each $i$, $1\leq i\leq k$. Then
\begin{description}
  \item[1)] $H^r_{prim}(Y)=0 \;$ for $r< N-k+t$.
  \item[2)] $H^r(Y)=H^{r-2t}(Y')(-t)$ for $r\geq 2t$,
  where $Y'\subset\PP^{N-t}$ is defined by the same polynomials.
\end{description}
\end{theorem}
\begin{proof}
Suppose first that $t=0$. We prove that $H^{r}(\PP^N\backslash Y)=0$
for $r\geq N+k$ using induction on $k$. For $k=1$ we have an affine
$\PP^N\backslash Y$ and the statement is exactly Atrin's vanishing.
Assume that $k>1$ and the statement holds for all $Y$ defined by at
most $s<k$ polynomials.
Let $Y:=\cV(f_1,f_2,\ldots,f_k)$ and $U=\PP^N\backslash Y$. Define
the covering $U_1,U_2\subset U$ by $U_1:=\PP^N\backslash\cV(f_1)$
and $U_2:=\PP^N\backslash\cV(f_2,\ldots,f_k)$. Note that the
intersection
\begin{equation}
    U_3:=U_1\cap
    U_2=\PP^N\backslash(\cV(f_1)\cup\cV(f_2,\ldots,f_k))=
    \PP^N\backslash\cV(f_1f_2,\ldots,f_1f_k)
\end{equation}
is again the complement of a complete intersection defined by at
most $k-1$ polynomials. We write a Mayer-Vietoris sequence
\begin{equation}
    \longrightarrow H^{r-1}(U_3)\longrightarrow H^r(U)\longrightarrow
    H^r(U_1)\oplus H^r(U_2)\longrightarrow
\end{equation}
By the assumption both the cohomology to the left and the summands
to the right vanish for $r-1\geq N+k-1$. Thus, the sequence implies
$H^r(U)=0$ for $r\geq N+k$. The induction hypothesis follows.

By duality, one has $H^r_c(\PP^N\backslash Y)=0$ for $r\leq N-k$. We
have an exact sequence
\begin{equation}
    \longrightarrow H^{r-1}_{prim}(\PP^N)\longrightarrow
H^{r-1}_{prim}(Y)\longrightarrow H^r_c(\PP^N\backslash
Y)\longrightarrow H^r_{prim}(\PP^N) \longrightarrow \label{cc1.28}
\end{equation}
Since $H^i_{prim}(\PP^N)=0$ for all $i$, the sequence gives us an
isomorphism\newline $H^{r-1}_{prim}(Y)\cong H^r_c(\PP^N\backslash
Y)$. Thus $H^r_{prim}(Y)=0$ for $r< N-k$.

Suppose now that $t\geq 1$. Define
$\Delta:=\cV(a_t,\ldots,a_N)\cong\PP^{t-1}$. Consider the natural
projection $\pi:\PP^N\backslash\Delta \longrightarrow \PP^{N-t}$.
Note that $\Delta\subset Y$ is a closed subscheme, thus one has an
exact sequence
\begin{equation}
\longrightarrow H^r_c(Y\backslash\Delta)\longrightarrow
H^r_{prim}(Y)\longrightarrow H^r_{prim}(\Delta) \longrightarrow.
\end{equation}
The map $\pi$ gives us an $\AAA^t$-fibration over
$\pi(Y\backslash\Delta)=Y'$, by homotopy invariance
\begin{equation}
    H^r_c(Y\backslash\Delta)\cong H^{r-2t}(Y')(-t).
\end{equation}
Now, $H^r(Y\backslash \Delta)=0$ for $r\leq 2n-1$ and
$H^r(\Delta)=0$ for $r\geq 2n-1$. The sequence above implies
$H^r_{prim}(Y)\cong H^r_{prim}(\Delta)=0$ for $r\leq 2n-2$,
$H^{2n-1}(Y)=0$, and $H^r(Y)\cong H^r(Y\backslash\Delta)\cong
H^{r-2t}(Y')(-t)=0$ for $2t\leq r\leq N-k+t$. We applied here the
case $t=0$ for $Y'$. The statement follows.
\end{proof}
\begin{theorem}[Vanishing Theorem B]
For homogeneous polynomials $f_1,\ldots,f_k,h\in K[a_0,\ldots,a_N]$,
$k\geq 0$, define subscheme $U\subset\PP^{N}$ by equations
$f_1=\ldots=f_k=0$ and inequality $h\neq 0$, i.e.
$$
    U:=\cV(f_1,\ldots,f_k)\backslash \cV(f_1,\ldots,f_k,h).
$$
Suppose that all the polynomials are independent of the first $t$
variables $a_0,\ldots,a_{t-1}$, and let $U'\subset\PP^{N-t}$ be
defined by the same polynomials but in $\PP^{N-t}(a_t:\ldots:a_N)$.
Then the following equalities hold:
\begin{description}
  \item[1)] $H^i_c(U)=0$ for $i < N-k+t$.
  \item[2)] $H^i_c(U)=H^{i-2t}_c(U')(-t)$.
\end{description}
\end{theorem}
\begin{proof}
Suppose first that $t=0$. Let $Y:=\cV(f_1\ldots,f_n)\subset\PP^N$,
then $U\cong$\newline $Y\backslash Y\cap\cV(h)$. We have an exact
sequence
\begin{equation}
    \longrightarrow H^{r-1}_{prim}(Y\cap\cV(h)) \longrightarrow H^r_c(U) \longrightarrow H^r_{prim}(Y)\longrightarrow
\end{equation}
By Theorem A, both the cohomology to the right and the cohomology to
the left vanish for $r-1< N-k-1$. Thus $H^r_c(U)=0$, $r< N-k$.

Let $t\geq 1$, consider the natural projection
$\PP^N\backslash\Delta\longrightarrow\PP^{N-t}$, where
$\Delta:=\cV(a_t,\ldots,a_N)$. It maps $U$ onto $U'$ with fibres
$\AAA^t$. Thus
\begin{equation}
    H^r_c(U)\cong H^{r-2t}(U')(-t).
\end{equation}
for all $r$. The case $t=0$ applied to $U'\subset\PP^{N-t}$ gives us
$H^{r-2t}_c(U')=0$ for $r-2t< N-t-k$, thus $H^r_c(U)=0$ for
$r<N-k+t$.
\end{proof}
In the previous section we introduced the notion of graph
hypersurface. By Corollary \ref{Cor1.2.4}, a graph hypersurface $X$
is always defined over $\ZZ$. In the \'etale case we work with
$H^i(X\otimes_{\ZZ}\overline{\QQ},\QQ_{\ell})$. We write $H^i(X)$
for this kipping in mind that we have $Gal(\overline{\QQ}/\QQ)$
acting on this $\QQ_{\ell}$-vector spaces. This action will
distinguish $\QQ_{\ell}(-i)$ from $\QQ_{\ell}(-j)$ for $i\neq j$.
Finally, for Betti cohomology we write $H^i(X)$ for
$H^i(X\otimes_{\ZZ}\CC, \QQ)$.

\chapter{GZZ}
\section{$ZZ_5$}

\noindent Here we are going to compute the middle dimensional
cohomology of the graph hypersurface for the graph $ZZ_5$ in all
details.

\noindent Recall that $ZZ_5$ is a primitively divergent graph and
the smallest graph in the zigzag series (see, for example,
\cite{BrKr}, sect. 1) which is not isomorphic to a $WS_n$ graph for
some $n$.
\newline
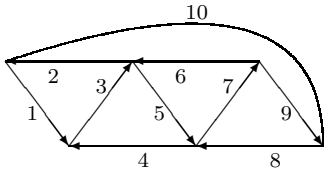
\medskip
\begin{picture}(0,0)(0,40)
\put(10,50){\vector(3,-4){24}} \put(18,28){$\scriptstyle 1$}
\put(58,50){\vector(-1,0){48}} \put(26,42){$\scriptstyle 2$}
\put(34,18){\vector(3,4){24}} \put(44,38){$\scriptstyle 3$}
\put(82,18){\vector(-1,0){48}} \put(60,10){$\scriptstyle 4$}
\put(58,50){\vector(3,-4){24}} \put(66,28){$\scriptstyle 5$}
\put(106,50){\vector(-1,0){48}} \put(74,42){$\scriptstyle 6$}
\put(82,18){\vector(3,4){24}} \put(92,38){$\scriptstyle 7$}
\put(130,18){\vector(-1,0){48}} \put(110,10){$\scriptstyle 8$}
\put(106,50){\vector(3,-4){24}} \put(114,28){$\scriptstyle 9$}
\qbezier(10,50)(130,90)(130,18) \put(78,66){$\scriptstyle 10$}
\end{picture}
\mathstrut\qquad\qquad\qquad\qquad\qquad\qquad\qquad
\begin{tabular}{c|c c c c c c c c c c|}
      & 1 & 2 & 3 & 4 & 5 & 6 & 7 & 8 & 9 &\!10\\ \hline
    1 & 1 & 1 & 1 & 0 & 0 & 0 & 0 & 0 & 0 & 0\\
    2 & 0 & 0 & 1 & 1 & 1 & 0 & 0 & 0 & 0 & 0\\
    3 & 0 & 0 & 0 & 0 & 1 & 1 & 1 & 0 & 0 & 0\\
    4 & 0 & 0 & 0 & 0 & 0 & 0 & 1 & 1 & 1 & 0\\
    5 & 0 &\!-1 & 0 & 0 & 0 &\!-1 & 0 & 0 & 1 & 1\\
\end{tabular}
\bigskip\\
In the way it was done in Example \ref{Ex2.4}, we build a matrix
$M_{ZZ_5}(T)$, and then, changing the variables, we come to the
matrix
\begin{equation}
M:= M_{ZZ_5}(A,B)=
 \left(
  \begin{array}{ccccc}
     B_0  & A_0 &  0   &  0   & A_5  \\
     A_0  & B_1 & A_1  &  0   &  0   \\
      0   & A_1 & C_2  & A_2  & A_4  \\
      0   &  0  & A_2  & B_3  &  A_3 \\
    A_5   &  0  & A_4  & A_3  & B_4  \\
  \end{array}
\right) \label{c1}
\end{equation}
with variables $A_0,\ldots,A_5,B_0,B_1,B_3,B_4$ and
\begin{equation}
    C_2=A_1+A_2-A_4.
\end{equation}
This dependent $C_2$ is a main problem for translating the technics
of the case of $WS_n$ (see \cite{BEK}, sect. 11)  to the $ZZ_5$
case. Define the hypersurface associated to the graph $ZZ_5$:
\begin{equation}
    X:=\cV(\det M)\subset\PP^9.
\end{equation}
We will work in $\PP^9(A_0:\ldots:A_5:B_0:B_1:B_3:B_4)$. Similarly
to notation in Chapter 1, the projective scheme corresponding to the
finite set of homogeneous polynomials $f_1,\ldots,f_k$ will be
denoted by $\cV(f_1,\ldots,f_k)$. Sometimes we will write
$\cV(f_1,\ldots,f_k)^{(N)}$ to indicate that we consider the variety
in $\PP^N$. We use the term \emph{variety} for reduced (but not
necessary irreducible) schemes. We define
$H^*_{prim}(X)=\coker(H^*(\PP^N)\longrightarrow H^*(X))$ for a
subvariety $X\subset\PP^N$.
\begin{theorem}
    Let $X\subset\PP^9$ be the hypersurface associated to $ZZ_5$, then
    \begin{equation}
        H^8_{prim}(ZZ_5)\cong \QQ(-2).
    \end{equation}
\end{theorem}
\begin{proof}
The variety $X$ is defined by the equation $I_5=0$. We write
\begin{equation}
    I_5=I_4B_4-G_4,
    \label{c4}
\end{equation}
where
\begin{equation}
\begin{aligned}
    G_4=G_4'+ &A_3^2I_3-2A_3A_4A_2I_2-2A_3A_5S_3=\\
    &G_4'+ A_3^2I_3-2A_3A_4A_2I_2-2A_3A_5A_2A_1A_0
\end{aligned}\label{c5}
\end{equation}
and
\begin{equation}
\begin{aligned}
    G_4'=A_4^2B_3I_2& + 2A_4A_5B_3S_2
          +A_5^2I^1_3=\\
          &A_4^2B_3I_2 + 2A_4A_5B_3A_1A_0+A_5^2I^1_3,
\end{aligned}\label{c6}
\end{equation}
see (\ref{b25}) and (\ref{b26}). By (\ref{c4}), $I_5$ is linear in
the variable $B_4$. Consider
\begin{equation}
    \hat{Y}:=\cV(I_5,I_4)\subset\PP^9.
\end{equation}
The variety $\hat{Y}$ is closed in $X$ and one has the localization
sequence
\begin{equation}
    \rightarrow H^8_c(X\backslash \hat{Y})\rightarrow
H^8(X)\rightarrow\\
H^8(\hat{Y})\rightarrow H^9_c(X\backslash \hat{Y})\rightarrow.
 \label{c10}
\end{equation}
Let $P_1=(0,...,0,1)\subset\PP^9$ be the point where all the
variables but $B_4$ are zero. We project from $P_1$ and get an
isomorphism
\begin{equation}
    X\backslash \hat{Y} \cong \PP^8\backslash\cV(I_4).
 \label{c11}
\end{equation}
Note that $I_4$ is independent of $A_5$ and $A_3$, thus we can apply
\emph{Theorem B} to $\PP^8\backslash\cV(I_4)$ ($N=8$, $k=0$, $t=2$)
and get $H^i(\PP^8\backslash\cV(I_4))=0$ for $i<10$. The isomorphism
(\ref{c11}) implies the vanishing $H^i(X\backslash\hat{Y})=0$ for
$i<10$. Substituting this into the sequence (\ref{c10}), we get an
isomorphism
\begin{equation}
    H^8(X)\cong H^8(\hat{Y}).
  \label{c12}
\end{equation}
By (\ref{c4}), we can rewrite
\begin{equation}
    \hat{Y}=\cV(I_4,G_4)^{(9)}.
  \label{c13}
\end{equation}
The superscript means that this is a subscheme in  $\PP^9$. We will
use the same notation without superscript for the variety after
forgetting $B_4$. Both the polynomials do not depend on $B_4$.
Applying \emph{theorem A} ($N=9$, $k=2$, $t=1$), we get
\begin{equation}
    H^8(\hat{Y})\cong H^6(Y)(-1),
\end{equation}
where $Y:=\cV(I_4,G_4)\subset\PP^{8}$(no $B_4$). Together with
(\ref{c12}) and (\ref{c13}), this implies
\begin{equation}
    H^8(X)\cong H^6(Y)(-1).
  \label{c15}
\end{equation}
Define $\hat{V},U\subset\PP^8$(no $B_4$) by
\begin{equation}
    \widehat{V}=\cV(I_4,I_3,G_4)\quad\text{and}\;\;
    U:= \cV(I_4,G_4)\backslash \widehat{V}.
\end{equation}
We can write an exact sequence
\begin{equation}
     \longrightarrow H^6_c(U)\longrightarrow H^6(Y)
   \longrightarrow H^6(\widehat{V})\longrightarrow H^7_c(U)\longrightarrow.
  \label{c18}
\end{equation}
\begin{lemma} One has $H^i_c(U)=0$ for $i<8$.
\end{lemma}
\begin{proof}
$U$ is defined by the following system
\begin{equation}
    \left\{
        \begin{array}{ll}
         I_4=0=G_4 \\
         I_3\neq 0.
        \end{array}
    \right.
    \label{c19}
\end{equation}
We have studied such schemes in Section 1.1 (see \ref{b44}). By
Theorem \ref{T1.7}, we have an isomorphism
\begin{equation}
    U\cong\ U_2:=\cV(I_4)\backslash\cV(I_4,I_3)\subset\PP^7(\text{no}\;B_4, A_3).
\end{equation}
Using
\begin{equation}
    I_4=B_3I_3-A_2^2I_2
\end{equation}
and projecting from the point $P_3\in\PP^7$ where all variables but
$B_3$ are zero, we get
\begin{equation}
    U_2\cong\PP^6\backslash I_3.
\end{equation}
One has
\begin{equation}
    I_3=\begin{vmatrix} B_0 & A_0 & 0 \\
                        A_0 & B_1 & A_1\\
                        0  & A_1 & C_2\\
                        \end{vmatrix},\;\; C_2=A_1+A_2-A_4
    \label{c26}
\end{equation}
After a linear change of coordinates, we can assume that $I_3$ does
not depend on $A_4$ and $A_5$. The \emph{Theorem B} ($N=6$, $k=0$,
$T=2$) implies
\begin{equation}
    H_c^i(U)\cong
    H_c^i(\PP^6\backslash\cV(I_3))=0\quad\text{for}\;\;
    i<8.
\end{equation}
\end{proof}
From the sequence (\ref{c18}) we now obtain an isomorphism
\begin{equation}
    H^6(Y)\cong H^6(\widehat{V}).
\end{equation}
Combining this with (\ref{c15}), one has
\begin{equation}
    H^8(X)\cong H^6(\widehat{V})(-1).
    \label{c31}    
\end{equation}
Remember that $\widehat{V}$ lives in $\PP^8$ (no $B_4$) and is
defined by
\begin{equation}
    \widehat{V}:=\cV(I_4,I_3,G_4).
\end{equation}
Theorem \ref{T1.5} implies that $G_4$ is independent of $A_3$ on
$\widehat{V}$. Define $G_4'=G_4\mid_{A_3=0}$ as in (\ref{c6}). One
has
\begin{equation}
    \widehat{V}=\cV(I_4,I_3,G_4')\subset\PP^8(\text{no}\; B_4).
\end{equation}
Applying \emph{Theorem A} ($N=8$,  $k=3$, $t=1$) for the last
variety, the defining equations of which are independent of $A_3$,
we obtain
\begin{equation}
    H^6(\widehat{V})\cong H^4(V)(-1),
    \label{c36}
\end{equation}
with
\begin{equation}
    V=\cV(I_4,I_3,G_4')\subset\PP^7 (\text{no}\; B_4, A_4).
    \label{c35}
\end{equation}
We combine this with (\ref{c31}) and get
\begin{equation}
    H^8(X)\cong H^4(V)(-2).
  \label{c37}  
\end{equation}
The next step is to get rid of $B_3$. Define
\begin{equation}
  G_3:=  A_4^2\begin{vmatrix} B_0 & A_0\\
                                 A_0 & B_1\\
                 \end{vmatrix}
        +2A_4A_5\begin{vmatrix} A_0 & 0\\
                                 B_1 & A_1\\
                 \end{vmatrix}
          +A_5^2\begin{vmatrix} B_1 & A_1\\
                                 A_1 & C_2\\
                \end{vmatrix}.
    \label{31}
\end{equation}
We have the following equality
\begin{multline}
    G_4'=A_4^2B_3\begin{vmatrix} B_0 & A_0\\
                                 A_0 & B_1\\
                 \end{vmatrix}
       + 2A_4A_5B_3\begin{vmatrix} A_0 & 0\\
                                 B_1 & A_1\\
                 \end{vmatrix}+\\
          A_5^2\left(\begin{vmatrix} B_1 & A_1\\
                                A_1 & C_2\\
                 \end{vmatrix} B_3-A_2^2B_1 \right)
    = B_3G_3-A_2^2A_5^2B_1.
\end{multline}
So, one has
\begin{multline}
    V=\cV(I_4,I_3,G_4')=\cV(B_3I_3-A_2^2I_2,I_3,B_3G_3-A_2^2A_5^2B_1)=\\
    \cV(I_3,A_2I_2,B_3G_3-A_2^2A_5^2B_1).
\end{multline}
Set
\begin{equation}
    \widehat{W}:=V\cap
    \cV(G_3)=\cV(I_3,A_2I_2,G_3,A_2A_5B_1)\subset\PP^7
    \label{c41}
\end{equation}
with all the polynomials to the right independent of $B_3$. We have
an exact sequence
\begin{equation}
    \rightarrow H^3(\widehat{W})\rightarrow H^4_c(V\backslash \widehat{W})
    \rightarrow H^4_{prim}(V)\rightarrow H^4_{prim}(\widehat{W})\rightarrow
  \label{c42}        
\end{equation}
We apply \emph{theorem A} ($N=7$, $k=4$, $t=1$) to
$\widehat{W}\subset\PP^7$(no $B_4, A_3$) and get
\begin{equation}
H^3(\widehat{W})=0\quad \text{and}\;\;\;
H^4_{prim}(\widehat{W})\cong H^2_{prim}(W)(-1)
\end{equation}
with
\begin{equation}
    W=\cV(I_3,A_2I_2,G_3,A_2A_5B_1)\subset\PP^6(\text{no}\;B_4,A_4,B_3).
\end{equation}
Substituting this in (\ref{c42}), we get
\begin{equation}
    0 \rightarrow H^4_c(V\backslash \widehat{W})
    \rightarrow H^4_{prim}(V)\rightarrow
    H^2_{prim}(W)(-1)\rightarrow.
  \label{c45}   
\end{equation}
Now we show that $H^2_{prim}(W)$ also vanishes.\\
Consider the subvariety ${W\cap\cV(A_2)\subset W}$ and an exact
sequence
\begin{multline}
   \;\;\;\; \rightarrow H^1(W\cap\cV(A_2))\rightarrow H^2_c(W\backslash W\cap\cV(A_2))
    \rightarrow\\
    H^2_{prim}(W)\rightarrow H^2_{prim}(W\cap\cV(A_2))\rightarrow. \;\;\;\;\;
  \label{c46}  
\end{multline}
We write
\begin{equation}
    W\cap\cV(A_2)=\cV(I_3,A_2I_2,G_3,A_2A_5B_1,A_2)=\cV(I_3,G_3,A_2)\subset\PP^6,
\end{equation}
\emph{Theorem A}($N=6$, $k=3$, $t=0$) implies
\begin{equation}
    H^i_{prim}(W\cap\cV(A_2))=0,\quad i\leq 2.
\end{equation}
The sequence (\ref{c46}) gives us
\begin{equation}
    H^2_{prim}(W)\cong H^2_c(W\backslash W\cap\cV(A_2)).
    \label{c49}    
\end{equation}
The scheme $W\backslash W\cap\cV(A_2)$ is defined by
\begin{multline}\;\;
    \left\{
        \begin{aligned}
        I_3=A_2I_2=0\\
        G_3=A_2A_5B_1=0\\
        A_2\neq 0\\
        \end{aligned}
    \right.
   \Leftrightarrow \left\{
        \begin{aligned}
        I_3=I_2=0\\
        G_3=A_5B_1=0\\
        A_2\neq 0.\!\!\\
        \end{aligned}
    \right. \;\;\;\;
\end{multline}
Now define $S,T\subset \PP^{6}$(no $B_4$, $A_3$, $B_3$) by
\begin{equation}
    S=\cV(I_3,I_2,G_3,A_5B_1)
    \label{c51}      
\end{equation}
and $T=S\cap\cV(A_2)$. For these varieties we have an exact sequence
\begin{equation}
   \rightarrow H^1(S)\rightarrow H^1(T)\rightarrow H^2_c(W\backslash
    W\cap\cV(A_2))\rightarrow H^2_{prim}(S)\rightarrow
  \label{c52}       
\end{equation}
Note that the polynomial
\begin{equation}
    G_3= A_4^2\begin{vmatrix} B_0 & A_0\\
                                 A_0 & B_1\\
                 \end{vmatrix}
        +2A_4A_5\begin{vmatrix} A_0 & 0\\
                                 B_1 & A_1\\
                 \end{vmatrix}
          +A_5^2\begin{vmatrix} B_1 & A_1\\
                                 A_1 & C_2\\
                 \end{vmatrix}
\end{equation}
is of the same shape as $G_n$ studied in sect. 1.1, and by the
Theorem \ref{T1.5}, $G_3$ loses first two summands on $S$. Thus,
\begin{multline}
    S=\cV(I_3,I_2,G_3,A_5B_1)=\cV(I_3,I_2,A_5^2(B_1C_2-A_1^2),A_5B_1)=\\
    \cV(C_2I_2-A_1^2B_0,I_2,A_5A_1,A_5B_1)=\cV(A_1B_0,I_2,A_5A_1,A_5B_1).
    \label{c55}
\end{multline}
We see that $S$ is defined by the equations all independent of $A_2$
and $A_4$. \emph{Theorem A} ($N=6$, $k=4$, $t=2$) implies
\begin{equation}
    H^i_{prim}(S)=0,\quad\text{for}\;\; i<4.
    \label{c56}
\end{equation}
Let us look closely at $T$. By (\ref{c55}), we have
\begin{equation}
    T=S\cap\cV(A_2)=\cV(A_1B_0,I_2,A_5A_1,A_5B_1,A_2)\subset\PP^6(\text{no}\;B_4,A_4,B_3).
\end{equation}
The defining polynomials do not depend on $A_4$, and  \emph{Theorem
A} ($N=6$, $k=5$, $t=1$) implies
\begin{equation}
    H^i_{prim}(T)=0,\quad\text{for}\;\; i<2.
    \label{c62}
\end{equation}
By (\ref{c56}) and (\ref{c62}), the sequence (\ref{c52}) implies the
vanishing
\begin{equation}
    H^2_c(W\backslash W\cap\cV(A_2))=0.
  \label{c64}       
\end{equation}
Hence, (\ref{c37}), (\ref{c45}) and (\ref{c49}) yield an isomorphism
\begin{equation}
    H^8_{prim}(X)\cong H^4_{prim}(V)(-2) \cong H^4_c(V\backslash
    \widehat{W})(-2).
    \label{c66}
\end{equation}
\medskip
The subscheme $V\backslash \widehat{W}\subset\PP^7$(no $B_4$, $A_3$)
is defined by
\begin{multline}\;\;
    \left\{
        \begin{aligned}
        I_3=A_2I_2=0\\
        B_3G_3-A_2^2A_5^2B_1=0\\
        G_3\neq 0.\!\!\\
        \end{aligned}
    \right.\;\;\;\;
  \label{c68}   
\end{multline}
We solve the middle equation on $B_3$. Projecting from the point
where all the coordinates but $B_3$ are zero, one gets an
isomorphism
\begin{equation}
    V\backslash \widehat{W}\cong Y\backslash Y\cap\cV(G_3)
  \label{c69}    
\end{equation}
for $Y=\cV(I_3,A_2I_2)\subset\PP^6$(no $B_4$,$A_3$,$B_3$). Consider
the exact sequence
\begin{equation}
\rightarrow H^3(Y)\rightarrow H^3(Y\cap\cV(G_3))\rightarrow
H^4_c(Y\backslash Y\cap\cV(G_3))\rightarrow
H^4_{prim}(Y)\rightarrow. \label{c70}
\end{equation}
The equations of $Y$ do not depend on $A_5$. Applying \emph{Theorem
A} ($N=6$, $k=2$, $t=1$), we obtain
\begin{equation}
    H^i_{prim}(Y)=0\quad \text{for}\;\; i<5.
\end{equation}
Then the sequence (\ref{c70}) implies
\begin{equation}
    H^4_c(Y\backslash Y\cap\cV(G_3))\cong H^3(Y\cap\cV(G_3)).
\end{equation}
Comparing this with (\ref{c66}) and (\ref{c69}), one gets
\begin{equation}
    H^8_{prim}(X)\cong H^4_c(V\backslash\widehat{W})(-2)\cong  H^3(Y\cap\cV(G_3))(-2).
  \label{c74}   
\end{equation}
Consider the subvariety
\begin{equation}
Y_1:=\cV(I_3,I_2,G_3)\subset\cV(I_3,A_2I_2,G_3)=Y\cap\cV(G_3)\subset\PP^6.
\end{equation}
One has an exact sequence
\begin{equation}
    \rightarrow H^2_{prim}(Y_1)\rightarrow H^3_c(Y\cap\cV(G_3) \backslash Y_1)
    \rightarrow H^3(Y\cap\cV(G_3))\rightarrow H^3(Y_1)\rightarrow.
  \label{c75}    
\end{equation}
Theorem \ref{T1.5} applied to $G_3$ on $Y_1$ allows us to rewrite
$Y_1\subset\PP^6$(no $B_4$, $A_3$, $B_3$) as
\begin{equation}
    Y_1=\cV(I_3,I_2,G_3)=\cV(I_2,A_1B_0,A_5(B_1C_2-A_1^2)).
    \label{54}
\end{equation}
After the change of the variables $C_2:=A_2$, the defining equations
of $Y_1$ become independent of $A_4$. Applying \emph{Theorem
A}($N=6$, $k=3$, $t=1$) to $Y_1$, we get
\begin{equation}
    H^i(Y_1)=0\quad\text{for}\;\;i<4.
\end{equation}
The sequence (\ref{c75}) gives us
\begin{equation}
     H^3(Y\cap\cV(G_3))\cong H^3_c(Y\cap\cV(G_3) \backslash Y_1).
  \label{c79} 
\end{equation}
The scheme to the right $Y\cap\cV(G_3) \backslash
Y_1\subset\PP^6$(no $B_4$,$A_3$,$B_3$) is defined by
\begin{multline}\;
    \left\{
        \begin{aligned}
        I_3=A_2I_2=0\\
        G_3=0\\
        I_2\neq 0\\
        \end{aligned}
    \right.
   \Leftrightarrow \left\{
        \begin{aligned}
        C_2I_2-A_1^2B_0=A_2=0\\
        A_4^2I_2 + 2 A_4A_5A_0A_1 + A_5^2(B_1C_2-A_1^2)=0\\
        I_2\neq 0.\!\!\\
        \end{aligned}
    \right. \;\;
  \label{c80}    
\end{multline}
By Corollary \ref{C1.4},
\begin{equation}
    G_3I_2=(A_4I_2+A_0A_1A_5)^2=:Li_3^2
\end{equation}
on $Y\cap\cV(G_3) \backslash Y_1$, thus $G_3=0$ implies
\begin{equation}
    A_4=-\frac{A_0A_1A_5}{I_2}.
  \label{c82}
\end{equation}
Furthermore, solving the first equation of (\ref{c80}) on $C_2$, we
get
\begin{equation}
    C_2=A_1-A_4=\frac{A_1^2B_0}{I_2} \;\;\;\Leftrightarrow\;\;\;
        A_4=\frac{A_1I_2-A_1^2B_0}{I_2}.
  \label{c83}
\end{equation}
By (\ref{c82}) and (\ref{c83}), it follows that
\begin{equation}
    A_1I_2-A_1^2B_0=-A_0A_1A_5.
  \label{62}
\end{equation}
on $Y\cap\cV(G_3) \backslash Y_1$.

We project from the point where the all the coordinates but $A_4$
are zero and forgetting $A_2$, which is zero on $Y\cap\cV(G_3)
\backslash Y_1$, we get an isomorphism
\begin{equation}
Y\cap\cV(G_3) \backslash Y_1\cong R\backslash Z,
\end{equation}
where $R,Z\subset\PP^4$ (no $B_4$, $A_3$, $B_3$, $A_4$, $A_2$) are
defined by
\begin{equation}
    R:=\cV(A_1I_2+A_5A_0A_1-A_1^2B_0)
  \label{64}
\end{equation}
and $Z:=R\cap\cV(I_2)$. By (\ref{c74}) and (\ref{c79}), one gets
\begin{equation}
\begin{aligned}
    H^8_{prim}(X)\cong &H^3(Y\cap\cV(G_3))(-2)\cong\\
    &H^3_c(Y\cap\cV(G_3) \backslash Y_1)(-2)\cong H^3_c(R\backslash
    Z)(-2).
  \label{c89}
\end{aligned}
\end{equation}
The last step of the proof is the following.
\begin{lemma} We have $H^3_c(R\backslash Z)\cong Q(0)$.
\end{lemma}
\begin{proof} The variety $R$ is defined by the equation
\begin{equation}
    A_1I_2+A_5A_0A_1-A_1^2B_0= A_1(I_2+A_0A_5-A_1B_0)=0.
\end{equation}
Consider the Mayer-Vietoris sequence for $R\backslash Z$:
\begin{equation}
\begin{aligned}
\rightarrow H^2_c(\cV(A_1)&\backslash\cV(A_1,I_2))\oplus
H^2_c(R_1\backslash Z_1)\rightarrow H^2_c(\cV(A_1)\cap R_1\backslash
Z_1)\rightarrow\\
&H^3_c(R\backslash Z)\rightarrow
H^3_c(\cV(A_1)\backslash\cV(A_1,I_2))\oplus H^3_c(R_1\backslash
Z_1)\rightarrow
\end{aligned}  \label{c91}    
\end{equation}
with $Z_1\subset R_1\subset R\subset\PP^4(A_0:A_1:A_5:B_0:B_1)$
defined by
\begin{equation}
R_1= \cV(I_2+A_0A_5-A_1B_0)
\end{equation}
and $Z_1=R_1\cap\cV(I_2)$. \emph{Theorem B} ($N=4$, $k=1$, $t=0$)
implies
\begin{equation}
H^2_c(R_1\backslash Z_1)=0. \label{c91.5}
\end{equation}
We prove that $H^3_c(R_1\backslash Z_1)$ also vanishes. One has an
exact sequence
\begin{equation}
    \rightarrow H^2_{prim}(R_1) \rightarrow H^2_{prim}(Z_1)\rightarrow H^3_c(R_1\backslash
    Z_1)\rightarrow H^3(R_1)\rightarrow.
  \label{c92}
\end{equation}
The leftmost term vanishes because $R_1\subset\PP^4$ is a
hypersurface. To compute $H^3(R_1)$, we write the following exact
sequence
\begin{equation}
\begin{aligned}
   \longrightarrow H^2_{prim}(R_1\cap\cV(A_0))\longrightarrow &H^3_c(R_1\backslash
   R_1\cap\cV(A_0))\longrightarrow\\
    &H^3(R_1)\longrightarrow H^3(R_1\cap\cV(A_0))\longrightarrow.
\end{aligned}
  \label{c94} 
\end{equation}
Since
\begin{equation}
\begin{aligned}
R_1\cap\cV(A_0)=\cV(B_0B_1-A_0^2+A_0A_5&-A_1B_0,A_0)=\\
&\cV(A_0,B_0(B_1-A_1))^{(4)},
\end{aligned}
\end{equation}
the defining polynomials are independent of $A_5$. Applying
\emph{Theorem A} ($N=4$, $k=2$, $t=1$), we get
$H^2_{prim}(R_1\cap\cV(A_0))=0$ and
\begin{equation}
    H^3(R_1\cap\cV(A_0))\cong
H^1(\cV(A_0,B_0(B_1-A_1)))(-1)
\end{equation}
with the variety to the right living in $\PP^3(A_0:A_1:B_0:B_1)$.
But this variety is just the union of two lines intersected at one
point and has trivial first cohomology group. Thus
\begin{equation}
    H^3(R_1\cap\cV(A_0))\cong H^3(\cV(A_0,B_0(B_1-A_1))^{(4)})=0,
\end{equation}
and the sequence (\ref{c94}) implies an isomorphism
\begin{equation}
    H^3(R_1)\cong H^3_c(R_1\backslash R_1\cap\cV(A_0)).
    \label{c97}    
\end{equation}
The scheme $R_1\backslash
R_1\cap\cV(A_0)\subset\PP^4(A_0:A_1:A_5:B_0:B_1)$ is defined by
\begin{multline}\;\;
    \left\{
        \begin{aligned}
        B_0B_1-A_0^2+A_0A_5-A_1B_0=0\\
        A_0\neq 0\\
        \end{aligned}
    \right.\;\;\;\;
  \label{106}
\end{multline}
Projecting from the point where all the variables but $A_5$ are zero
( and solving the first equation of the system on $A_5$), we get an
isomorphism
\begin{equation}
    R_1\backslash R_1\cap\cV(A_0)\cong \PP^3\backslash\cV(A_0)\cong\AAA^3.
\end{equation}
Hence,
\begin{equation}
    H^3_c(R_1\backslash R_1\cap\cV(A_0))=0.
\end{equation}
Together with (\ref{c97}), this simplifies (\ref{c92}) to
\begin{equation}
    H^3_c(R_1\backslash Z_1)\cong H^2_{prim}(Z_1).
    \label{c102}   
\end{equation}
Now,
\begin{equation}
    Z_1=R_1\cap\cV(I_2)=\cV(B_0B_1-A_0^2,A_0A_5-A_1B_0)\subset\PP^4.
\end{equation}
Consider
\begin{equation}
    Z_1\cap\cV(B_0)=\cV(B_0, A_0^2,A_0A_5)=\cV(A_0,B_0)\cong\PP^2\subset\PP^4.
\end{equation}
One has an exact sequence
\begin{equation}
    \rightarrow H^1(\PP^2)\rightarrow H^2_c(Z_1\backslash Z_1\cap\cV(B_0))\rightarrow
H^2_{prim}(Z_1)\rightarrow H^2_{prim}(\PP^2)\rightarrow.
\end{equation}
Thus, we have an isomorphism
\begin{equation}
    H^2_{prim}(Z_1)\cong H^2_c(Z_1\backslash Z_1\cap\cV(B_0)).
    \label{c107}
\end{equation}
The scheme $Z_1\backslash Z_1\cap\cV(B_0)$ is defined by the
following system:
\begin{multline}\;
    \left\{
        \begin{aligned}
        B_0B_1-A_0^2=0\\
        A_0A_5-A_1B_0=0\\
        B_0\neq 0\\
        \end{aligned}
    \right.
   \Leftrightarrow \left\{
        \begin{aligned}
        B_1&=\frac{A_0^2}{B_0}\\
        A_1&=\frac{A_0A_5}{B_0}\\
        B_0&\neq 0.\!\!\\
        \end{aligned}
    \right. \;\;
  \label{c108}    
\end{multline}
Set
\begin{equation}
\ell=(0:A_1:0:0:B_1)\subset\PP^4(A_0:A_1:A_5:B_0:B_1).
\end{equation}
The projection
\begin{equation}
    \pi : \PP^4\backslash \ell \longrightarrow \PP^2(A_0:A_5:B_0)
\end{equation}
from the line $\ell$ gives an isomorphism
\begin{equation}
    Z_1\backslash Z_1\cap\cV(B_0)\cong \PP^2\backslash\cV(B_0)\cong\AAA^2.
\end{equation}
Thus, together with (\ref{c107}) and (\ref{c102}), we get
\begin{equation}
    H^3_c(R_1\backslash Z_1)\cong H^2_{prim}(Z_1)\cong H^2_c(Z_1\backslash Z_1\cap\cV(B_0))=0.
    \label{c114}
\end{equation}
Return now to the sequence (\ref{c91}). 
The defining polynomials of
\begin{equation}
\cV(A_1)\backslash\cV(A_1,I_2)\subset\PP^4(A_0,A_1,A_5,B_0,B_1)
\end{equation}
are independent of $A_5$,  \emph{Theorem B} ($N=4$, $k=1$, $t=1$)
implies
\begin{equation}
    H^i_c(\cV(A_1)\backslash\cV(A_1,I_2))=0\quad\text{for}\;\;i<4.
\end{equation}
By (\ref{c114}) and (\ref{c91.5}), the sequence (\ref{c91}) gives us
\begin{equation}
    H^3_c(R\backslash Z)\cong H^2_c(\cV(A_1)\cap R_1\backslash Z_1).
  \label{c117}   
\end{equation}
The variety $\cV(A_1)\cap R_1\backslash Z_1$ is defined by
\begin{multline}\;
    \left\{
        \begin{aligned}
        I_2+A_0A_5-A_1B_0=0\\
        A_1=0\\
        I_2\neq 0\\
        \end{aligned}
    \right.
   \Leftrightarrow \left\{
        \begin{aligned}
        I_2+A_0A_5=0\\
        A_1=0\\
        I_2\neq 0\\
        \end{aligned}
    \right. \;\;
  \label{96}
\end{multline}
Define $S,T\subset\PP^3(A_0:A_5:B_0:B_1)$ by
\begin{equation}
\begin{array}{ll}
S:=B_0B_1-A_0^2+A_0A_5,\\
T:=S\cap\cV(B_0B_1-A_0^2).
\end{array}
\end{equation}
We get an isomorphism (forgetting $A_1$)
\begin{equation}
    \cV(A_1)\cap R_1\backslash Z_1 \cong S\backslash T.
\end{equation}
We have to compute
\begin{equation}
    H^2_c(\cV(A_1)\cap R_1\backslash Z_1)\cong H^2_c(S\backslash T).
  \label{c120}   
\end{equation}
The exact sequence
\begin{equation}
 \longrightarrow H^1(T)\longrightarrow H^2_c(S\backslash T)\longrightarrow
    H^2_{prim}(S)\longrightarrow
  \label{c121}   
\end{equation}
and stratification further gives us only
\begin{equation}
 0\longrightarrow \QQ(0)\longrightarrow H^2_c(S\backslash T)\longrightarrow
    \QQ(-1)\longrightarrow,
  \label{c122}   
\end{equation}
so we must compute $H^2_c(S\backslash T)$ directly.

The variety $S\subset\PP^3(A_0:A_5:B_0:B_1)$ is a quadric which is
smooth. Up to a change of variables $S$ is the image of Segre
imbedding. More precisely, $S=Im(\gamma)$ for
\begin{equation}
\begin{aligned}
\gamma:\PP^1\times\PP^1\hookrightarrow\PP^3: (a:b),(c:d)\mapsto
(ac:ac-bd:ad:bc).
\end{aligned}
\end{equation}
Now, $T\subset S\subset\PP^3$ is defined by
\begin{equation}
    T:=S\cap\cV(B_0B_1-A_0^2)=\cV(A_0A_5,B_0B_1-A_0^2).
\end{equation}
So $T$ is a union of 3 components $T=\ell_1\cup\ell_2\cup\ell_3$,
where $\ell_1$ and $\ell_2$ coincide with the lines
$\gamma(\{\infty\}\times\PP^1)$ and $\gamma(\PP^1\times\{\infty\})$
respectively, and $\ell_3$ is a zero of a nontrivial section of
$\mathcal O(1,1)$. Now,
\begin{equation}
S\backslash(\ell_1\cup\ell_2)\cong\PP^1\times \PP^1\setminus
(\PP^1\times\{\infty\}\cup\{\infty\}\times\PP^1)=\mathbb A^2
\end{equation}
has affine coordinates $b,d$ and then $\ell_3\cap \mathbb A^2$ has
defining ideal $1-bd$, so is isomorphic to $\mathbb G_m$. Thus we
get
\begin{equation}
    S\backslash T \cong \AAA^2\backslash\GG_m.
\end{equation}
Since $\GG_m$ is closed in $\AAA^2$, we can consider an exact
sequence
\begin{equation}
\longrightarrow H^1_c(\AAA^2)\longrightarrow
H^1_c(\GG_m)\longrightarrow H^2_c(S\backslash T) \longrightarrow
H^2_c(\AAA^2) \longrightarrow.
\end{equation}
Now it follows that
\begin{equation}
    H^2_c(S\backslash T)\cong H^1_c(\GG_m)\cong \QQ(0).
    \label{c128}
\end{equation}
By (\ref{c117}), (\ref{c120}) and (\ref{c128}), we get
\begin{equation}
    \QQ(0)\cong H^2_c(S\backslash T)\cong
    H^2_c(\cV(A_1)\cap R_1\backslash Z_1)\cong H^3_c(R\backslash Z).
\end{equation}
\end{proof}
\medskip
The isomorphism (\ref{c89}) now yields the desired
\begin{equation}
    H^8_{prim}(X)\cong H^3_c(R\backslash Z)(-2)\cong \QQ(-2).
\end{equation}
This concludes the proof.
\end{proof}
\newpage
\section{Generalized zigzag graphs}
\begin{definition}
Fix some $t\geq 1$ and  consider a set $V(\Gamma)$ of $t+2$ vertexes
$u_i$, $1\leq i\leq t+2$. Define $p(u_1,\dots,u_{t+2})$ to be the
set of $t+1$ edges $(u_i,u_{i+1})$, $1\leq i\leq t+1$. Let
$E(\Gamma):=p(u_1,\dots,u_{t+2})$. Now choose some positive integers
$l_i$ for $1\leq i\leq t$ with $l_1\geq 2$ and $l_t\geq 2$. For each
$i$, $1\leq i\leq t$, we add $l_i-1$ new vertexes $v_{i j}$, $1\leq
j\leq l_i-1$, and $l_i$ new edges $p(u_i,v_{i 1},\ldots,
v_{i\,l_i-1}, u_{i+2})$, and $l_i-1$ edges $(v_{i j},u_{i+1})$,
$1\leq j\leq l_i-1$. Finally, we add an edge $(u_1,u_{t+2})$. We
call the constructed graph $\Gamma=(V(\Gamma),E(\Gamma))$ the
\emph{generalized zigzag} graph $GZZ(l_1,\dots,l_t)$.
\end{definition}
For $GZZ(l_1,\dots,l_t)$ we define $n=1+\sum_{i=1}^t l_i$. The graph
$GZZ(l_1,\dots,l_t)$ has $n+1$ vertexes, $2n$ edges and the Betti
number equals $n$. Thus $GZZ(l_1,\dots,l_t)$ is a logarithmically
divergent graph.
\begin{example}\label{Ex1}
    The graph $GZZ(3,2,3,4)$ looks like
\end{example}
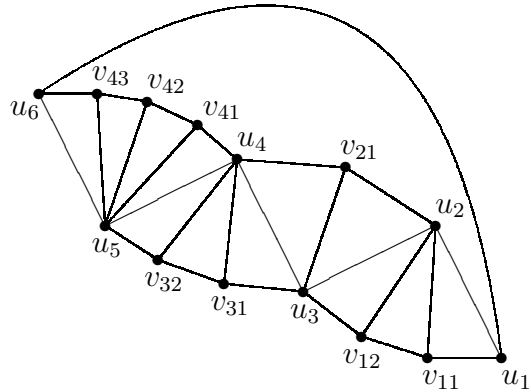
\mathstrut\quad\quad\quad\qquad\begin{picture}(200,160)
\put(30,120){\line(1,-2){25}} \put(55,70){\line(2,1){50}}
\put(105,95){\line(1,-2){25}} \put(130,45){\line(2,1){50}}
\put(180,70){\line(1,-2){25}} 
\put(30,120){\circle*{4}} \put(55,70){\circle*{4}}
\put(105,95){\circle*{4}} \put(130,45){\circle*{4}}
\put(180,70){\circle*{4}} \put(205,20){\circle*{4}}
\qbezier(30,120)(180,220)(205,20)%
\qbezier(55,70)(52,120)(52,120) \qbezier(55,70)(71,117)(71,117)
\qbezier(30,120)(52,120)(52,120) \qbezier(52,120)(71,117)(71,117)
\qbezier(55,70)(90,108)(90,108) \qbezier(71,117)(90,108)(90,108)
\qbezier(90,108)(105,95)(105,95)
\qbezier(105,95)(75,57)(75,57) \qbezier(105,95)(100,48)(100,48)
\qbezier(55,70)(75,57)(75,57) \qbezier(75,57)(100,48)(100,48)
\qbezier(100,48)(130,45)(130,45)
\qbezier(130,45)(146,92)(146,92) \qbezier(105,95)(146,92)(146,92)
\qbezier(146,92)(180,70)(180,70)
\qbezier(180,70)(152,28)(152,28) \qbezier(180,70)(177,20)(177,20)
\qbezier(130,45)(152,28)(152,28) \qbezier(152,28)(177,20)(177,20)
\qbezier(177,20)(205,20)(205,20)
\put(52,120){\circle*{4}} \put(71,117){\circle*{4}}
\put(90,108){\circle*{4}} \put(75,57){\circle*{4}}
\put(100,48){\circle*{4}} \put(146,92){\circle*{4}}
\put(152,28){\circle*{4}} \put(177,20){\circle*{4}}
\put(19,112){$u_6$} \put(50,60){$u_5$} \put(105,100){$u_4$}
\put(125,35){$u_3$} \put(180,75){$u_2$} \put(205,10){$u_1$}
\put(175,10){$v_{1 1}$} \put(145,18){$v_{1 2}$} %
\put(144,97){$v_{2 1}$} \put(95,38){$v_{3 1}$} %
\put(70,47){$v_{3 2}$} \put(92,112){$v_{4 1}$} %
\put(71,122){$v_{4 2}$} \put(50,125){$v_{4 3}$}
\end{picture}

\begin{example}
The wheel with spokes graph $WS_n$ is isomorphic to the generalized
zigzag graph $GZZ(n-1)$ , $n\geq 3$.
\end{example}

\begin{example}\label{E2.2.4}
The zigzag graph $ZZ_n$ is isomorphic to the $GZZ(2,1,\dots,1,2)$
(with $n-5$\: 1's in the middle) for $n\geq 5$.
\end{example}



\begin{theorem} \label{T2.2.5}
A generalized zigzag graph $\Gamma=GZZ(l_1,\dots,l_t)$ is
primitively log divergent.
\end{theorem}
\begin{proof}
We need to prove that for any subgraph $\Gamma'\subset\Gamma$ the
inequality $|E(\Gamma')|>2h_1(\Gamma')$ holds, which means that
$\Gamma'$ is not logarithmically divergent. We do not distinguish
between a graph and it's set of edges. Because our graph $\Gamma$ is
planar, it partitions the plain into exactly $h_1+1$ pieces. This is
a good way to compute $h_1$. We will call the loops of length 3
simple loops. We can order the simple loops from the the right
bottom corner to the left top keeping in mind the drawing like in
Example \ref{Ex1}. Formally, let $\Delta_{1}=p(u_1,u_2,v_{1 1},u_1)$
and for each $i$ we define the next simple loop $\Delta_{i+1}$ to be
a simple loop which has a common edge with $\Delta_i$ but was not
already labeled. Define $\Gamma_0:=\Gamma\backslash (u_1,u_{t+2})$.
The main point of the proof is the following. The graph $\Gamma_0$
is a strip of $\Delta$'s, for each i, $1\leq i\leq k$, we can cut
this strip along $(u_i,u_{i+1})$, turn over one piece and glue along
the same edge. Denote this operation by $\phi_i$. This gives a map
\begin{equation}
    \phi:=\phi_t\circ\ldots\circ\phi_2:\Gamma_0\longrightarrow
    \hat{\Gamma}_0,
\end{equation}
where $\hat{\Gamma}_0$ is isomorphic to $WS_n$ without one boundary
edge; this graph is topologically the same as a half of $WS_n$, we
denote it by $hWS_n$. Note that the maps $\phi_i$ and $\phi$ are the
isomorphisms between sets of edges of the graphs in the described
way. On some vertexes this map is not single-valued. For the Example
\ref{Ex1} we have the following
$\hat{\Gamma}_0=hWS_{13}$\\
\mathstrut\quad\quad\quad\quad\quad\quad
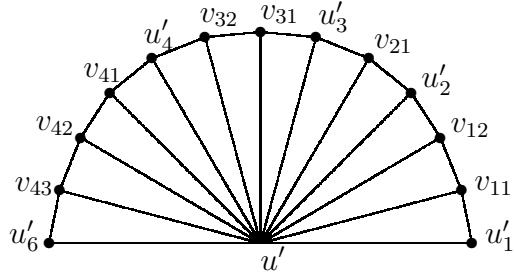
\begin{picture}(200,130)
\put(110,25){\line(-1,0){80}} \put(110,25){\line(1,0){80}}
\put(110,25){\line(0,1){80}} %
\qbezier(53,82)(110,25)(110,25) \qbezier(167,82)(110,25)(110,25)
\qbezier(34,45)(110,25)(110,25) \qbezier(186,45)(110,25)(110,25)
\qbezier(42,65)(110,25)(110,25) \qbezier(178,65)(110,25)(110,25)
\qbezier(69,95)(110,25)(110,25) \qbezier(151,95)(110,25)(110,25)
\qbezier(89,103)(110,25)(110,25) \qbezier(131,103)(110,25)(110,25)
\put(190,25){\circle*{4}} \put(186,45){\circle*{4}}
\put(178,65){\circle*{4}} \put(167,82){\circle*{4}}
\put(151,95){\circle*{4}} \put(131,103){\circle*{4}}
\put(110,105){\circle*{4}} \put(89,103){\circle*{4}}
\put(69,95){\circle*{4}} \put(53,82){\circle*{4}}
\put(42,65){\circle*{4}} \put(34,45){\circle*{4}}
\put(30,25){\circle*{4}}
\qbezier(190,25)(186,45)(186,45) \qbezier(186,45)(178,65)(178,65)
\qbezier(178,65)(167,82)(167,82) \qbezier(167,82)(151,95)(151,95)
\qbezier(151,95)(131,103)(131,103)\qbezier(131,103)(110,105)(110,105)
\qbezier(110,105)(89,103)(89,103)\qbezier(89,103)(69,95)(69,95)
\qbezier(69,95)(53,82)(53,82) \qbezier(53,82)(42,65)(42,65)
\qbezier(42,65)(34,45)(34,45) \qbezier(34,45)(30,25)(30,25)
\put(195,25){$u_1'$} \put(191,45){$v_{1 1}$} \put(182,67){$v_{1 2}$}
\put(171,84){$u_2'$} \put(153,99){$v_{2 1}$} \put(132,108){$u_3'$}
\put(110,110){$v_{3 1}$} \put(87,109){$v_{3 2}$}
\put(66,101){$u_4'$} \put(43,88){$v_{4 1}$} \put(25,68){$v_{4 2}$}
\put(17,45){$v_{4 3}$} \put(15,25){$u_6'$} \put(110,15){$u'$}
\end{picture}
\par\noindent The vertex $u_i$ under described operations goes to
$u_{i-1}'$, or $u_i'$, or $u'$ depending on the edge that we take.
We can label the simple loops of $hWS_n$ from the right to the left
by $\hat{\Delta}_1,\ldots,\hat{\Delta}_{n-1}$, these are the images
of $\Delta$'s
\begin{equation}
    \phi(\Delta_i)=\hat{\Delta}_i.
\end{equation}
Each $\phi_i$ preserves loops; this means that a subgraph
$\gamma\subset\phi_{i-1}\ldots\phi_2(\Gamma)$ is a loop if and only
if $\phi_i(\gamma)$ is a loop of the same length. Thus this
condition holds for $\phi$. It follows that $\Gamma_0$ and $hWS_n$
have the same Betti numbers. Moreover, for each subgraph
$\Gamma_0''\subset\Gamma_0$ we have
\begin{equation}
    h_1(\Gamma_0'')=h_1(\phi(\Gamma_0'')).
    \label{e9}
\end{equation}
To involve the "special" edge $(u_1, u_{t+2})$ into consideration,
note that if the graph $\Gamma_0''$ is disconnected and we have no
path $p'(u_1,\ldots,u_{t+1})$ with endpoints $u_1$ and $u_{t+1}$,
then the adding of $(u_1,u_{t+2})$ doesn't change the Betti number;
otherwise this increases the number by one.
\begin{equation}
h(\Gamma_0''\cup(u_1,u_{t+1}))=\left\{
\begin{aligned}
&h(\Gamma_0''),\quad\quad\; p'(u_1,\ldots,u_{t+2})\not\subset\Gamma_0'',\\
&h(\Gamma_0'')+1,\;\;\;\text{otherwise}.
\end{aligned}
\right.
\end{equation}
This proves that we can extend the map $\phi$ to
\begin{equation}
    \bar{\phi}: \Gamma\longrightarrow\hat{\Gamma}
\end{equation}
which maps our graph to $\hat{\Gamma}$; this graph is nothing but
$\hat{\Gamma}_0\cong hWS_n$ compactified by adding the missing
boundary edge and is isomorphic to $WS_n$. The map $\bar{\phi}$
satisfies the same condition as $\phi$ in (\ref{e9}). For the
example of $hWS_{13}$ above, we add the edge $(u_1,u_6)$ on the
drawing and get $WS_{13}$.\medskip

\noindent It remains to prove that $WS_n$ is primitively divergent.
We label the spokes by $a_i$, and any other edge that has common
vertexes with $a_i$ and $a_{i+1}$ (the indices modulo $n$) for some
$i$ is denoted by $b_i$. Take a subgraph $\Gamma'\subset WS_n$ and
assume that $WS_n\backslash\Gamma'$ has $p$ $b$-edges and $q$
$a$-edges. Let $\Gamma$ by an intermediate graph which we get from
$WS_n$ after removing this $p$ $a$'s. This graph $\Gamma$ is a
disjoint union of $p$ graphs $\Gamma_i$ isomorphic to $hWS_{n_i}$
for some $n_i\leq 1$ and $1\leq i\leq p$, assuming $hWS_1$ and
$hWS_2$ to be an edge and a triangle respectively. It follows that
\begin{equation}
    h_1(\Gamma'')=\bigoplus_{i=1}^p h_1(\Gamma''_i)
\end{equation}
for any subgraph $\Gamma''\subset\Gamma$ and
$\Gamma''_i=\Gamma_i\cap\Gamma''$, $1\leq i\leq p$.

To get the initial subgraph $\Gamma'\subset\Gamma\subset WS_n$ we
need to drop $q$ $b$'s. Assume that we drop $q_i$ $b$'s in
$\Gamma_i$ and get $\Gamma_i'\cong hWS_{n_i}'$ for $1\leq i\leq p$.
One can easily compute
\begin{equation}
h_1(hWS_{n_i})-h_1(hWS_{n_i}')= \left\{
\begin{aligned}
    &q_i-1,\quad &\text{if}\; q_i=n_i\\
    &q_i,\quad &\text{otherwise}.\\
\end{aligned}
\right.\label{e10}
\end{equation}
Assume that we have $p_1$ $i$'s with the upper assumption, $p_1\leq
p$. Taking the sum over all $i$, we get
\begin{equation}
    h_1(\Gamma)-h_1{\Gamma'} = q -p_1.
    \label{e11}
\end{equation}
Now we recall that $\Gamma$ is the $WS_n$ without $p$\, $a$'s. Each
dropping of $a$-edge decreases the Betti number by one. Thus,
\begin{equation}
    h_1(\Gamma)-p.
\end{equation}
Together with (\ref{e11}), this gives us
\begin{equation}
    h_1(\Gamma')-p-q+p_1.
\end{equation}
Now we can compute
\begin{equation}
    2h_1(\Gamma')-|E(\Gamma')|=2(n-p-q+p_1)-(2n-p-q)=-p-q+2p_1\leq 0
\end{equation}
since $p_1\leq p$ by definition, and $p_1\leq q$ since each of $p_1$
$i$'s gives us some $q_i=n_i\neq 0$. The equality can hold only when
$q_i=n_i=1$, $p=p_1$, but this means that all edges are dropped.
This concludes the proof.
\end{proof}

Let $X\subset\PP^{2n-1}$ be a graph hypersurface. We consider the
middle dimensional Betti cohomology $H^{mid}(X)=H^{2n-2}(X)$. By
Deligne's theory of MHS (\cite{De2}, \cite{De3}), there is a
$\QQ$-mixed Hodge structure associated to $H^{mid}(X)$. We try to
study the graded pieces of weight filtration $W$:
$\gr^W_i(H^{2n-2}(X))$, $0\leq i\leq 2n-2$.

\begin{theorem}
For the hypersurface $X$ associated to a generalized zigzag
  graph\newline
$GZZ(l_1,\dots,l_t)$, one has an inclusion
\begin{equation}
    \gr^W_4(H^{mid}(X))=W_4(H^{mid}(X))=W_5(H^{mid}(X))\cong\QQ(-2).
\end{equation}
\end{theorem}

\begin{proof} Denote $GZZ(l_1,\dots,l_t)$ by $\Gamma$.
We consider the case when $t$ is even and start with labeling of
edges and choosing orientations. For simplicity, let $n_0:=0$ and
\begin{equation}
n_i:=\sum_{j=1}^i l_j, \quad \text{for}\;\: 1\leq i\leq t.
\end{equation}
For each $i$, $1\leq i\leq t$, define $e_{n_{i-1}+1}:=(u_{i+1},u_i)$
for odd $i$ and $e_{n_{i-1}+1}:=(u_i,u_{i+1})$ for even $i$,
\begin{equation}
e_{n_{i-1}+j}:=\left\{
\begin{aligned}
(u_{i+1},v_{i\, j-1})\quad &\text{for}\;\: 2\leq j\leq l_i,\; i\;\text{odd},\\
(v_{i\, j-1},u_{i+1})\quad &\text{for}\;\: 2\leq j\leq l_i,\; i\;\text{even}.\\
\end{aligned}
\right.
\end{equation}
Together with $e_{n_t+1}:=(u_{t+2},u_{t+1})$ for even $t$ and
$e_{n_t+1}:=(u_{t+1},u_{t+2})$ for odd $t$, these are the first
$n_t+1=:n$ edges. Now, for each $i$, $1\leq i\leq t$, define
$e_{n+n_{i-1}+1}:=(v_{i 1},u_i)$,
\begin{equation}
    e_{n+n_{i-1}+j}:=(v_{i j},v_{i j-1}), \quad \text{for}\;\: 2\leq
j\leq l_i-1,
\end{equation}
and $e_{n+n_{i-1}+l_i}:=(u_{i+2},v_{i\, l_i-1})$. Roughly speaking,
all edges are oriented from the left top corner to the right bottom
and from the right top corner to the left bottom corner. Define
$e_{2n}:=(u_1,u_{t+2})$.
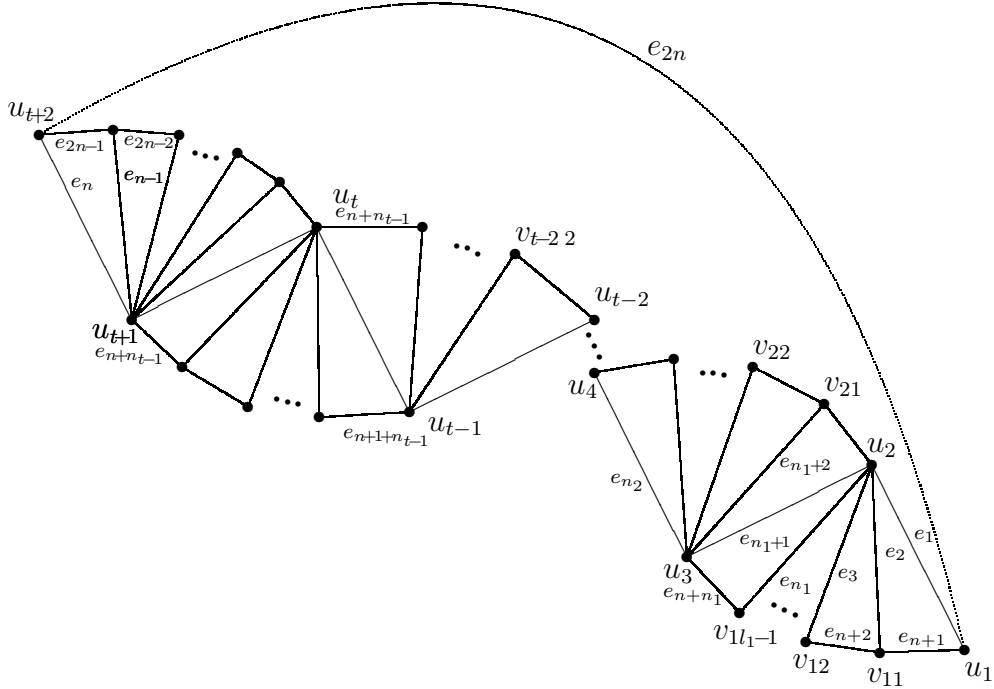
\newline
\begin{picture}(200,290)
\put(20,220){\line(1,-2){35}} \put(55,150){\line(2,1){70}}
\put(125,185){\line(1,-2){35}} \put(160,115){\line(2,1){70}}
\put(230,130){\line(1,-2){35}} \put(265,60){\line(2,1){70}}
\put(335,95){\line(1,-2){35}}
\put(20,220){\circle*{4}} \put(55,150){\circle*{4}}
\put(125,185){\circle*{4}} \put(160,115){\circle*{4}}
\put(230,150){\circle*{4}} \put(230,130){\circle*{4}}
\put(265,60){\circle*{4}} \put(335,95){\circle*{4}}
\put(370,25){\circle*{4}}
\qbezier(20,220)(290,380)(370,25)%
\qbezier(48,222)(55,150)(55,150) 
\qbezier(73,220)(55,150)(55,150) 
\qbezier(95,213)(55,150)(55,150) \qbezier(111,202)(55,150)(55,150)
\put(48,222){\circle*{4}} 
\put(73,220){\circle*{4}} 
\put(95,213){\circle*{4}} \put(111,202){\circle*{4}}
\qbezier(20,220)(48,222)(48,222) \qbezier(48,222)(73,220)(73,220)
\qbezier(95,213)(111,202)(111,202)
\qbezier(111,202)(125,185)(125,185)%
\put(79,213){\circle*{2}} \put(83,212){\circle*{2}}
\put(87,211){\circle*{2}}

\qbezier(74,132)(125,185)(125,185)\qbezier(99,117)(125,185)(125,185)
\qbezier(126,113)(125,185)(125,185) %
\put(74,132){\circle*{4}} \put(99,117){\circle*{4}}
\put(126,113){\circle*{4}} 
\qbezier(55,150)(74,132)(74,132) \qbezier(74,132)(99,117)(99,117)
\qbezier(126,113)(160,115)(160,115)%
\put(110,120){\circle*{2}} \put(114,119){\circle*{2}}
\put(118,118){\circle*{2}}
    \qbezier(200,175)(160,115)(160,115)
    \qbezier(165,185)(160,115)(160,115) %
    \put(200,175){\circle*{4}}
    \put(165,185){\circle*{4}} %
    \qbezier(125,185)(165,185)(165,185)
    \qbezier(230,150)(200,175)(200,175)%
    \put(178,178){\circle*{2}} \put(182,177){\circle*{2}}
    \put(186,175){\circle*{2}}
    \put(228,144){\circle*{2}} \put(230,140){\circle*{2}}
    \put(232,136){\circle*{2}}
\qbezier(260,135)(265,60)(265,60)\qbezier(290,132)(265,60)(265,60)
\qbezier(317,118)(265,60)(265,60) %
\put(260,135){\circle*{4}} \put(290,132){\circle*{4}}
\put(317,118){\circle*{4}} 
\qbezier(230,130)(260,135)(260,135)
\qbezier(290,132)(317,118)(317,118)
\qbezier(317,118)(335,95)(335,95)%
\put(271,130){\circle*{2}} \put(275,130){\circle*{2}}
\put(279,129){\circle*{2}}
\qbezier(285,39)(335,95)(335,95)\qbezier(310,28)(335,95)(335,95)
\qbezier(338,24)(335,95)(335,95) %
\put(285,39){\circle*{4}} \put(310,28){\circle*{4}}
\put(338,24){\circle*{4}} 
\qbezier(265,60)(285,39)(285,39) \qbezier(310,28)(338,24)(338,24)
\qbezier(338,24)(370,25)(370,25) %
\put(298,42){\circle*{2}} \put(302,40){\circle*{2}}
\put(306,38){\circle*{2}}
\put(370,15){$u_1$} \put(333,13){$v_{1 1}$} \put(305,18){$v_{1 2}$}
\put(275,29){$v_{1 l_1\!-\!1}$}\put(333,100){$u_2$}
\put(256,52){$u_3$} \put(317,123){$v_{2 1}$} %
\put(290,137){$v_{2 2}$} \put(220,122){$u_4$}
\put(8,227){$u_{t\!+\!2}$}
\put(40,143){$u_{t\!+\!1}$} \put(40,143){$u_{t\!+\!1}$} 
\put(351,67){$\s{e_1}$} \put(340,60){$\scs{e_2}$}
\put(322,52){$\scs{e_3}$} \put(300,49){$\s{e_{n_1}}$}
\put(285,65){$\s{e_{n_1\!+\!1}}$} \put(300,94){$\s{e_{n_1\!+\!2}}$}
\put(236,88){$\s{e_{n_2}}$}
\put(345,28){\s{e_{n+1}}} \put(256,45){\s{e_{n+n_1}}}
\put(317,30){\s{e_{n+2}}}

\put(26,215){$\s{e_{2n\!-\!1}}$} \put(52,216){$\s{e_{2n\!-\!2}}$}
\put(41,136){\s{e_{n\!+\!n_{t\!-\!1}}}}

\put(32,200){$\s{e_n}$} \put(52,203){$\scs{e_{n\!-\!1}}$}
\put(52,203){$\scs{e_{n\!-\!1}}$}
\put(132,189){\s{e_{n+n_{t\!-\!1}}}}
\put(135,105){\s{e_{n\!+\!1\!+\!n_{t\!-\!1}}}}
\put(250,250){$e_{2n}$} \put(131,194){$u_t$}
\put(167,108){$u_{t-1}$} \put(230,157){$u_{t-2}$}
\put(200,180){$v_{t\!-\!2\:2}$}

\end{picture}
\medskip\\
$\mathstrut$ \!\!\!\!\!
\begin{tabular}{p{0.1cm}|p{0.1cm}p{0.1cm}p{0.1cm}p{0.1cm}p{0.1cm}p{0.1cm}
p{0.1cm}p{0.1cm}p{0.1cm}p{0.1cm}p{0.1cm}p{0.1cm}p{0.1cm}p{0.1cm}p{0.1cm}p{0.1cm}
p{0.1cm}p{0.1cm}p{0.1cm}p{0.1cm}p{0.1cm}p{0.1cm}
p{0.1cm}p{0.1cm}p{0.1cm}p{0.1cm}|}
       & \s{1}&\s{2} &\s{3}&\s{4}&\s{5}&\s{6}&\s{7}&\s{8} &\s{9}&\s{10}&\s{11}&\s{12} &\s{13} &\s{14} &\s{15} &\s{16} &\s{17} &\s{18} &\s{19}&\s{20} & \s{21}&\s{22} &\s{23} &\s{24} &\s{25} &\s{26}\\ \hline
    \s{1}  &\sk &\sj &\so &\so &\so  &\so  &\so  &\so  &\so   &\so  &\so   &\so   &\so &\sj  &\so  &\so  &\so  &\so   &\so  &\so &\so  &\so  &\so  &\so   &\so  &\so    \\
    \s{2}  &\so &\sk &\sj &\so &\so  &\so  &\so  &\so  &\so   &\so  &\so   &\so   &\so &\so &\sj  &\so  &\so  &\so  &\so   &\so  &\so  &\so  &\so  &\so   &\so  &\so    \\
    \s{3}  &\so &\so &\sk &\sj &\so  &\so  &\so  &\so  &\so   &\so  &\so   &\so   &\so &\so &\so  &\sj  &\so  &\so  &\so   &\so  &\so  &\so  &\so  &\so   &\so  &\so\\
    \s{4}  &\so &\so &\so &\sj &\sk  &\so  &\so  &\so  &\so   &\so  &\so   &\so   &\so &\so &\so  &\so  &\sj  &\so  &\so   &\so  &\so  &\so  &\so  &\so   &\so  &\so\\
    \s{5}  &\so &\so &\so &\so &\sj  &\sk  &\so  &\so  &\so   &\so  &\so   &\so   &\so &\so &\so  &\so  &\so  &\sj  &\so   &\so  &\so  &\so  &\so  &\so   &\so  &\so\\
    \s{6}  &\so &\so &\so &\so &\so  &\sk  &\sj  &\so  &\so   &\so  &\so   &\so   &\so &\so &\so  &\so  &\so  &\so  &\sj   &\so  &\so  &\so  &\so  &\so   &\so  &\so\\
    \s{7}  &\so &\so &\so &\so &\so  &\so  &\sk  &\sj  &\so   &\so  &\so   &\so   &\so &\so &\so  &\so  &\so  &\so  &\so   &\sj  &\so  &\so  &\so  &\so   &\so  &\so\\
    \s{8}  &\so &\so &\so &\so &\so  &\so  &\so  &\sk  &\sj   &\so  &\so   &\so   &\so &\so &\so  &\so  &\so  &\so  &\so   &\so  &\sj  &\so  &\so  &\so   &\so  &\so\\
    \s{9}  &\so &\so &\so &\so &\so  &\so  &\so  &\so  &\sj   &\sk  &\so   &\so   &\so &\so &\so  &\so  &\so  &\so  &\so   &\so  &\so  &\sj  &\so  &\so   &\so  &\so\\
    \s{\!\!10} &\so &\so &\so &\so &\so  &\so  &\so  &\so  &\so   &\sj  &\sk   &\so   &\so &\so &\so  &\so  &\so  &\so  &\so   &\so  &\so  &\so  &\sj  &\so   &\so  &\so\\
    \s{\!\!11} &\so &\so &\so &\so &\so  &\so  &\so  &\so  &\so   &\so  &\sj   &\sk   &\so &\so &\so  &\so  &\so  &\so  &\so   &\so  &\so  &\so  &\so  &\sj   &\so  &\so\\
    \s{\!\!12} &\so &\so &\so &\so &\so  &\so  &\so  &\so  &\so   &\so  &\so   &\sj   &\sk &\so &\so  &\so  &\so  &\so  &\so   &\so  &\so  &\so  &\so  &\so   &\sj  &\so\\
    \s{\!\!13} &\sj &\so &\so &\so &\so  &\so  &\so  &\so  &\so   &\so  &\so   &\so   &\so &\so &\so  &\so  &\sj  &\sj  &\so   &\so  &\so  &\sj  &\sj  &\sj   &\sj  &\sj\\
\end{tabular}
\medskip\\
For building the table, we take the small loops from right bottom
corner of the drawing to the left top corner, and the last loop to
be chosen is the loop with the edge $(u_1,u_{t+2})$. Because of lack
of space, we draw the table for the graph in Example \ref{Ex1}.
\medskip\\
\noindent Now we take $2n$ variables $T_1,\ldots,T_{2n}$ and build a
matrix $M(T)$ as the sum of elementary matrices (see Section 2,
Chapter 1). After a change of the coordinates similar to the case of
$ZZ_5$, we get the matrix
\begin{equation}
 M_{GZZ}=
\left(
  \begin{array}{ccccccccccccc}
    \s{\!B_0\!}& \s{\!A_0\!}&\so    & \so & \so &\so & \so & \so &\so & \so & \so & \so &\s{\!A_{17}\!} \\
    \s{\!A_0\!}& \s{\!B_1\!}&\s{A_1}& \so & \so &\so & \so & \so &\so & \so & \so & \so & \so \\
    \so    & \s{\!A_1\!}&\s{\!B_2\!}&\s{\!A_{2}\!} & \so &\so & \so & \so &\so & \so & \so & \so & \so \\
    \so    & \so    &\s{\!A_2\!}&\s{C_3} &\s{\!A_{3}\!} &\so & \so & \so &\so & \so & \so  & \so &\s{\!A_{16}\!}\\
    \so    & \so    &\so & \s{\!A_3\!} & \s{\!C_4\!} &\s{\!A_{4}\!} & \so & \so &\so & \so & \so & \so &\s{\!A_{15}\!} \\
    \so    & \so    &\so & \so & \s{\!A_4\!} &\s{\!B_5\!} &\s{\!A_{5}\!} & \so &\so & \so & \so & \so & \so \\
    \so    & \so    &\so & \so & \so &\s{\!A_5\!} & \s{\!B_6\!} &\s{\!A_{6}\!} &\so & \so & \so & \so & \so \\
    \so    & \so    &\so & \so & \so &\so & \s{\!A_6\!} & \s{\!B_7\!} &\s{\!A_{7}\!} & \so & \so & \so & \so \\
    \so    & \so    &\so & \so & \so &\so & \so & \s{\!A_7\!} &\s{\!C_8\!} &\s{\!A_{8}\!} & \so & \so &\s{\!A_{14}\!}\\
    \so    & \so    &\so & \so & \so &\so & \so & \so &\s{\!A_8\!} & \s{\!C_9\!} &\s{\!A_{9}\!} & \so &\s{\!A_{13}\!} \\
    \so    & \so    &\so & \so & \so &\so & \so & \so &\so &\s{\!A_9\!} & \s{\!C_{10}\!}&\s{\!A_{10}\!} &\s{\!A_{12}\!}  \\
    \so    & \so    &\so & \so & \so &\so & \so & \so &\so & \so &\s{\!A_{10}\!} & \s{\!B_{12}\!} & \s{\!A_{11}\!} \\
    \s{\!A_{17}\!} & \so    &\so & \s{\!A_{16}\!} & \s{\!A_{15}\!} &\so & \so & \so &\s{\!A_{14}\!} &\s{\!A_{13}\!} &\s{\!A_{12}\!} & \s{\!A_{11}\!} & \s{\!B_{12}\!} \\
  \end{array}
\right).
\end{equation}
The $A$'s appear in the last row in the zero column and in the
columns $n_i+j-1$ for all $i\not\equiv t\!\!\mod 2$ , $1\leq i\leq
t$, and all $1\leq j\leq l_i$. In the same columns (but  $0$ and
$n-2$) we have $C$'s in the main diagonal. This $C$'s are defined by
\begin{equation}
C_k:=\left\{\begin{aligned}
&A_v+A_{k-1}-A_k,\quad& &k_i,l_{i+1}>1,i\neq 0,\\
&A_v-A_{k-1}-A_k,\quad& &k_i+j, 1\leq j\leq l_{i+1}-2,\\
&A_v-A_{k-1}+A_k,\quad& &k_{i+1}-1, l_{i+1}>1,i\neq t-1,\\
&A_v+A_{k-1}+A_k,\quad& &k_i,l_{i+1}=1,\\
\end{aligned}
\right.
\end{equation}
where $i\not\equiv t\!\!\mod 2$, and $A_v$ is always in the last row
in the same column as $C_k$. Formally, if $k_i+j-1$, then
\begin{equation}
    v=v(k)-2+\sum_{\substack{r=i+2\\r\not\equiv
    t\:\text{mod}\:2}}^{t-1} l_r +l_{i+1}-j.
\end{equation}
Sometimes we denote by $A_m$ the entry in the left bottom corner of
$M_{GZZ}$.

For the case of odd $t$ we can derive the tables and the matrices
from the even case. Indeed, consider some
$\Gamma'=GZZ(l_1,\dots,l_t)$ with even $t$ and let $\Gamma$ be the
graph which we get from $\Gamma'$ after forgetting edges of simple
loops $\Delta_1$,\ldots,$\Delta_{l_1}$ (see Theorem \ref{T2.2.5} for
definition), we assume that $(u_2,u_3)$ remains, and we take
$(u_{t+2},u_2)$ instead of $(u_{t+2},u_1)$. So,
$\Gamma=GZZ(l_2,\dots,l_t)$. Constructing everything similar, the
table for $\Gamma$ is that for $\Gamma'$ without first $l_1$ rows.
The matrix of $\Gamma$ looks similar to that of $\Gamma'$ with the
same assumptions on $A$'s in the last row and on $C$'s.

Consider the projective space $\PP^{2n-1}$ with coordinates all the
$A_i$'s and $B_j$'s appearing in the matrix and define
\begin{equation}
    X:=\cV(\det(M_{GZZ}))=\cV(I_n)\subset\PP^{2n-1},
\end{equation}
where
\begin{equation}
 M_{GZZ}=
\left(
  \begin{array}{cccc}
  \ddots &  \vdots &  \vdots & \vdots\\
  \ldots & C_{n-3} & A_{n-3} & A_{n-1}\\
  \ldots & A_{n-3} & B_{n-2} & A_{n-2}\\
  \ldots &  A_{n-1} & A_{n-2} & B_{n-1}\\
  \end{array}
\right).
\end{equation}
Since $l_t>1$, the entry $a_{n-3\,n-3}$ is really not independent,
thus $C_{n-3}$.\medskip\newline \emph{\underline{Step 1.}} For the
closed subscheme $\cV(I_n,I_{n-1})\subset X$ we have the
localization sequence
\begin{equation}
\begin{aligned}
\rightarrow H^{2n-2}_c(X&\backslash\cV(I_n,I_{n-1}))\rightarrow
H^{2n-2}(X)\rightarrow\\
&H^{2n-2}(\cV(I_n,I_{n-1}))\rightarrow H^{2n-1}_c(X\backslash
\cV(I_n,I_{n-1}))\rightarrow \label{e20}. \end{aligned}
\end{equation}
We can write
\begin{equation}
    I_n=B_{n-1}I_{n-1}-G_{n-1},
    \label{e21}
\end{equation}
where $G_n$ is independent of $B_{n-1}$. Projecting from the point
where all the variables but $B_{n-1}$ are zero, we get
\begin{equation}
    X\backslash \cV(I_n,I_{n-1})\cong
    \PP^{2n-2}\backslash\cV(I_{n-1}).
    \label{e22}
\end{equation}
Because $I_{n-1}$ is independent of $A_{n-2}$ and $A_m$,
\emph{Theorem B} ($N=2n-2$, $k=0$, $t=2$) applied to the scheme on
the right hand side of (\ref{e22}) implies
\begin{equation}
    H^i_c(X\backslash \cV(I_n,I_{n-1}))\cong
    H^i_c(\PP^{2n-2}\backslash\cV(I_{n-1}))=0
\end{equation}
for $i<2n$. The sequence (\ref{e20}) implies an isomorphism
\begin{equation}
    H^{2n-2}(X)\cong H^{2n-2}(\cV(I_n,I_{n-1})).
    \label{e24}
\end{equation}
By (\ref{e21}), one has
\begin{equation}
\cV(I_n,I_{n-1})\cong\cV(I_{n-1},G_{n-1})^{(2n-1)}.
\end{equation}
Both polynomials to the right are independent of are independent of
$B_{n-1}$. \emph{Theorem A} ($N=2n-1$, $k=2$, $t=1$) and (\ref{e24})
imply
\begin{equation}
    H^{2n-2}(X)\cong H^{2n-4}(\cV(I_{n-1},G_{n-1}))(-1).
    \label{e25}
\end{equation}
The variety to the right lives in $\PP^{2n-2}$(no $B_{n-1}$). Define
the closed subscheme $\hat{V}\subset\cV(I_{n-1},G_{n-1})$ by
\begin{equation}
\hat{V}:=\cV(I_{n-1},I_{n-2},G_{n-1})\subset\PP^{2n-2}(\,\text{no}\;
B_{n-1}). \label{e26}
\end{equation}
One has an exact sequence
\begin{equation}
\begin{aligned}
\rightarrow
H^{2n-4}_c(\cV(&I_{n-1},G_{n-1})\backslash\hat{V})\rightarrow
H^{2n-4}(\cV(I_{n-1},G_{n-1}))\rightarrow\\
&H^{2n-4}(\hat{V})\rightarrow
H^{2n-3}_c(\cV(I_{n-1},G_{n-1})\backslash \hat{V})\rightarrow
\end{aligned}
\label{e27}
\end{equation}
The polynomial $I_{n-1}$ is independent of $A_{n-2}$ and the
coefficient of $A_{n-2}^2$ in $G_{n-1}$ is $I_{n-2}$. By Theorem
\ref{T1.7}, we have
\begin{equation}
    \cV(I_{n-1},G_{n-1})\backslash\hat{V}\cong \cV(I_{n-1})\backslash
    \cV(I_{n-1},I_{n-2})\subset\PP^{2n-3}(\,\text{no}\; B_{n-1}, A_{n-2}).
    \label{e28}
\end{equation}
The polynomials $I_{n-2}$ and $I_{n-1}$ are independent of $A_{n-1}$
and $A_{m}$. Applying \emph{Theorem B} ($N=2n-3$, $k=1$, $t=2$), we
get
\begin{equation}
    H^i_c(\cV(I_{n-1},G_{n-1})\backslash\hat{V})\cong
    H^i_c(\cV(I_{n-1})\backslash\cV(I_{n-1},I_{n-2}))=0
\end{equation}
for $i\leq 2n-3$. The sequence (\ref{e27}) yields
\begin{equation}
    H^{2n-4}(\cV(I_{n-1},G_{n-1}))\cong H^{2n-4}(\hat{V}).
    \label{e30}
\end{equation}
By the Theorem \ref{T1.5}, the polynomial $G_{n-1}$ is independent
of $A_{n-2}$ on $\hat{V}$. Thus, $\hat{V}$ is defined by the
vanishing of three polynomials that are independent of $A_{n-2}$.
Applying the \emph{Theorem A} ($N=2n-2$, $k=3$,$t=1$), we get
\begin{equation}
    H^{2n-4}(\hat{V})\cong H^{2n-6}(V)(-1),
    \label{e31}
\end{equation}
where the variety on the right hand side is defined by
\begin{equation}
    V:=\cV(I_{n-1},I_{n-2},G'_{n-1})\subset\PP^{2n-3}(\,\text{no}\;
B_{n-1}, A_{n-1})
\end{equation}
and
\begin{equation}
    G'_{n-1}:=G_{n-1}|_{A_{n-1}=0}.
\end{equation}
Combining (\ref{e25}), (\ref{e30}) and (\ref{e31}), we get
\begin{equation}
    H^{2n-2}(X)\cong H^{2n-6}(V)(-2).
    \label{e34}
\end{equation}
\emph{\underline{Step 2.}} Now we get rid of $B_{n-2}$. We can write
\begin{equation}
    G'_{n-1}=B_{n-2}G_{n-2}-A_{n-3}^2G_{n-3},
\end{equation}
where $G_{n-2}$ and $G_{n-3}$ are considered to be polynomials of
variables $A_{n-1},\ldots, A_{m}$ and $A_n,\ldots, A_{m}$ with
"coefficients" from the matrices $I_{n-2}$ and $I_{n-3}$
respectively. The decomposition follows from the fact that each
coefficient of $G'_{n-1}$ is a factor of some $I_{n-j-1}^{j}$ for
$0\leq j\leq n-2$, and the 3-diagonal matrix $I_{n-j-1}^{j}$ has the
right bottom entry $B_{n-2}$. Define the variety
$\hat{T}_{n-2}\subset V$ by
\begin{equation}
\begin{aligned}
\hat{T}&_{n-2}:=V\cap\cV(G_{n-2})=\\
&\cV(A_{n-3}I_{n-3},I_{n-2},G_{n-2},A_{n-3}G_{n-3}))\subset\PP^{2n-3}(\,\text{no}\;
B_{n-1}, A_{n-2}).
\end{aligned}\label{e36}
\end{equation}
One has an exact sequence
\begin{equation}
H^{2n-7}(\hat{T})\rightarrow
H^{2n-6}_c(V\backslash\hat{T})\rightarrow
H^{2n-6}_{prim}(V)\rightarrow H^{2n-6}_{prim}(\hat{T})\rightarrow.
\end{equation}
Since the defining polynomials of $\hat{T}$ are independent of
$B_{n-2}$, we apply \emph{Theorem A} ($N=2n-3$, $k=4$, $t=1$) and
get
\begin{equation}
0 \rightarrow H^{2n-6}_c(V\backslash\hat{T})\rightarrow
H^{2n-6}_{prim}(V)\rightarrow H^{2n-8}_{prim}(T)(-1)\rightarrow
\label{e38}
\end{equation}
for $T\subset\PP^{2n-4}$ (no $B_{n-1}$, $A_{n-2}$ and $B_{n-2}$)
defined by the same equations as $\hat{T}$. Applying the exact
functors $\gr^W_i$ to the sequence above, we obtain
\begin{equation}
\gr^W_i H^{2n-6}_{prim}(V) \cong \gr^W_i
H^{2n-6}_c(V\backslash\hat{T}), \quad i=0,1.
  \label{e38.5}
\end{equation}
The subscheme $V\backslash\hat{T}\subset V$ is defined by the system
\begin{multline}\;\;
    \left\{
        \begin{aligned}
        I_{n-2}=A_{n-3}I_{n-3}=0\\
        B_{n-2}G_{n-2}-A_{n-3}^2G_{n-3}=0\\
        G_{n-2}\neq 0.\!\!\\
        \end{aligned}
    \right.\;\;\;
  \label{e39}
\end{multline}
Projecting from the point where all the variables but $B_{n-2}$ are
zero and solving the middle equation on $B_{n-2}$, we get an
isomorphism
\begin{equation}
\begin{aligned}
    V\backslash\hat{T} \cong
    \cV(I_{n-2},A_{n-3}I_{n-3})\backslash
    \cV(I_{n-2},A_{n-3}I_{n-3},G_{n-2})\\
    =: U_1 \subset\PP^{2n-4}(\,\text{no}\;
    B_{n-1}, A_{n-2}, B_{n-2}).
    \label{e40}
\end{aligned}
\end{equation}
One has an exact sequence
\begin{multline}
H^{2n-7}(\cV(I_{n-2},A_{n-3}I_{n-3}))\rightarrow
H^{2n-7}(\cV(I_{n-2},A_{n-3}I_{n-3},G_{n-2})) \rightarrow\\
H^{2n-6}_c(U_1) \rightarrow
H^{2n-6}_{prim}(\cV(I_{n-2},A_{n-3}I_{n-3}))\rightarrow.\;\;\;\;
\label{e41}
\end{multline}
The variety $\cV(I_{n-2},A_{n-3}I_{n-3})\subset\PP^{2n-4}$ is
defined by the polynomials that are independent of $A_m$.
\emph{Theorem A} ($N=2n-4$, $k=2$, $t=1$) implies the vanishing of
the rightmost and the leftmost terms, and the sequence simplifies to
\begin{equation}
   H^{2n-6}_c(U_1)\cong H^{2n-7}(\cV(I_{n-2},A_{n-3}I_{n-3},G_{n-2})).
   \label{e42}
\end{equation}
Define
$\hat{S},U_2\subset\cV(I_{n-2},A_{n-3}I_{n-3},G_{n-2})\subset\PP^{2n-4}$
by
\begin{equation}
    \hat{S}:=\cV(I_{n-2},I_{n-3},G_{n-2})
\end{equation}
and $U_2:=\cV(I_{n-2},A_{n-3}I_{n-3},G_{n-2})\backslash\hat{S}$. One
has an exact sequence
\begin{equation}
\begin{aligned}
    \longrightarrow H^{2n-8}_{prim}&(\hat{S})\longrightarrow
H^{2n-7}_c(U_2)\longrightarrow \\
& H^{2n-7}(\cV(I_{n-2},A_{n-3}I_{n-3},G_{n-2})) \longrightarrow
H^{2n-7}(\hat{S})\longrightarrow. \label{e44}
\end{aligned}
\end{equation}
The only appearance of $A_{n-3}$ in the polynomials defining $S$ is
in $G_{n-2}$, namely in $C_{n-3}$. After a linear change of the
variables we may assume that $C_{n-3}=A_{n-3}$ is independent.
Furthermore, the same argument as for $\hat{V}$ at \emph{step 1}
(see (\ref{e26}) and (\ref{e31})) gives us
\begin{equation}
    H^{2n-7}(\hat{S})\cong H^{2n-9}(S)(-1)
\end{equation}
with
\begin{equation}
    S:=\cV(I_{n-2},I_{n-3},G''_{n-2})\subset\PP^{2n-5}(\,\text{no}\;
B_{n-1}, A_{n-2}, B_{n-2}, A_{n-1})
\end{equation}
and $H^{2n-6}_{prim}(\hat{S})=0$. The sequence (\ref{e44})
simplifies to
\begin{equation}
\begin{aligned}
    0 \longrightarrow
H^{2n-7}_c(U_2)\longrightarrow
 H^{2n-7}(\cV(I_{n-2}&,A_{n-3}I_{n-3},G_{n-2})) \longrightarrow\\
&H^{2n-9}(S)(-1)\longrightarrow.
\end{aligned}\label{e47}
\end{equation}
Applying the functors $\gr^W_i$ for the sequence, by (\ref{e38.5}),
(\ref{e40}), (\ref{e42}) and (\ref{e47}), we get
\begin{equation}
\gr^W_i H^{2n-6}_{prim}(V) \cong \gr^W_i H^{2n-7}_c(U_2)
\quad\text{for}\;\; i=0,1. \label{e48}
\end{equation}
Now, the scheme $U_2$ is defined by the system
\begin{multline}\;\;
    \left\{
        \begin{aligned}
        I_{n-2}=G_{n-2}=0\\
        A_{n-3}I_{n-3}=0\\
        I_{n-3}\neq 0\\
        \end{aligned}
    \right.
   \Leftrightarrow \left\{
        \begin{aligned}
        G_{n-2}=I_{n-2}=0\\
        A_{n-3}=0\\
        I_{n-3}\neq 0\\
        \end{aligned}.
    \right. \;\;\;\;
    \label{e49}
\end{multline}
Eliminating $A_{n-3}$, which is zero on $U_2$, we get an isomorphism
\begin{equation}
    U_2\cong U_2'
    \label{e50}
\end{equation}
with $U_2'\subset\PP^{2n-5}$ (no $B_{n-1}$, $A_{n-2}$, $B_{n-2}$,
$A_{n-3}$) defined by the system
\begin{equation}
    \left\{
        \begin{aligned}
        I'_{n-2}=0\\
        G'_{n-2}=0\\
        I_{n-3}\neq 0,\!\!\\
        \end{aligned}
    \right.
    \label{e51}
\end{equation}
where primes mean that we set $A_{n-3}=0$ in the polynomials, namely
in $C_{n-3}$. Now we write
\begin{equation}
C'_{n-3}=A_{n-1}\pm A_{n-4}, \label{e53}
\end{equation}
with "+" only when $a_{n-4 n-4}=B_{n-4}$ in the matrix.

Such schemes $U_2$ were studied in the first chapter of Chapeter 1
(see (\ref{b44}) and (\ref{b46})). It follows that $U'_2$ is defined
by the system
\begin{equation}
    \left\{
        \begin{aligned}
        C'_{n-3}I_{n-3}-A_{n-4}^2I_{n-4}=0\\
        Li_{n-2}=0\qquad\\
        I_{n-3}\neq 0\qquad\\
        \end{aligned}
    \right.
    \label{e54}
\end{equation}
with
\begin{equation}
\begin{aligned}
    Li_{n-2}:=A_{n-1}I_{n-3}+\sum_{s}(-1)^{s+n-1} A_{v(s)}I_{n-2}(s,n-3)=\\
    A_{n-1}I_{n-3}+\sum_{s}(-1)^{s+n-1}
    A_{v(s)}I_s\prod_{k=s}^{n-4}A_k.
\end{aligned} \label{e55}
\end{equation}
The sum goes over all $s_i+j-1<n-3$, $i\not\equiv t\!\!\mod 2$,
$1\leq i\leq t$, $1\leq j\leq l_i$, so over all $s<n-3$ such that
$a_{s s}=C_s$. It is convenient to use the recurrence formula
\begin{equation}
    Li_{s+1}=\left\{
    \begin{aligned}
        A_{v(s)}I_{s}-A_{s-1}Li_{s},\;\; a_{s+1 s+1}=C_{s+1},\\
        -A_{s-1}Li_s,\;\; a_{s+1 s+1}=B_{s+1}.
    \end{aligned}
    \right.
\end{equation}
We can express $A_{n-1}$ from the second equation of the system
(\ref{e54}) and $C_2$ from the first one.
\begin{equation}
    \left\{
        \begin{aligned}
        A_{n-1}\pm A_{n-4}=C'_{n-3} = A_{n-4}^2I_{n-4}/I_{n-3}\\
        A_{n-1}= A_{n-4}Li_{n-3}/I_{n-3}\\
        I_{n-3}\neq 0.\!\!\\
        \end{aligned}
    \right.
    \label{e56}
\end{equation}
These two expressions for $A_{n-1}$ must be equal on $U_2'$. We
introduce the polynomials $Ni_s$ defined by
\begin{equation}
    A_{s-1}Ni_s=\pm A_{s-1}I_s+A_{s-1}^2I_{s-1}-A_{s-1}Li_s.
\end{equation}
Sometimes we can write $Ni_s^-$ and $Ni_s^+$ to indicate the sign
taken in $Ni_s$. The natural projection from the point where all the
variables but $A_{n-1}$ are zero induces an isomorphism
\begin{equation}
    U'_2 \cong
    U_3:=\cV(A_{n-4}Ni_{n-3})\backslash\cV(A_{n-4}Ni_{n-3},I_{n-3})
    \label{e58}
\end{equation}
with $U_3$ living in $\PP^{2n-6}$(no $B_{n-1}$, $A_{n-2}$,
$B_{n-2}$, $A_{n-3}$, $A_{n-1}$). By (\ref{e48}), (\ref{e50}) and
(\ref{e58}),
\begin{equation}
\gr^W_i H^{2n-6}_{prim}(V)\cong \gr^W_i
H^{2n-7}_c(U_3)\quad\text{for}\;\; i=0,1. \label{e59.5}
\end{equation}
We have two possibilities : $a_{n-4\:n-4}=C_{n-4}$ or
$a_{n-4\:n-4}=B_{n-4}$. When the letter holds, go to \emph{Step 4}
with $Ni_{n-3}=Ni_{n-3}^-$; do the next step with
$Ni_{n-3}=Ni_{n-3}^+$ otherwise.\medskip\\
\emph{\underline{Step 3.}} Suppose that the entry $a_{s s}$ of
$M_{GZZ}$ is $C_s$ and $a_{s+1 s+1}=C_{s+1}$. This means that
$n_i\leq s\leq n_i+l_i-2$ for some $i\not\equiv t\!\!\mod 2$. This
corresponds to the case $s-4$ if we had come from \emph{Step 2}.
One has
\begin{equation}
    C_s=A_v-A_s\pm A_{s-1}
    \label{e59}
\end{equation}
with "+" only when $a_{s-1\,s-1}=B_s$. We work in $\PP^N$(no $DV_s$)
for
\begin{equation}
    N=2n-1-2(n-1-s-1)-1=2s+2,
    \label{e60}
\end{equation}
and the Dropped Variables ($DV_s$) are all the variables in
$I_{n-1-s}^{s+1}$ but $A_s$. The thing to compute is $H^{2s+1}_c(U)$
for $U$ defined by
\begin{equation}
    U:=\cV(A_s Ni_{s+1})\backslash\cV(A_s Ni_{s+1},I_{s+1}),
    \label{e61}
\end{equation}
where
\begin{equation}
    Ni_{s+1}=I_{s+1}+A_sI_s-Li_{s+1}.
\end{equation}
Define $T,Y\subset\PP^{2s+2}$(no $DV_s$) by
\begin{equation}
\begin{aligned}
    T:&=\cV(A_s Ni_{s+1}),\\
    Y:&=\cV(A_s Ni_{s+1},I_{s+1}).
    \label{e63}
\end{aligned}
\end{equation}
One has an exact sequence
\begin{equation}
\rightarrow H^{2s}_{prim}(T) \rightarrow H^{2s}_{prim}(Y)\rightarrow
H^{2s+1}_c(U)\rightarrow \\
H^{2s+1}(T) \rightarrow. \label{e64}
\end{equation}
Using (\ref{e59}), we rewrite
\begin{equation}
\begin{aligned}
    Ni_{s+1}=&I_{s+1}+A_sI_s-Li_{s+1}=(A_v-A_s\pm A_{s-1})I_s-\\
    &A_{s-1}^2I_{s-1}+A_sI_s-A_vI_s+A_{s-1}Li_s=\\
    &-A_{s-1}(\pm I_s + A_{s-1}I_{s-1}-Li_s)=-A_{s-1}Ni_s.
    \label{e65}
\end{aligned}
\end{equation}
and see that $Ni_{s+1}$ is actually independent of $A_v$ and $A_s$.
This allows us to apply \emph{Theorem A} ($N=2s+2$, $k=1$, $t=1$) to
$T$ and get
\begin{equation}
    H^{i}_{prim}(T)=0,\quad i<2s+2.
\end{equation}
Thus, the sequence (\ref{e64}) implies an isomorphism
\begin{equation}
H^{2s+1}_c(U)\cong H^{2s}_{prim}(Y). \label{e67}
\end{equation}
Define the subvariety $\hat{Y}_1\subset Y$ by
\begin{equation}
\hat{Y}_1:=Y\cap\cV(I_s)=\cV(A_s Ni_{s+1},I_s,A_{s-1}I_{s-1}).
\end{equation}
The polynomial $Ni_{s-1}$ is independent of $A_v$ by (\ref{e65}).
Applying  \emph{Theorem A} ($N=2s+2$, $k=3$, $t=1$) to $\hat{Y}_1$,
we come to an exact sequence
\begin{equation}
    0\rightarrow
H^{2s}_c(Y\backslash\hat{Y}_1)\rightarrow \\
H^{2s}_{prim}(Y) \rightarrow
H^{2s-2}_{prim}(Y_1)(-1)\rightarrow\label{e69}
\end{equation}
with $Y_1\subset\PP^{2s+1}$(no $DV_s$, $A_v$) defined by the same
polynomials. The scheme $Y\backslash\hat{Y}_1$ is defined by the
system
\begin{equation}
    \left\{
        \begin{aligned}
        A_sA_{s-1}Ni_s=0\\
        C_sI_s-A_{s-1}^2I_{s-1}=0\\
        I_s\neq 0.\!\!\\
        \end{aligned}
    \right.
    \label{e70}
\end{equation}
By (\ref{e59}), we express $A_v$ from the second equation.
Projecting from the point where all the variables but $A_v$ are
zero, we get isomorphisms
\begin{equation}
    Y\backslash\hat{Y}_1 \cong R \qquad\text{and}\qquad
    H^{2s}_c(Y\backslash\hat{Y}_1)\cong H^{2s}_c(R),
    \label{e71}
\end{equation}
where $R\subset\PP^{2s+1}$(no $DV_s$, $A_v$) is given by the system
\begin{equation}
\left\{
        \begin{aligned}
        A_sA_{s-1}Ni_s&=0\\
        I_s&\neq 0.
        \end{aligned}
    \right.
    \label{e72}
\end{equation}
Define $R_1,R_2\subset R$  by
\begin{equation}
    \left\{
        \begin{aligned}
       A_{s-1}Ni_s =0\\
       I_s \neq 0
        \end{aligned}
    \right.
   \qquad\text{and}\qquad \left\{
        \begin{aligned}
        A_s &=0\\
        I_s&\neq 0.
        \end{aligned}\;\;\;\;
    \right.
    \label{e73}
\end{equation}
One has the Mayer-Vietoris sequence
\begin{multline}
    \longrightarrow H^{2s-1}_c(R_1)\oplus H^{2s-1}_c(R_2)\longrightarrow H^{2s-1}_c(R_3)\longrightarrow\\
     H^{2s}_c(R)\longrightarrow H^{2s}_c(R_1)\oplus H^{2s}_c(R_2)
    \longrightarrow
    \label{e74}
\end{multline}
with $R_3:=R_1\cap R_2$. The defining polynomials of $R_1$ and $R_2$
are independent of $A_s$ and $A_m$ respectively. Applying
\emph{Theorem B} ($N=2s+1$, $k=1$, $t=1$) to them, we get
\begin{equation}
    H^i_c(R_1)=H^i_c(R_2)=0
\end{equation}
for $i<2s+1$. The sequence (\ref{e74}) implies an isomorphism
\begin{equation}
     H^{2s}_c(R)\cong H^{2s-1}_c(R_3).
     \label{e76}
\end{equation}
Now, $R_3\subset\PP^{2s+1}$(no $DV_s$, $A_v$) is defined by the
system
\begin{equation}
    \left\{
        \begin{aligned}
       A_s=A_{s-1}Ni_s =0\\
       I_s \neq 0.\!\!
        \end{aligned}
    \right.
\end{equation}
Projecting from the point where all the variables but $A_s$ are
zero, we get isomorphisms
\begin{equation}
    R_3\cong U'\qquad \text{and} \qquad H^{2s-1}_c(R_3)\cong
    H^{2s-1}_c(U')
    \label{e78}
\end{equation}
for $U'\subset\PP^{2s}$(no $DV_s$, $A_v$, $A_s$) defined by
\begin{equation}
    U'=\cV(A_{s-1}Ni_s)\backslash\cV(A_{s-1}Ni_s,I_s).
    \label{e79}
\end{equation}
Collecting (\ref{e67}), (\ref{e69}), (\ref{e71}) (\ref{e76}) and
(\ref{e78}) together, we obtain an exact sequence
\begin{equation}
    0\rightarrow
H^{2s-1}_c(U')\rightarrow \\
H^{2s+1}_c(U) \rightarrow
H^{2s-2}_{prim}(Y_1)(-1)\rightarrow,\label{e80}
\end{equation}
where $U$ is defined by (\ref{e61}). Applying $\gr^W_i$, one gets
\begin{equation}
    \gr^W_i H^{2n+1}_c(U)\cong \gr^W_i H^{2n+1}_c(U'),\quad i=0,1.
    \label{e81}
\end{equation}
If $s=1$, go to \emph{the Last Step}.
\medskip\\
When we come to \emph{Step 3} with some $s$, $n_i\leq s\leq
n_i+l_i-2$, $i\not\equiv t\!\!\mod 2$, we must apply this step
$s-n_i-1$ times with $Ni_s=Ni_s^+$ and then one more time with
$Ni_s=Ni_s^-$.
After this, we are in a new situation.\medskip\\
%
%
%
%
\emph{\underline{Step 4.}} Suppose that the entry $a_{s s}$ of
$M_{GZZ}$ is $B_s$ and $a_{s+\!1\,s+\!1}=C_{s+1}$. This means that
$s = n_i-1$ for some $i\not\equiv t\!\!\mod 2$. Denote by $DV_s$ the
dropped variables that are all the variables appearing in
$I_{n-1-s}^{s+1}$, but $A_s$. Again, we have to compute
$H^{2s+1}_c(U)$ for $U\subset\PP^{2s+2}$ defined by
\begin{equation}
    U:=\cV(A_s Ni_{s+1})\backslash\cV(A_s Ni_{s+1},I_{s+1}),
    \label{e87}
\end{equation}
where
\begin{equation}
    Ni_{s+1}=-I_{s+1}+A_sI_s-Li_{s+1}.
\end{equation}
Define closed subschemes $U_1,U_2\subset U$ by
\begin{equation}
\begin{aligned}
    U_1:&=\cV(A_s)\backslash\cV(A_s,I_{s+1}),\\
    U_2:&=\cV(Ni_{s+1})\backslash\cV(Ni_{s+1},I_{s+1}).
    \label{e89}
\end{aligned}
\end{equation}
This covering gives us an exact sequence
\begin{equation}
\begin{aligned}
\longrightarrow H^{2s}_c(U_1)\oplus H^{2s}_c(U_1)&\longrightarrow
H^{2s}_c(U_3)\longrightarrow\\
&H^{2s+1}_c(U)\longrightarrow  H^{2s+1}_c(U_1)\oplus H^{2s+1}_c(U_1)
\longrightarrow, \label{e90}
\end{aligned}
\end{equation}
where $U_3:=U_1\cap U_2$. The polynomials in the definition of $U_1$
do not depend on $A_m$. \emph{Theorem B} ($N=2s+2$, $k=1$, $t=1$)
implies
\begin{equation}
    H^i_c(U_1)=0\quad\text{for}\;\; i< 2s+2.
    \label{e91}
\end{equation}
We rewrite
\begin{equation}
\begin{aligned}
    Ni_{s+1}=&-I_{s+1}+A_sI_s-Li_{s+1}=-B_sI_s+A_{s-1}^2I_{s-1}+A_sI_s+\\
    &A_{s-1}Li_s=(A_s-B_s)I_s+A_{s-1}^2I_{s-1}+A_{s-1}Li_s
    \label{e92}
\end{aligned}
\end{equation}
and see that $Ni_{s+1}$ depends neither on $B_s$ nor on $A_s$ but on
the difference $A_s-B_s$. After the change of variables
$B_s:=A_s-B_s$, the polynomial $Ni_{s+1}$ becomes independent of
$A_s$. Applying \emph{Theorem B} ($N=2s+2$, $k=1$, $t=1$), we get
\begin{equation}
    H^i_c(U_2)=0\quad\text{for}\;\; i< 2s+2.
\end{equation}
Together with (\ref{e91}), the sequence (\ref{e90}) gives an
isomorphism
\begin{equation}
 H^{2s+1}_c(U)\cong H^{2s}_c(U_3). \label{e94}
\end{equation}
Now, $U_3\subset\PP^{2s+2}$(no $DV_s$) is given by the system
\begin{multline}\;\;
    \left\{
        \begin{aligned}
        A_s=0\\
        Ni_{s+1}=0\\
        I_{s+1}\neq 0\\
        \end{aligned}
    \right.
   \Leftrightarrow \left\{
        \begin{aligned}
        A_s=0\\
        I_{s+1}+Li_{s+1}=0\\
        I_{s+1}\neq 0.\!\!\\
        \end{aligned}
    \right. \;\;\;\;
    \label{e95}
\end{multline}
We eliminate the variable $A_s$ and consider the open
$U_4\subset\PP^{2s+1}$(no $DV_s$, $A_s$) defined by the last two
conditions, then
\begin{equation}
    H^{2s}_c(U_3)\cong H^{2s}_c(U_4).
    \label{e96}
\end{equation}
Define $T,Y\subset\PP^{2s+1}$(no $DV_s$, $A_s$) by
\begin{equation}
\begin{aligned}
    T:&=\cV(I_{s+1}+Li_{s+1}),\\
    Y:&=\cV(I_{s+1}+Li_{s+1},I_{s+1}).
    \label{e97}
\end{aligned}
\end{equation}
We can write an exact sequence
\begin{equation}
\longrightarrow H^{2s-1}(T)\longrightarrow
H^{2s-1}(Y)\longrightarrow H^{2s}_c(U_4)\longrightarrow
H^{2s}_{prim}(T)\longrightarrow. \label{e98}
\end{equation}
\emph{Theorem A} ($N=2s+1$, $k=1$, $t=0$) gives us the vanishing of
the term to the left. Similar to (\ref{e92}), one has
\begin{equation}
    I_{s+1}+Li_{s+1}=B_sI_s-A_{s-1}^2I_{s-1}-A_{s-1}Li_s.
    \label{e99}
\end{equation}
Let $\hat{T}_1$ be a subvariety of $T\subset\PP^{2s+1}$(no $DV_s$,
$A_s$) defined by
\begin{equation}
    \hat{T}_1:=T\cap\cV(I_s)=\cV(I_s,A_{s-1}^2I_{s-1}+A_{s-1}Li_s).
\end{equation}
We write an exact sequence
\begin{equation}
\longrightarrow
    H^{2s}_c(T\backslash\hat{T}_1)\longrightarrow
    H^{2s}_{prim}(T)\longrightarrow H^{2s}_{prim}(\hat{T}_1)\longrightarrow.
    \label{e100}
\end{equation}
The defining polynomials of $\hat{T}_1$ are independent of $B_s$. We
apply \emph{Theorem A} ($N=2s+1$, $k=2$, $t=1$) and get
\begin{equation}
    H^{2s}_{prim}(\hat{T}_1)\cong H^{2s-2}_{prim}(T_1)(-1).
    \label{e101}
\end{equation}
On $T\backslash\hat{T}_1$ we can express $B_s$ (see (\ref{e99})) and
get an isomorphism
\begin{equation}
    T\backslash \hat{T}_1 \cong \PP^{2s}\backslash\cV(I_s)
    \label{e102}
\end{equation}
with $\PP^{2s}$(no $DV_s$, $A_s$, $B_s$). The polynomial $I_s$ does
not depend on $A_m$, and \emph{Theorem B} ($N=2s$, $k=0$, $t=1$)
yields
\begin{equation}
    H^i_c(\PP^{2s}\backslash\cV(I_s))=0\quad \text{for}\; i< 2s+1.
    \label{e103}
\end{equation}
By (\ref{e101}) and (\ref{e103}), the sequence (\ref{e100})
simplifies to
\begin{equation}
   0\longrightarrow H^{2s}_{prim}(T)\longrightarrow
   H^{2s-2}_{prim}(\hat{T}_1)(-1)\longrightarrow.
   \label{e104}
\end{equation}
Applying $\gr^W_i$, we get
\begin{equation}
    \gr^W_0 H^{2s}(T)=\gr^W_1 H^{2s}(T)=0.
    \label{e105}
\end{equation}
We return to the variety $Y$ which is defined by
\begin{equation}
\begin{aligned}
    Y:=\cV(I_{s+1}+Li_{s+1}, I_{s+1})=\cV(Li_{s+1}, I_{s+1})=\\
    \cV(A_{s-1}Li_s,B_sI_s-A_{s-1}^2I_{s-1}).
    \label{e107}
\end{aligned}
\end{equation}
One can write an exact sequence
\begin{equation}
    \rightarrow
    H^{2s-2}_{prim}(\hat{Y}_1) \rightarrow H^{2s-1}_c(Y\backslash\hat{Y}_1)\rightarrow H^{2s-1}(Y)\rightarrow
    H^{2s-1}(\hat{Y}_1)\rightarrow,
    \label{e108}
\end{equation}
where $\hat{Y}_1$ is the subvariety of $Y\subset\PP^{2s+1}$(no
$DV_s$, $A_s$) defined by
\begin{equation}
    \hat{Y}_1:=Y\cap\cV(I_s)=\cV(I_s,A_{s-1}Li_{s-1},A_{s-1}I_{s-1}).
    \label{e109}
\end{equation}
The last three polynomials are independent of $B_s$; applying
\emph{Theorem A} ($N=2s+1$, $k=3$, $t=1$), we get
\begin{equation}
    H^{2s-1}(\hat{Y}_1)\cong H^{2s-3}(Y_1)(-1)
    \label{e110}
\end{equation}
and $H^{2s-2}(\hat{Y}_1)=0$. This implies that the sequence
(\ref{e108}) simplifies to
\begin{equation}
   0 \rightarrow H^{2s-1}_c(Y\backslash\hat{Y}_1)\rightarrow H^{2s-1}(Y)\rightarrow
    H^{2s-3}(Y_1)(-1)\rightarrow,
    \label{e111}
\end{equation}
Now, $Y\backslash\hat{Y}_1\subset\PP^{2s+1}$(no $DV_s$, $A_s$) is
defined by the system
\begin{equation}
    \left\{
        \begin{aligned}
        A_{s-1}Li_s=0\\
        B_sI_s-A_{s-1}^2I_{s-1}=0\\
        I_s\neq 0.\!\!
        \end{aligned}
    \right.
    \label{e112}
\end{equation}
We can express $B_s$ from the second equation and, projecting from
the point where all variables but $B_s$ are zero, we get an
isomorphism
\begin{equation}
    Y\backslash\hat{Y}_1 \cong U',
    \label{e113}
\end{equation}
where $U'\subset\PP^{2s}$(no $DV_s$, $A_s$, $B_s$) is defined by the
system
\begin{equation}
    \left\{
        \begin{aligned}
        A_{s-1}Li_s&=0\\
        I_s&\neq 0.
        \end{aligned}
    \right.
    \label{e114}
\end{equation}
Finally, combining (\ref{e94}), (\ref{e96}),
(\ref{e98}),(\ref{e105}),
 (\ref{e111}) and (\ref{e113}), we get
\begin{equation}
    \gr^W_i H^{2s+1}_c(U)\cong \gr^W_i H^{2s-1}(Y)\cong
    \gr^W_i H^{2s-1}_c(U'),\;\;i=0,1,
\end{equation}
for $U$ and $U'$ defined in (\ref{e87}) and (\ref{e114})
respectively.
\medskip\\
Now, if $s=1$, go to \emph{the Last Step}. If $a_{s-1\,
s-1}=C_{s-1}$, we go to \emph{Step 6}. Otherwise do the next step.
\medskip\\
\emph{\underline{Step 5.}} Consider an entry $a_{s s}=B_s$ of
$M_{GZZ}$ such that $a_{s+1 s+1}=B_{s+1}$. With other words, $s$
satisfies the condition $n_i\leq s\leq n_i+l_{i+1}-2$ for some
$i\equiv t\!\!\mod 2$. Let $U\subset\PP^{2s+2}$(no $DV_s$) be
defined by
\begin{equation}
    U:=\cV(A_s Li_{s+1})\backslash\cV(A_s Li_{s+1}, I_{s+1}),
    \label{e117}
\end{equation}
and denote by $DV_s$ all the variables appearing in
$I_{n-1-s}^{s+1}$. As usual, we try to compute $H^{2s+1}_c(U)$.
Define $U_1,U_2\subset U$ by
\begin{equation}
\begin{aligned}
    U_1:&=\cV(A_s)\backslash\cV(A_s,I_{s+1}),\\
    U_2:&=\cV(Li_{s+1})\backslash\cV(Li_{s+1},I_{s+1}).
    \label{e118}
\end{aligned}
\end{equation}
One can write an exact sequence
\begin{equation}
\begin{aligned}
\longrightarrow H^{2s}_c(U_1)\oplus H^{2s}_c&(U_2)\longrightarrow
H^{2s}_c(U_3)\longrightarrow\\
&H^{2s+1}_c(U)\longrightarrow  H^{2s+1}_c(U_1)\oplus H^{2s+1}_c(U_2)
\longrightarrow \label{e119}
\end{aligned}
\end{equation}
where $U_3:=U_1\cap U_2$. The defining polynomials of $U_1$ do not
depend on $A_m$, thus \emph{Theorem B} ($N=2s+2$, $k=1$, $t=1$)
implies
\begin{equation}
    H^i_c(U_1)=0\quad\text{for}\;\; i< 2s+2.
\end{equation}
Since
\begin{equation}
    Li_{s+1}=-A_{s-1}Li_s
    \label{e121}
\end{equation}
and $I_{s+1}$ are independent of $A_s$, we apply \emph{Theorem B}
($N=2s+2$, $k=1$, $t=1$) to $U_2$ and get
\begin{equation}
    H^i_c(U_2)=0\quad\text{for}\;\; i\leq 2s+1.
\end{equation}
Thus, (\ref{e119}) gives us the isomorphism
\begin{equation}
    H^{2s+1}_c(U)\cong H^{2s}_c(U_3).
    \label{e123}
\end{equation}
We can eliminate $A_s$, which is zero along $U_3$, and get an
isomorphism
\begin{equation}
    U_3\cong U_4:=\cV(Li_{s+1})\backslash\cV(Li_{s+1},I_{s+1})
    \label{e125}
\end{equation}
with $U_4\subset\PP^{2s+1}$(no $DV_s$, $A_s$). Defining
$T,Y\subset\PP^{2s+1}$ by
\begin{equation}
\begin{aligned}
    T:&=\cV(Li_{s+1}),\\
    Y:&=\cV(Li_{s+1},I_{s+1}),
    \label{e127}
\end{aligned}
\end{equation}
we get an exact sequence
\begin{equation}
     \rightarrow
    H^{2s-1}(T)\rightarrow H^{2s-1}(Y)\rightarrow
    H^{2s}_c(U_4)\rightarrow
    H^{2s}_{prim}(T)\rightarrow.
    \label{e128}
\end{equation}
By (\ref{e121}), $Li_s$ is independent of $B_s$, thus \emph{Theorem
A} ($N=2s+1$, $k=1$, $t=1$) yields
\begin{equation}
    H^{i}_{prim}(T)=0\quad\text{for}\;\; i< 2s+1.
\end{equation}
The sequence (\ref{e128}) implies an isomorphism
\begin{equation}
    H^{2s}_c(U_4)\cong H^{2s-1}(Y).
    \label{e130}
\end{equation}
Now, let
\begin{equation}
    \hat{Y_1}:=Y\cap\cV(I_s)=\cV(Li_s,I_{s+1},I_s)
\end{equation}
be a subvariety of $Y\subset\PP^{2s+1}$(no $DV_s$, $A_s$). One has
an exact sequence
\begin{equation}
    \rightarrow
    H^{2s-2}_{prim}(\hat{Y}_1) \rightarrow H^{2s-1}_c(Y\backslash\hat{Y}_1)\rightarrow H^{2s-1}(Y)\rightarrow
    H^{2s-1}(\hat{Y}_1)\rightarrow.
    \label{e132}
\end{equation}
Since
\begin{equation}
    \hat{Y_1}=\cV(Li_{s+1},
    B_sI_s-A_{s-1}^2I_{s-1},I_s)=\cV(A_{s-1}Li_s,I_s,A_{s-1}I_{s-1}),
\end{equation}
the defining polynomials forget $B_s$; by \emph{Theorem A}
($N=2s+1$, $k=3$, $t=1$), the sequence (\ref{e132}) simplifies to
\begin{equation}
    0 \rightarrow H^{2s-1}_c(Y\backslash\hat{Y}_1)\rightarrow H^{2s-1}(Y)\rightarrow
    H^{2s-3}(Y_1)(-1)\rightarrow,
    \label{e134}
\end{equation}
where $Y_1\subset\PP^{2s}$(no $DV_s$, $A_s$, $B_s$) is defined by
\begin{equation}
    Y_1:=\cV(A_{s-1}Li_{s-1},I_s,A_{s-1}I_{s-1}).
\end{equation}
The open subscheme $Y\backslash\hat{Y}_1\subset Y$ is given by the
system
\begin{equation}
    \left\{
        \begin{aligned}
        A_{s-1}Li_{s-1}&=0\\
        B_sI_s-A_{s-1}^2I_{s-1}&=0\\
        I_s&\neq 0.
        \end{aligned}
    \right.
    \label{e136}
\end{equation}
Expressing  $B_s$ from the second equation and projecting from the
point where all the variables but $B_s$ are zero, we get an
isomorphism
\begin{equation}
    Y\backslash\hat{Y}_1 \cong U',
    \label{e137}
\end{equation}
where $U'\subset\PP^{2s}$(no $DV_s$, $A_s$, $B_s$) defined by
\begin{equation}
    U':=\cV(A_{s-1}Li_s)\backslash\cV(A_{s-1}Li_s,I_s).
    \label{e138}
\end{equation}
Collecting (\ref{e123}),(\ref{e125}),(\ref{e134}) and (\ref{e137})
together, we get an exact sequence
\begin{equation}
    0\longrightarrow H^{2s-1}_c(U')\longrightarrow H^{2s+1}_c(U)\longrightarrow
    H^{2s-3}(Y_1)(-1)\longrightarrow
    \label{e139}
\end{equation}
for $U$ and $U'$ defined by (\ref{e117}) and (\ref{e138})
respectively. Consequently, we obtain
\begin{equation}
\begin{aligned}
    \gr^W_i H^{2n+1}_c(U)\cong \gr^W_i H^{2n-1}_c(U'),\quad i=0,1.
    \label{e140}
\end{aligned}
\end{equation}
If $s=1$, go to \emph{the Last Step}.
\medskip\\
After repeating a suitable number of times \emph{Step 5}, we come to
the following situation.
\medskip\\
\emph{\underline{Step 6.}} Suppose that the entry $a_{s s}$ of the
matrix $M_{GZZ}$ is $C_s$ and $a_{s+1 s+1}=B_{s+1}$. This happens
when $s_i-1$ for some $i\equiv t\!\!\mod 2$. For $C_s$ we have
\begin{equation}
    C_s=A_v + A_s \pm A_{s-1}
    \label{e141}
\end{equation}
with "+" only when $l_i=1$. Let $U\subset\PP^{2s+2}$(no $DV_s$) be
defined by
\begin{equation}
    U:=\cV(A_s Li_{s+1})\backslash\cV(A_s Li_{s+1}, I_{s+1}),
    \label{e142}
\end{equation}
and denote by $DV_s$ all the variables appearing in
$I_{n-1-s}^{s+1}$. As in the previous case, we define
$U_1,U_2\subset U$ to be
\begin{equation}
\begin{aligned}
    U_1:&=\cV(A_s)\backslash\cV(A_s,I_{s+1}),\\
    U_2:&=\cV(Li_{s+1})\backslash\cV(Li_{s+1},I_{s+1}).
    \label{e143}
\end{aligned}
\end{equation}
and write an exact sequence
\begin{equation}
\begin{aligned}
\longrightarrow H^{2s}_c(U_1)&\oplus H^{2s}_c(U_2)\longrightarrow
H^{2s}_c(U_3)\longrightarrow\\
&H^{2s+1}_c(U)\longrightarrow  H^{2s+1}_c(U_1)\oplus H^{2s+1}_c(U_2)
\longrightarrow \label{e144}
\end{aligned}
\end{equation}
where $U_3:=U_1\cap U_2$. For this step we have
\begin{equation}
\begin{aligned}
    Li_{s+1}=A_vI_s-A_{s-1}Li_s,\\
    I_{s+1}=C_sI_{s}-A_{s-1}^2I_{s-1}.
    \label{e145}
\end{aligned}
\end{equation}
Noting that the polynomials defining $U_1$ and $U_2$ are independent
of $A_m$ and $A_s$ respectively, we apply \emph{Theorem B}
($N=2s+2$, $k=1$, $t=1$) and get
\begin{equation}
    H^i_c(U_1)=H^i_c(U_2)=0\quad\text{for}\;\; i< 2s+2.
\end{equation}
The sequence (\ref{e144}) implies an isomorphism
\begin{equation}
    H^{2s+1}_c(U)\cong H^{2s}_c(U_3).
    \label{e147}
\end{equation}
Eliminating $A_s$, which is zero on $U_3$, we get an isomorphism
\begin{equation}
    U_3\cong U_4:=\cV(Li_{s+1})\backslash\cV(Li_{s+1},I_{s+1})
    \label{e148}
\end{equation}
with $U_4\subset\PP^{2s+1}$(no $DV_s$, $A_s$). Thus,
\begin{equation}
    H^{2s}_c(U_3)\cong H^{2s}_c(U_4).
    \label{e149}
\end{equation}
Denoting by $I_{s+1}'$ the polynomial $I_{s+1}$ after setting
$A_s=0$, we define\newline $T,Y\subset\PP^{2s+1}$ (no $DV_s$, $A_s$)
by
\begin{equation}
\begin{aligned}
    T:&=\cV(Li_{s+1}),\\
    Y:&=\cV(Li_{s+1},I_{s+1}').
    \label{e150}
\end{aligned}
\end{equation}
One gets an exact sequence
\begin{equation}
     \rightarrow
    H^{2s-1}(T)\rightarrow H^{2s-1}(Y)\rightarrow
    H^{2s}_c(U_4)\rightarrow
    H^{2s}_{prim}(T)\rightarrow.
    \label{e151}
\end{equation}
\emph{Theorem A} ($N=2s+1$, $k=1$, $t=0$) implies the vanishing of
the term to the left. Motivated by (\ref{e145}), we define
\begin{equation}
    \hat{T}_1:=T\cap\cV(I_s)=\cV(I_s,A_{s-1}Li_s)
    \subset\PP^{2s+1}(\text{no}\, DV_s,\, A_s).
\end{equation}
One can write an exact sequence
\begin{equation}
    \longrightarrow H^{2s}_c(T\backslash\hat{T}_1)\longrightarrow H^{2s}_{prim}(T)\longrightarrow
    H^{2s}_{prim}(\hat{T}_1)\longrightarrow
    \label{e154}
\end{equation}
On $T\backslash\hat{T}_1$ we can express $A_v$ from the equation
$Li_s=0$. Projecting from the point where all the variables but
$A_v$ are zero, we get an isomorphism
\begin{equation}
    T\backslash\hat{T}_1 \cong \PP^{2s}\backslash\cV(I_{s-1}).
    \label{e155}
\end{equation}
The polynomial $I_{s-1}$ does not depend on $A_m$. Applying
\emph{Theorem B} ($N=2s$, $k=0$, $t=1$), one gets
\begin{equation}
    H^{2s}_c(T\backslash\hat{T}_1)=0.
    \label{e156}
\end{equation}
The polynomials defining $\hat{T}_1$ are independent of $A_v$.
Applying \emph{Theorem A} ($N=2s+1$, $k=2$, $t=1$) to $\hat{T}_1$,
one gets
\begin{equation}
    H^{2s}_{prim}(T)\cong H^{2s-2}_{prim}(T_1)(-1),
    \label{e157}
\end{equation}
where $T_1\subset\PP^{2s}$(no $DV_s$, $A_s$, $A_v$) is defined by
the same equations as $\hat{T}_1$. By (\ref{e156}) and (\ref{e157}),
the sequence (\ref{e154}) gives us
\begin{equation}
    \gr^W_0 H^{2s}_{prim}(T)=\gr^W_1 H^{2s}_{prim}(T)=0.
    \label{e158}
\end{equation}
Now define $\hat{Y}_1\subset Y\subset\PP^{2s+1}$(no $DV_s$, $A_s$)
by
\begin{equation}
    \hat{Y}_1:=\cV(Li_{s+1},I_{s+1}',I_s)=\cV(A_{s-1}Li_s,A_{s-1}I_{s-1},I_s).
\end{equation}
One can write a exact sequence
\begin{equation}
     \rightarrow H^{2s-2}(Y)_{prim}\rightarrow
    H^{2s-1}_c(Y\backslash \hat{Y}_1)\rightarrow
    H^{2s-1}(Y)\rightarrow
    H^{2s-1}(\hat{Y}_1)\rightarrow
    \label{e160}
\end{equation}
The polynomials defining $\hat{Y}_1$ do not depend on $A_v$. After
application of \emph{Theorem A} ($N=2s$, $k=3$, $t=1$), the sequence
(\ref{e160}) simplifies to
\begin{equation}
   0\rightarrow
    H^{2s-1}_c(Y\backslash \hat{Y}_1)\rightarrow
    H^{2s-1}(Y)\rightarrow
    H^{2s-3}(\hat{Y}_1)(-1)\rightarrow,
    \label{e161}
\end{equation}
where $Y_1\subset\PP^{2s}$(no $DV_s$, $A_s$, $A_v$) is defined by
the same equations. It follows that
\begin{equation}
    \gr^W_i H^{2s-1}(Y)\cong\gr^W_i H^{2s-1}_c(Y\backslash
    \hat{Y}_1),\;\;i=0,1.
    \label{e161.5}
\end{equation}
The open subscheme $Y\backslash\hat{Y}_1\subset Y$ is defined by the
system
\begin{multline}
   \;\;\; \left\{
        \begin{aligned}
        Li_{s+1}=0\\
        I_{s+1}'=0\\
        I_s\neq 0\\
        \end{aligned}
    \right.
   \Leftrightarrow \left\{
        \begin{aligned}
        A_vI_s-A_{s-1}Li_s=0\\
        (A_v\pm A_{s-1})I_s-A_{s-1}^2I_{s-1}=0\\
        I_s\neq 0.\!\!\\
        \end{aligned}
    \right.
    \label{e162}
\end{multline}
We can express $A_v$ from the first and second equation and this
expressions must be equal. So we define $Ni_s$ by
\begin{equation}
    Ni_s:=\pm I_s+A_{s-1}I_{s-1}-Li_s
    \label{e163}
\end{equation}
with "$-$" only when $l_i=1$. The expression for $A_v$ and the
natural projection from the point where all the variables but $A_v$
are zero yield an isomorphism
\begin{equation}
    Y\backslash \hat{Y}_1 \cong U',
    \label{e164}
\end{equation}
where $U'\subset\PP^{2s}$(no $DV_s$, $A_s$, $A_v$) is defined by
\begin{equation}
    U':=\cV(A_{s-1}Ni_s)\backslash\cV(A_{s-1}Ni_s,I_s).
    \label{e165}
\end{equation}
By (\ref{e147}), (\ref{e149}), and (\ref{e151}), one has an exact
sequence
\begin{equation}
    0\longrightarrow H^{2s-1}(Y)\longrightarrow
    H^{2s+1}_c(U)\longrightarrow
    H^{2s}_{prim}(T)\longrightarrow.
    \label{e166}
\end{equation}
Hence, (\ref{e158}), (\ref{e161.5}) and (\ref{e164}) imply
\begin{equation}
    \gr^W_i H^{2s+1}_c(U)\cong\gr^W_i H^{2s-1}_c(U'),\;\;\;i=0,1,
    \label{e167}
\end{equation}
\begin{equation}
\end{equation}
with $U$ and $U'$ defined by (\ref{e142}) and (\ref{e165})
respectively.
\medskip\\
If $s=1$, go to \emph{the Last Step}. If $a_{s-1\, s-1}=B_{s-1}$,
return to \emph{Step 4} with $Ni_s=Ni_s^-$; return to \emph{Step 3}
with $Ni_s=Ni_s^+$ otherwise.
\medskip\\
\emph{\underline{the Last Step.}} Recall that $l_1>1$. In the case
$t\equiv 0\!\!\mod 2$ we have come from \emph{Step 4} or \emph{Step
5}. The matrix looks like
\begin{equation}
M_{GZZ}=\left(
  \begin{array}{cccc}
   B_0 &  A_0 &\vdots& A_m\\
   A_0 &  B_1 &\vdots&  0 \\
 \ldots&\ldots&\ddots&\vdots\\
   A_m &   0  &\ldots&\ddots\\
  \end{array}
\right).
\end{equation}
We are interested $H^1(U)$, where $U\subset\PP^2(A_0:A_m:B_0)$ is
defined by
\begin{equation}
    U:=\cV(A_0Li_1)\backslash\cV(A_0Li_1,I_1)=\cV(A_0A_m)\backslash\cV(A_0A_m,B_0).
\end{equation}
The exact sequence
\begin{equation}
\begin{aligned}
\longrightarrow H^0_{prim}(\cV(A_0A_m)) \longrightarrow
H^0_{prim}(&\cV(A_0A_m,B_0))\longrightarrow\\
&H^1_c(U)\longrightarrow H^1(\cV(A_0A_m))\longrightarrow
\end{aligned}
\end{equation}
implies
\begin{equation}
    H^1_c(U)\cong\QQ(0).
\end{equation}
In the opposite case, when $t\not\equiv 0\!\!\mod 2$, the matrix
looks like
\begin{equation}
M_{GZZ}=\left(
  \begin{array}{cccc}
   B_0 &  A_0 &\vdots& A_m\\
   A_0 &  C_1 &\vdots&A_{m-1} \\
 \ldots&\ldots&\ddots&\vdots\\
   A_m &A_{m-1}&\ldots&\ddots\\
  \end{array}
\right),
\end{equation}
and we had come from \emph{Step 3} or \emph{Step 6}. We deal with
$U\subset\PP^2(A_0:A_m:B_0)$ defined by
\begin{equation}
\begin{aligned}
    U:=\cV(&A_0Ni_1)\backslash\cV(A_0Ni_1,I_1)=\\
    &\cV(A_0(B_0+A_0-A_m))\backslash\cV(A_0(B_0+A_0-A_m),B_0).
\end{aligned}
\end{equation}
Changing the variables $A_m:=B_0+A_0-A_m$, we come to the situation
above, and we again obtain
\begin{equation}
    H^1_c(U)\cong\QQ(0).
\end{equation}
We have constructed a sequence of schemes $U=U^0$, $U^1$, \ldots,
$U^{n-4}=U_3$ (see (\ref{e58})) such that
\begin{equation}
    \gr^W_i H^{2s+1}_c(U^s)\cong\gr^W_i H^{2s-1}_c(U^{s-1}),\;\;\;i=0,1,
\end{equation}
for $0\leq s\leq n-3$, $U^s\subset\PP^{2s+2}$. By (\ref{e59.5}), we
obtain
\begin{equation}
\begin{aligned}
    \gr^W_i H^{2n-6}_{prim}&(V)\cong\gr^W_i H^{2n-7}_c(U_3)\cong\\
    &\gr^W_i H^{2n-7}_c(U^{n-4})\cong\ldots\cong\gr^W_i
    H^{1}_c(U^0),\quad i=0,1.
\end{aligned}
\end{equation}
Hence,
\begin{equation}
    \gr^W_0 H^{2n-6}_{prim}(V)\cong\QQ(0)\quad\text{and}\quad
    \gr^W_1
    H^{2n-6}_{prim}(V)=0.
\end{equation}
Using the isomorphism
\begin{equation}
    H^{2n-2}(X)\cong H^{2n-6}(V)(-2)
\end{equation}
(see (e34)), we finally get
\begin{equation}
\begin{aligned}
    \gr^W_4 H^{2n-2}_{prim}&(X)\cong W_4\, H^{2n-2}_{prim}(X)\cong
    \QQ(-2),\\
    &\gr^W_5 H^{2n-2}_{prim}(X)=0.
\end{aligned}
\end{equation}
\end{proof}

\newpage
\section{De Rham class for GZZ(n,2)}

Fix some $n\geq 2$ and define $\Gamma=\Gamma_n:=GZZ(n,2)$. This
graph has $2n+6$ edges and $h_1(\Gamma)=2(n+3)$. Let
$X_n\subset\PP^{2n+5}$ be the graph hypersurface associated to
$\Gamma_n$. By the results of the previous section, one has an
inclusion
\begin{equation}
\QQ(-2)\hookrightarrow H^{2n+4}_{prim}(X_n)\cong
H^{2n+5}_c(\PP^{2n+5}\backslash X).
\end{equation}
Hence, we get $\dim H^{2n+5}_{DR}(\PP^{2n+5}\backslash X_n)\neq 0$.
We do not know that this cohomology group is one-dimensional in
general. Nevertheless, according to Example \ref{E2.2.4},
$\Gamma_2:=GZZ(2,2)\cong ZZ_5$, thus
$H^{2n+5}_{DR}(\PP^{2n+5}\backslash X_n)\cong K$ for $n=2$. In this
chapter we consider
\begin{equation}
    \eta=\eta_{\Gamma}=
    \frac{\Omega_{2n+5}}{\Psi^2_{\Gamma_n}}\in\Gamma(\PP^{2n+5},\omega(2X_n))
    \label{g1}
\end{equation}
(see (\ref{h66})) and show that $[\eta_n]\neq 0$ in
$H^{2n+5}(\PP^{2n+5}\backslash X_n)$. We strongly follow Section 12,
\cite{BEK}, where the computations for $WS_n$ were done.

\begin{lemma}
Let $U=\Spec R$ be a smooth, affine variety and $0\neq f,g\in R$.
Define $Z:=\cV(f,g)\subset U$. We have a map of complexes
\begin{equation}
\Big(\Omega^*_{R[1/f]}/\Omega^*_R\Big)\oplus
\Big(\Omega^*_{R[1/g]}/\Omega^*_R\Big) \xrightarrow{\;\;\gamma\;\;}
\Big(\Omega^*_{R[1/fg]}/\Omega^*_R\Big)
\end{equation}
Then the de Rham cohomology with supports $H^*_{Z,DR}(U)$ can by
computed by the cone of $\gamma$ shifted by $-2$.
\end{lemma}
\begin{proof}
We write the localization sequence for $\cV(f)\subset U$
\begin{equation}
   \rightarrow H_{\cV(f),DR}^i(U)\rightarrow H_{DR}^i(U)\rightarrow
    H_{DR}^i(U\backslash\cV(f)) \rightarrow
    H_{\cV(f),DR}^{i+1}(U)\rightarrow.
\end{equation}
Since $U$ and $U\backslash\cV(f)$ are affine, the de Rham cohomology
is the cohomology of complexes of differential forms $\Omega^*_R$
and $\Omega^*_{R[1/f]}$. This implies
\begin{equation}
    H^*_{\cV(f),DR}(U)= H^*(\Omega^*_{R[1/f]}/\Omega^*_R[-1]).
\end{equation}
We replace $f$ by $g$ resp. $fg$ to get similar equalities for
$\cV(g)$ and $\cV(fg)$. Consider the Mayer-Vietoris sequence for
$\cV(f),\cV(g)\subset U$:
\begin{equation}
\begin{aligned}
\longrightarrow H^*_{Z,DR}(U) \longrightarrow H^*_{\cV(f)}(U)&\oplus
H^*_{\cV(f)}(U)\longrightarrow\\ &H^*_{\cV(f)\cup\cV(g)}(U)
\longrightarrow H^{*+1}_{Z,DR}(U) \longrightarrow.
\end{aligned}\label{g3}
\end{equation}
Now the five-lemma yields that the natural map
$H^*_{Z,DR}(U)\rightarrow H^*(C^*[-2])$ for $C^*:=Cone(\gamma)$
becomes an isomorphism.
\end{proof}
\begin{remark}\label{R2.3.2}
    The direct computation shows that $C^*$ is quasi-isomorphic to the cone of
\begin{equation}
\Big(\Omega^*_{R[1/f]}/\Omega^*_R\Big) \xrightarrow{\;\;\Delta\;\;}
\Big(\Omega^*_{R[1/fg]}/\Omega^*_{R[1/g]}\Big).
\end{equation}
\end{remark}

For the application, we use $U:=\PP^{2n+5}\backslash X_n$. Recall
that the matrix of $\Gamma_n$ looks like
\begin{equation}
 M_{GZZ(n,2)}=
\left(
  \begin{array}{cccccccc}
    \s{\!B_0\!}& \s{\!A_0\!}& \so & \vdots & \so & \so & \so &\s{\!A_{n+3}\!} \\
    \s{\!A_0\!}& \s{\!B_1\!}& \s{\!A_1\!}&\vdots& \so & \so & \so & \so \\
    \so    & \s{\!A_1\!}   & \s{\!B_6\!} &\vdots& \so & \so & \so & \so \\
    \ldots &\ldots &\ldots &\ddots &\ldots &\ldots &\ldots &\ldots \\
    \so    & \so    & \so &\vdots& \s{\!B_{n-1}\!} &\s{\!A_{n-1}\!} & \so &\so \\
    \so    & \so    & \so &\vdots&\s{\!A_{n-1}\!} & \s{\!C_{n}\!}&\s{\!A_n\!} &\s{\!A_{n+2}\!}  \\
    \so    & \so    & \so &\vdots& \so &\s{\!A_{n}\!} & \s{\!B_{n+1}\!} & \s{\!A_{n+1}\!} \\
    \s{\!A_{n+3}\!} & \so&\so&\vdots&\so &\s{\!A_{n+2}\!} &\s{\!A_{n+1}\!}&\s{\!B_{n+2}\!} \\
  \end{array}
\right).
\end{equation}
Define $a_i:=\frac{A_i}{A_{n+3}}$, $b_i:=\frac{B_i}{A_{n+3}}$ and
$c_n:=\frac{C_n}{A_{n+3}}=a_{n+2}+a_{n-1}-a_n$. (We will see that
the forms we work with have no poles along $A_{n+3}=0$.) Let
\begin{equation}
i_j=\frac{I_j}{A_{n+3}^j},\quad
g_{n+2}=\frac{G_{n+2}}{A_{n+3}^{n+3}}.
\end{equation}
Set $f=i_{n+2}$, $g=i_{n+1}$. The equation
$b_{n+2}i_{n+2}-g_{n+2}=0$ defines $X_n$ on $A_{n+3}\neq 0$, then
$b_{n+2}i_{n+2}-g_{n+2}$ is invertible on $U=\PP^{2n+5}\backslash
X$. Thus $g_{n+2}$ is invertible on $\cV(f)$. The element
\begin{equation}
\begin{aligned}
    \beta = -db_0\wedge\ldots&\wedge db_{n-1}\wedge db_{n+1}\wedge\\
    &\wedge da_0\wedge\ldots \wedge
    da_{n+2}\frac{1}{g_{n+2}i_{n+2}}\big(\frac{g_{n+2}}{b_{n+2}i_{n+2}-g_{n+2}}\big)
    \label{g6}
\end{aligned}
\end{equation}
is defined in $\Omega^{2n+4}_{R[1/f]}/\Omega^{2n+4}_R$. It satisfies
\begin{equation}
    d\beta=\eta=\frac{db_0\wedge\ldots\wedge db_{n-1}\wedge db_{n+1}\wedge db_{n+2}\wedge da_0
    \wedge\ldots \wedge da_{n+2}}{(b_{n+2}i_{n+2}-g_{n+2})^2}.
\end{equation}
By Corollary \ref{C1.4}, $I_{n+1}G_{n+2}\equiv (Li_{n+2})^2 \!\!\mod
I_{n+2}$, thus
\begin{equation}
\begin{aligned}
    i_{n+1}g_{n+2}\equiv &(a_{n+1}i_{n+1}-a_{n+2}a_n i_{n}+\\&(-1)^{n-1}a_na_{n-1}\ldots
    a_1a_0)^2 \mod i_{n+2}.
\end{aligned}
\end{equation}
We also use
\begin{equation}
    i_{k}=b_{k-1}i_{k-1}-a^2_{k-2}i_{k-2}
    \label{g9}
\end{equation}
for $k+2$ or $k<n+1$. We now compute in
$\Omega^*_{R[1/fg]}/\Omega^*_{R[1/g]}$ and get
\begin{multline}
    \beta= \frac{di_{n+2}}{i_{n+2}}\wedge\frac{da_{n+2}}{g_{n+2}i_{n+1}}\wedge db_0\wedge
    \ldots\wedge db_{n-1}\wedge\\ \wedge da_0\wedge \ldots \wedge da_{n+1}\cdot
    \Big(1-\frac{b_{n+2}i_{n+2}}{b_{n+2}i_{n+2}-g_{n+2}} \Big)=\\
    -d\Big(\frac{1}{a_{n+1}i_{n+1}-a_{n+2}a_n i_{n}+(-1)^{n-1}a_n\ldots a_0}\cdot\frac{di_{n+2}}{i_{n+2}}\wedge\nu\Big),\label{g10}
    \end{multline}
where
\begin{equation}
    \nu:=\frac{da_{n+2}}{i_{n+1}} \wedge db_0\wedge\ldots\wedge db_{n-1}\wedge da_0\wedge \ldots \wedge
    da_n.
\end{equation}
Using the equality
\begin{equation}
    i_{n+1}=c_ni_n-a_{n-1}^2i_{n-1}=
    (a_{n+2}+a_{n-1}-a_n)i_n-a_{n-1}^2i_{n-1}
\end{equation}
and (\ref{g9}), we get
\begin{equation}
\begin{aligned}
    \nu=\frac{di_{n+1}}{i_{n+1}}\wedge\frac{db_{n-1}}{i_n}\wedge db_{n-2}\wedge\ldots \wedge db_0\wedge da_0\wedge \ldots \wedge
    da_n.\\
    \frac{di_{n+1}}{i_{n+1}}\wedge\frac{di_n}{i_n}\wedge\ldots\wedge\frac{di_2}{i_2}\wedge\frac{db_0}{b_0}\wedge da_0\wedge \ldots \wedge da_n.\\
\end{aligned}
\end{equation}
By (\ref{g10}) one has $\beta=d\theta$ with
\begin{equation}
    \theta:=-\frac{1}{a_{n+1}i_{n+1}-a_{n+2}a_n i_{n}+(-1)^{n-1}a_n\ldots
    a_0}\cdot\frac{di_{n+2}}{i_{n+2}}\wedge\nu.
\end{equation}
Both $\beta$ and $\theta$ have no poles along $A_{n+3}=0$. Thus the
pair
\begin{equation}
    (\beta,\theta)\in H^{2n+5}_{Z,DR}(U)
    \label{g15}
\end{equation}
(see Remark \ref{R2.3.2}) represents a class mapping to $\eta_n\in
H^{2n+5}_{DR}(\PP^{2n+5}\backslash X_n)$, where $Z$ is defined by
\begin{equation}
    Z:=\cV(I_{n+2},I_{n+1}).
\end{equation}
\begin{lemma}\label{L2.3.3}
The natural map
\begin{equation}
    H^{2n+5}_Z(\PP^{2n+5}\backslash X_n)\longrightarrow H^{2n+5}(\PP^{2n+5}\backslash X_n)
\end{equation}
is injective.
\end{lemma}
\begin{proof} The proof goes almost word for word as the proof of
Lemma 12.3. in \cite{BEK} and works for Betti's, de Rham or \'etale
cohomology. Define $Y=\cV(I_{n+2})$. We will show that the desired
map is a composition of two injective maps
\begin{equation}
    H^{2n+5}_Z(\PP^{2n+5}\backslash X)
    \xrightarrow {u} H^{2n+5}_Y(\PP^{2n+5}\backslash X)
    \xrightarrow {v} H^{2n+5}(\PP^{2n+5}\backslash X).
\end{equation}
One has the localization sequence
\begin{equation}
    \rightarrow H^{2n+4}(\PP^{2n+5}\backslash (X\cup Y))\rightarrow H^{2n+5}_Y(\PP^{2n+5}\backslash X_n)
    \rightarrow H^{2n+5}(\PP^{2n+5}\backslash X_n)\rightarrow
\end{equation}
Recall that $I_{2n+2}$ is independent of $B_{n+2}$, $A_{n+1}$ and
$A_{n+3}$. Define $Y_0:=\cV(I_4)\subset\PP^{2n+4}$(no $B_{n+2}$) and
$Y_1:=\cV(I_4)\subset\PP^{2n+2}$(no $B_{n+2}$, $A_{n+1}$,
$A_{n+3}$). Consider the projections
\begin{equation}
    \PP^{2n+5}\backslash(X_n\cup Y)\xrightarrow {\:\pi_1\:} \PP^{2n+4}\backslash Y_0\xrightarrow
    {\:\pi_2\:} \PP^{2n+2}\backslash Y_1,
\end{equation}
where $\pi_1$ forgets the variable $B_{n+2}$ and $\pi_2$ forgets
$A_{n+1}$ and $A_{n+3}$. Since the map $X\backslash Y\rightarrow
\PP^{2n+4}\backslash Y_0$ induced by projection is an isomorphism,
it follows that $\pi_1$ is an $\GG_m$-bundle. The map $\pi_2$ is an
$\AAA^2$-bundle. One gets
\begin{equation}
    H^{2n+4}(\PP^{2n+5}\backslash(X\cup Y))\cong H^{2n+4}(\PP^{2n+2}\backslash
    Y_1)\oplus H^{{2n+3}}(\PP^{2n+2}\backslash Y_1)(-1)=0
\end{equation}
by homotopy invariance and Artin's vanishing. This proves that $v$
is injective.

Consider the two open subschemes
\begin{equation}
\PP^{2n+5}\backslash (X_n\cup Y)\subset \PP^{2n+5}\backslash
(X_n\cup Z)\subset \PP^{2n+5}\backslash X_n,
\end{equation}
then \cite{Milne1},ch.3, Remark 1.26 gives us the following exact
sequence
\begin{equation}
\begin{aligned}
    \longrightarrow H^{2n+4}_{Y-Z}(\PP^{2n+5}\backslash &(X\cup Z)) \longrightarrow\\
    &H^{2n+4}_Z(\PP^{2n+5}\backslash X) \xrightarrow {\;\;u\;\;} H^{2n+4}_Y(\PP^{2n+5}\backslash
    X)\longrightarrow.
\end{aligned}
\end{equation}
Because $I_{n+2}=B_{n+1}I_{n+1}-A_n^2I_n$, the singular locus of $Y$
is contained in $Z$. Thus, for the smooth subscheme $Y\backslash Z$
in $\PP^{2n+5}\backslash (X\cup Z)$ of codimension 1 we may apply
Gysin isomorphism (see \cite{Milne2}, Corollary 16.2)
\begin{equation}
    H^{2n+4}_{Y-Z}(\PP^{2n+5}\backslash (X\cup Z))\cong H^{2n+2}(Y\backslash ((X\cap Y)\cup
    Z))(-1).
    \label{g30}
\end{equation}
The injectivity of $u$ will follow from the vanishing of the
cohomology to the right. Denote by $\pi_3$ the projection from the
point where all the variables but $B_{n+2}$ are zero and define
\begin{equation}
    T:=\pi_3(Y\backslash ((X\cap Y)\cup Z))\subset\PP^{2n+4}.
\end{equation}
Since $X\cap Y=\cV(I_{n+2},G_{n+2})$, the defining polynomials are
independent of $B_{n+2}$, thus $\pi_3$ induces an $\AAA^1$-fibration
and
\begin{equation}
    H^{2n+2}(Y\backslash ((X\cap Y)\cup Z))\cong H^{2n+2}(T).
    \label{g32}
\end{equation}
Now define the projection $\pi_4$ obtained by dropping $A_{n+1}$ and
$A_{n+3}$
\begin{equation}
    \pi_4:T \rightarrow
    Y_1\backslash Z_1\subset\PP^{2n+2}.
    \label{g33}
\end{equation}
The map $\pi_3(Y\backslash Z)\rightarrow Y_1\backslash Z_1$ is an
$\AAA^2$-fibration while $\pi_3(Y\backslash (X\cap Y))\rightarrow
Y_1\backslash Z_1$ is an $\AAA^1$-fibration. It follows that the
fibers of $\pi_4$ are $\AAA^1\times\GG_m$. One gets
\begin{equation}
    H^{2n+2}(T)\cong H^{2n+2}(Y_1\backslash
    Z_1)\oplus H^{2n+1}(Y_1\backslash
    Z_1)(-1).
    \label{g34}
\end{equation}
We know that
\begin{equation}
Y_1\backslash
Z_1=\cV(I_{n+2})\backslash\cV(I_{n+2},I_{n+1})\cong\PP^{2n+1}\backslash\cV(I_{n+1}).
\end{equation}
We may change the variables $C_n:=A_{n+2}$, then $I_{n+1}$ forgets
$A_2$ and thus the scheme to the right becomes a cone over
$\PP^{2n}\backslash\cV(I_{n+1})$. Applying Artin's vanishing, we get
\begin{equation}
    H^i(Y_1\backslash Z_1)=0
\end{equation}
for $i>2n$. The equalities (\ref{g30}), (\ref{g32}) and (\ref{g34})
imply the injectivity of $u$.
\end{proof}
\begin{theorem}
    Let $X_n$ be the graph hypersurface for $\Gamma_n=GZZ(n,2)$ and let\newline
    $[\eta_n]\in H^{2n+5}_{DR}(\PP^{2n+5}\backslash X)$ be the
    de Rham class of (\ref{g1}). Then $[\eta_n]\neq 0$.
\end{theorem}
\begin{proof} The proof is almost the same as that of Theorem 12.4,
\cite{BEK}. We have lifted the class $[\eta_n]$ to a class $(\beta,
\eta)\in H^{2n+5}(\PP^{2n+5}\backslash X)$, see (\ref{g15}). By
Lemma \ref{L2.3.3}, it is enough to show that $(\beta,\eta)\neq 0$.
We localize an the generic point of $Z$ and compute further in the
function field of $Z$. Consider the long denominator of $\beta$ in
(\ref{g10}):
\begin{equation}
    D:= a_{n+1}i_{n+1}-a_{n+2}a_n i_{n}+(-1)^{n-1}a_n\ldots a_0.
\end{equation}
On $\cV(i_{n+2},i_{n+1})$ we have
\begin{equation}
\left\{
\begin{array}{rr}
b_{n+1}i_{n+1}-a^2_ni_n=0\\
i_{n+1}=0
\end{array}
\right. \Rightarrow \left\{
\begin{array}{rr}
a_ni_n=0\\
i_{n+1}=0,\!\!
\end{array}
\right.
\end{equation}
thus both the left and the middle summand of $D$ vanish. Now it
follows that as the class in the function field of $Z$, the class
$(\beta,\eta)$  is represented by
\begin{equation}
    \pm d\log(i_n)\wedge\ldots\wedge
    d\log(i_1) \wedge d\log(a_0)\wedge\ldots\wedge d\log(a_n)
\end{equation}
This is a nonzero multiple of
\begin{equation}
     d\log(b_{n-1})\wedge\ldots\wedge
    d\log(b_0) \wedge d\log(a_0)\wedge\ldots\wedge d\log(a_n),
\end{equation}
so in non-zero as a form. The Deligne theory of MHS yields that the
vector space of logarithmic forms injects into de Rham cohomology of
the open on which those forms are smooth (see (3.1.5.2) in
\cite{De2}). Thus the form above is nonzero.
\end{proof}

\begin{corollary}
    Let $X$ be the graph hypersurface for $\Gamma=ZZ_5$. Then for
    the class of
    $\eta$ defined in (\ref{g1}) one has
    \begin{equation}
        K[\eta]=H^9_{DR}(\PP^9\backslash X).
    \end{equation}
\end{corollary}

\chapter{Gluings}
\section{Classification}
\begin{definition}
The degree $\deg(v)$ of a vertex $v$ (of an undirected graph) is
defined to be the number of edges entering this vertex.
\end{definition}
\begin{lemma}\label{lemma_class}
    For any vertex $v$ of a primitively divergent graph $\Gamma$ the
    following inequality holds
    \begin{equation}
        \deg(v)\geq 3.
    \end{equation}
\end{lemma}
\begin{proof}
   Suppose that $\deg(v)=1$ for a vertex $v\in V(\Gamma)$, so we have an
edge $uv\in E(\Gamma)$ for some vertex $u\in V(\Gamma)$. We delete
the edge $uv$ --- the only one edge connecting $v$ with the other
vertexes of $\Gamma$ --- together with the vertex $v$ and define
$\Gamma'=\Gamma\backslash \{uv\}$. This graph has the same Betti
number but smaller number of edges. This is a contradiction with the
assumption that $\Gamma$ is primitively divergent.

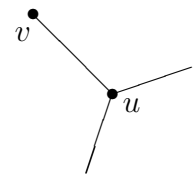
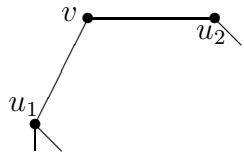
\begin{picture}(50,120)
\put(50,100){\line(1,-1){30}} \put(80,70){\line(3,1){30}}
\put(80,70){\line(-1,-3){10}} %
\put(80,70){\circle*{4}}\put(50,100){\circle*{4}}
\put(43,90){$v$} \put(84,63){$u$}
\put(40,20){\underline{$\deg(v)=1$.}}
\put(220,20){\underline{$\deg(v)=2$.}}
\put(230,100){\line(1,0){48}} 
\put(230,100){\line(-1,-2){20}} 
\put(210,60){\line(0,-1){10}} \put(210,60){\line(1,-1){10}}
\put(278,100){\line(1,-1){10}} %
%
%
%
\put(230,100){\circle*{4}} \put(278,100){\circle*{4}}
\put(210,60){\circle*{4}} \put(220,99){$v$} \put(200,65){$u_1$}
\put(271,92){$u_2$}
\end{picture}\\
\noindent In the case $\deg(v)=2$ for some vertex $v\in V(\Gamma)$,
we denote by $u_1$ and $u_2$ the two vertexes which are adjacent to
$v$. Let $\Gamma'=\Gamma\backslash\{u_1v,u_2v\}$. Note that
$|E(\Gamma')|=|E(\Gamma)|-2$. We know that $h_1(\Gamma)$ is
independent of the choice of a basis of $H_1(\Gamma)$, thus we can
take such a basis that $u_1v$ and $u_2v$ will only appear in one
basis element, and we get
\begin{equation}
    h_1(\Gamma')=h_1(\Gamma)-1.
\end{equation}
Since we found an divergent subgraph of $\Gamma$, $\Gamma$ is not
primitively divergent. Hence,  $d(v)\geq 3$ for any $v\in
V(\Gamma)$.
\end{proof}
We classify primitively divergent graphs with small number of edges.
\begin{theorem}\label{T3.1.3}
Let $\Gamma$ be a primitively divergent graph with $E(\Gamma)=2n$
and $n\leq 6$. Then for $\Gamma$ we have one of the following
possibilities
\begin{itemize}
  \item \underline{n=3}, then $\Gamma\cong WS_3$.
  \item \underline{n=4}, then $\Gamma\cong WS_4$.
  \item \underline{n=4}, then $\Gamma$ is isomorphic to the one
  of the following graphs $WS_5$, $ZZ_5$, $XX_5$ or $ST_5$.
\end{itemize}
\end{theorem}
\begin{proof} Set $m:=|V(\Gamma)|$. Denote by $\alpha_i$ the
degree of the vertex $v_i\in V$ for $1\leq i\leq m$. We can compute
the number of edges of $\Gamma$ by taking the sum of all $\alpha$'s,
and each edge will be counted twice. Thus, one has
\begin{equation}
    \sum_{i=1}^m \alpha_i = 4n.
    \label{f4}
\end{equation}
By Lemma \ref{lemma_class}, $\alpha_i\geq 3$ for every $i$. It
follows that
\begin{equation}
    3m\leq 4n.
    \label{f5}
\end{equation}
Since we do not allow multiple edges and self-loops, we may assume
that $n\geq 3$.

\medskip
\noindent\underline{$n=3$}. The inequality (\ref{f5}) gives us
$m\leq 4$. Define $K_r$ to be a complete graph with $r$ vertexes
Because $K_3$ has only 3 edges, $m\neq 3$. Hence, $m=4$, the
inequality (\ref{f5}) becomes equality, and $\alpha_i=3$ for all
$i$. A graph with 4 vertexes can have an most $\frac{4\cdot 3}{2}=6$
edges, thus $\Gamma$ is isomorphic to $K_4$.
Note that graph $WS_3$ is isomorphic to $K_4$.\par
\begin{picture}(50,80)
    \put(40,20){\line(1,0){72}} \put(40,20){\line(3,4){36}}
    \put(76,68){\line(3,-4){36}} \put(40,20){\line(2,1){36}}
    \put(76,38){\line(2,-1){36}} \put(76,38){\line(0,1){30}}
    \put(40,20){\circle*{4}} \put(76,68){\circle*{4}}
    \put(76,38){\circle*{4}} \put(112,20){\circle*{4}}
    \put(28,22){$v_1$} \put(64,68){$v_2$}
    \put(114,22){$v_3$} \put(71,28){$v_4$}
    \put(200,20){\line(1,0){55}} \put(200,20){\line(0,1){55}}
    \put(255,75){\line(-1,0){55}} \put(255,75){\line(0,-1){55}}
    \put(200,20){\line(1,1){55}} \put(200,75){\line(1,-1){55}}
    \put(200,20){\circle*{4}} \put(255,75){\circle*{4}}
    \put(255,20){\circle*{4}} \put(200,75){\circle*{4}}
    \put(188,22){$v_1$} \put(188,73){$v_2$}
    \put(259,75){$v_3$} \put(259,22){$v_4$}
\end{picture}

\medskip
\noindent\underline{$n=4$}. By the inequality (\ref{f5}) we have
$m\leq 5$. Because $K_4$ has only 6 edges, we get $m=5$. Moreover,
by (\ref{f4}), the only one possibility for $\alpha=(\alpha_i)$ with
$\alpha_1\geq\ldots\geq\alpha_5$ is $(4,3,3,3,3)$. This means that
up to a graph isomorphism, we have the following situation. The
vertex $v_1$ is connected by an edge to each other four vertexes,
and this 4 vertexes are lying on a loop of length 4. Indeed, the
vertex $v_3$ is adjacent to $v_1$ and to two more vertexes, say
$v_2$ and $v_4$. If $v_4$ is adjacent to $v_2$, then $v_5$ must be
connected with itself, this is not allowed. Thus, $v_4$ is adjacent
to $v_5$ and the remaining edge is $v_5v_2$.

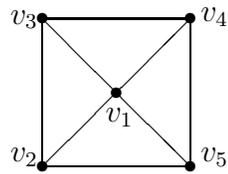
\begin{picture}(50,90)
 \put(100,20){\line(1,0){56}} \put(100,20){\line(0,1){56}}
    \put(156,76){\line(-1,0){56}} \put(156,76){\line(0,-1){56}}
    \put(100,20){\line(1,1){56}} \put(100,76){\line(1,-1){56}}
    \put(100,20){\circle*{4}} \put(156,76){\circle*{4}}
    \put(156,20){\circle*{4}} \put(100,76){\circle*{4}}
    \put(128,48){\circle*{4}} \put(124,37){$v_1$}
    \put(88,22){$v_2$} \put(88,74){$v_3$}
    \put(160,75){$v_4$} \put(160,22){$v_5$}
\end{picture}

\noindent We get an isomorphism $\Gamma\cong WS_4$.

\medskip
\noindent\underline{$n=5$}. By the same argument as above, we get
$m=6$. Take an order $\alpha_1\geq\ldots\geq\alpha_6$, it follows
that we have 2 possibilities: $(5,3,3,3,3,3)$ and $(4,4,3,3,3,3)$.

\medskip \noindent \underline{Case A:} For the first case
$\alpha=(5,3,3,3,3,3)$ we have again one vertex, $v_1$, adjacent to
all other other ones, and this 5 vertexes $v_2,\ldots,v_6$ build a
loop of length 5.

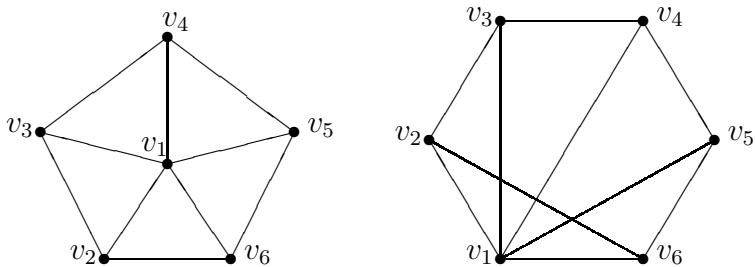
\begin{picture}(50,120)

    \put(50,20){\line(1,0){48}}
    \put(50,20){\line(-1,2){24}}
    \put(98,20){\line(1,2){24}}
    \put(74,56){\line(0,1){48}}
    \put(26,68){\line(4,3){48}}
    \put(50,20){\line(2,3){24}}
    \put(98,20){\line(-2,3){24}}
    \put(74,56){\line(-4,1){48}}
    \put(74,56){\line(4,1){48}}
    \put(74,104){\line(4,-3){48}}
    \put(74,104){\circle*{4}} \put(74,56){\circle*{4}}
    \put(26,68){\circle*{4}} \put(122,68){\circle*{4}}
    \put(50,20){\circle*{4}} \put(98,20){\circle*{4}}
    \put(72,108){$v_4$} \put(64,61){$v_1$}
    \put(13,68){$v_3$} \put(127,68){$v_5$}
    \put(37,20){$v_2$} \put(103,20){$v_6$}

    \put(200,20){\line(-3,5){27}} \put(200,20){\line(1,0){54}}
    \put(254,20){\line(3,5){27}} \put(173,65){\line(3,5){27}}
    \put(200,110){\line(1,0){54}} \put(254,110){\line(3,-5){27}}
    \put(200,20){\line(0,1){90}}
    \put(200,20){\line(3,5){54}}
    \qbezier(200,20)(281,65)(281,65)
    \qbezier(254,20)(173,65)(173,65)
    \put(200,110){\circle*{4}} \put(254,110){\circle*{4}}
    \put(173,65){\circle*{4}} \put(281,65){\circle*{4}}
    \put(200,20){\circle*{4}} \put(254,20){\circle*{4}}
    \put(187,110){$v_3$} \put(259,110){$v_4$}
    \put(160,65){$v_2$} \put(286,65){$v_5$}
    \put(187,20){$v_1$} \put(259,20){$v_6$}

\end{picture}
\newline
Recall that the adjacency matrix for a undirected graph with $m$
loops is a symmetric $m\times m$ matrix $Ad=(a_{i,j})$ such that
$a_{i j}=1$ when $v_i$ is connected to $v_j$ and $a_{i j}=0$
otherwise. The degree of a vertex $v_i$ can be computed as
\begin{equation}
    \deg(v_i)=\sum_{j=1}^m a_{i j}.
\end{equation}
For Case A the adjacency matrix is the following.

\begin{equation*}
   Ad_A=\left(
   \begin{tabular}{c c c c c c}
        0 &1 &1 &1 &1 &1\\
        1 &0 &1 &0 &0 &1\\
        1 &1 &0 &1 &0 &0\\
        1 &0 &1 &0 &1 &0\\
        1 &0 &0 &1 &0 &1\\
        1 &1 &0 &0 &1 &0\\
   \end{tabular}
   \right).
\end{equation*}

\medskip \noindent \underline{Case B}: Suppose now that
$\alpha=(4,4,3,3,3,3)$. There are two different situations,
depending on wether the vertexes of degree 4 are connected to each
other or not.

\medskip \noindent \underline{Case B.1}: Consider the case where the
two vertexes of degree 4 (namely $v_1$ and $v_2$) are adjacent.
Without loss of generality, we may assume that the one vertex that
is not adjacent to $v_1$ is $v_6$. We get the following adjacency
matrix.
\begin{equation*}
  Ad_{B.1}=  \left(
   \begin{array}{c c c c c c}
        0 &1 &1 &1 &1 &0\\
        1 &0 &\ast &\ast &\ast &\ast\\
        1 &\ast &0 &\ast &\ast &\ast\\
        1 &\ast &\ast &0 &\ast &\ast\\
        1 &\ast &\ast &\ast &0 &\ast\\
        0 &\ast &\ast &\ast &\ast &0\\
   \end{array}
   \right).
\end{equation*}
Again, we have to distinguish two cases.

\medskip \noindent \underline{Case B.1.1}: Suppose that $v_2v_6\in
E(\Gamma)$. The vertex $v_2$ is adjacent to $v_1$ and $v_6$, and
$\deg(v_2)=4$. Without loss of generality, we may assume that
$v_2v_3,v_2v_4\in E(\Gamma)$ and $v_2v_5\not \in E(\Gamma)$.
\begin{equation*}
   Ad_{B.1.1}\left(
   \begin{array}{c c c c c c}
        0 &1 &1 &1 &1 &0\\
        1 &0 &1 &1 &0 &1\\
        1 &1 &0 &\ast &\ast &\ast\\
        1 &1 &\ast &0 &\ast &\ast\\
        1 &0 &\ast &\ast &0 &\ast\\
        0 &1 &\ast &\ast &\ast &0\\
   \end{array}
   \right).
\end{equation*}
Assume for a moment that $v_3$ is adjacent to $v_4$. Because
$\deg(v_3)=\deg(v_4)=3$ and both $v_3$ and $v_4$ are adjacent to
$v_1$ and $v_2$, we conclude that $v_6v_3$, $v_6v_4\not\in
E(\Gamma)$. Then $v_6$ is only adjacent to $v_1$ and, may be, to
$v_5$; this contradicts $\deg(v_6)=3$. Hence $v_3v_4\not\in
E(\Gamma)$ and $v_3$ is adjacent to $v_5$ or $v_6$. Note that if we
interchange vertexes $v_1\leftrightarrow v_2$ and
$v_5\leftrightarrow v_6$, we get the same adjacency matrix and the
graph isomorphic to $\Gamma$. Thus, one can assume that $v_3v_6\in
E(\Gamma)$ and, consequently, $v_3v_5\not\in E(\Gamma)$. Since
$\deg(v_5)=3$, it follows that $v_5$ is adjacent to $v_4$ and $v_6$.
We obtain
\begin{equation*}
  Ad_{B.1.1}= \left(
   \begin{array}{c c c c c c}
        0 &1 &1 &1 &1 &0\\
        1 &0 &1 &1 &0 &1\\
        1 &1 &0 &0 &0 &1\\
        1 &1 &0 &0 &1 &0\\
        1 &0 &0 &1 &0 &1\\
        0 &1 &1 &0 &1 &0\\
   \end{array}
   \right).
\end{equation*}\pagebreak[3]
We can redraw this graph in a suitable way and see that it is
isomorphic to $ZZ_5$.

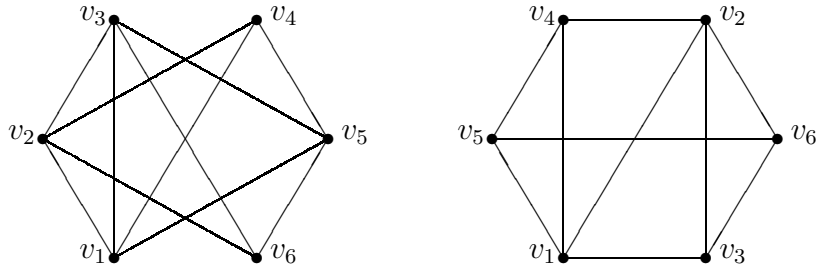
\begin{picture}(100,120)
    \put(70,20){\line(-3,5){27}} \put(70,20){\line(3,5){54}}
    \qbezier(70,20)(151,65)(151,65) \put(70,20){\line(0,1){90}}
     \put(43,65){\line(3,5){27}}
     \qbezier(43,65)(124,110)(124,110)
     \qbezier(43,65)(124,20)(124,20)
    \put(70,110){\line(3,-5){54}}
    \qbezier(70,110)(151,65)(151,65)
     \put(124,110){\line(3,-5){27}}
    \put(124,20){\line(3,5){27}}


    \put(70,110){\circle*{4}} \put(124,110){\circle*{4}}
    \put(43,65){\circle*{4}} \put(151,65){\circle*{4}}
    \put(70,20){\circle*{4}} \put(124,20){\circle*{4}}
    \put(57,110){$v_3$} \put(129,110){$v_4$}
    \put(30,65){$v_2$}  \put(156,65){$v_5$}
    \put(57,20){$v_1$}  \put(129,20){$v_6$}

    \put(240,20){\line(-3,5){27}}  \put(213,65){\line(3,5){27}}
    \put(240,110){\line(1,0){54}}  \put(294,110){\line(3,-5){27}}
    \put(321,65){\line(-3,-5){27}} \put(240,20){\line(1,0){54}}
      \put(240,20){\line(0,1){90}}   \put(240,20){\line(3,5){54}}
      \put(213,65){\line(1,0){108}}
      \put(294,20){\line(0,1){90}}
    \put(240,110){\circle*{4}}  \put(294,110){\circle*{4}}
    \put(213,65){\circle*{4}}   \put(321,65){\circle*{4}}
    \put(240,20){\circle*{4}}   \put(294,20){\circle*{4}}
    \put(227,110){$v_4$} \put(299,110){$v_2$}
    \put(200,65){$v_5$} \put(326,65){$v_6$}
    \put(227,20){$v_1$} \put(299,20){$v_3$}
\end{picture}

\medskip \noindent \underline{Case B.1.2}: Now we suppose that the vertex
$v_2$ is not adjacent to $v_6$. Together with $\deg(v_6)=3$ this
implies that  $v_6$ is connected to $v_5$, $v_4$ and $v_3$. Since
$\deg(v_3)=\deg(v_4)=\deg(v_5)=3$, all this three vertexes are
mutually not connected. We finally get the following adjacency
matrix
\begin{equation*}
   Ad_{B.1.2}\left(
   \begin{array}{c c c c c c}
        0 &1 &1 &1 &1 &0\\
        1 &0 &1 &1 &1 &0\\
        1 &1 &0 &0 &0 &1\\
        1 &1 &0 &0 &0 &1\\
        1 &1 &0 &0 &0 &1\\
        0 &0 &1 &1 &1 &0\\
   \end{array}
   \right).
\end{equation*}
We will refer to this graph as $ST_5$, which means "strange". For
this graph we cannot say anything important about the graph
hypersurface on the cohomological level.\medskip\\
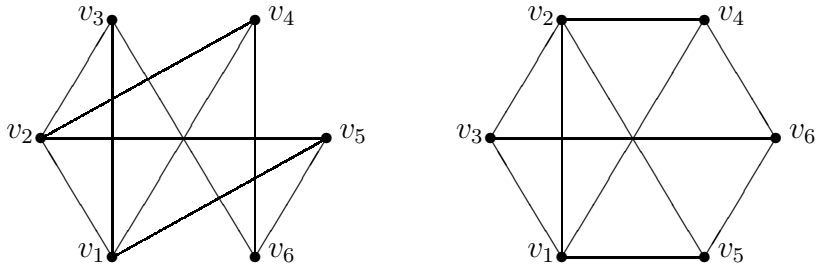
\noindent
\begin{picture}(100,120)
    \put(70,20){\line(-3,5){27}} \put(70,20){\line(0,1){90}}
    \put(70,20){\line(3,5){54}}
    \qbezier(70,20)(151,65)(151,65)
     \put(43,65){\line(3,5){27}}
     \qbezier(43,65)(124,110)(124,110)
     \put(43,65){\line(1,0){108}}
    \put(70,110){\line(3,-5){54}}
     \put(124,110){\line(0,-1){90}}
    \put(124,20){\line(3,5){27}}


    \put(70,110){\circle*{4}} \put(124,110){\circle*{4}}
    \put(43,65){\circle*{4}} \put(151,65){\circle*{4}}
    \put(70,20){\circle*{4}} \put(124,20){\circle*{4}}
    \put(57,110){$v_3$} \put(129,110){$v_4$}
    \put(30,65){$v_2$}  \put(156,65){$v_5$}
    \put(57,20){$v_1$}  \put(129,20){$v_6$}

    \put(240,20){\line(-3,5){27}}  \put(213,65){\line(3,5){27}}
    \put(240,110){\line(1,0){54}}  \put(294,110){\line(3,-5){27}}
    \put(321,65){\line(-3,-5){27}} \put(240,20){\line(1,0){54}}
      \put(240,20){\line(3,5){54}}   \put(240,110){\line(3,-5){54}}
      \put(213,65){\line(1,0){108}}  \put(240,20){\line(0,1){90}}
    \put(240,110){\circle*{4}}  \put(294,110){\circle*{4}}
    \put(213,65){\circle*{4}}   \put(321,65){\circle*{4}}
    \put(240,20){\circle*{4}}   \put(294,20){\circle*{4}}
    \put(227,110){$v_2$} \put(299,110){$v_4$}
    \put(200,65){$v_3$} \put(326,65){$v_6$}
    \put(227,20){$v_1$} \put(299,20){$v_5$}
\end{picture}

\medskip \noindent \underline{Case B.2}: We consider the case when
the two vertexes of degree four are not adjacent. Since
$\deg(v_3)=3$, the vertex $v_3$ must be adjacent to only one of the
vertexes $v_4$, $v_5$ or $v_6$. Without loss of generality, we may
assume that $v_3v_4\in E(\Gamma)$. Now, $v_6$ cannot be adjacent to
$v_4$, thus $v_5v_6\in E(\Gamma)$. We obtain
\begin{equation*}
   Ad_{B.2}=\left(
   \begin{array}{c c c c c c}
        0 &0 &1 &1 &1 &1\\
        0 &0 &1 &1 &1 &1\\
        1 &1 &0 &1 &0 &0\\
        1 &1 &1 &0 &0 &0\\
        1 &1 &0 &0 &0 &1\\
        1 &1 &0 &0 &1 &0\\
   \end{array}
   \right).
\end{equation*}
We redraw the graph in a suitable way and call the right drawing
$XX_5$.
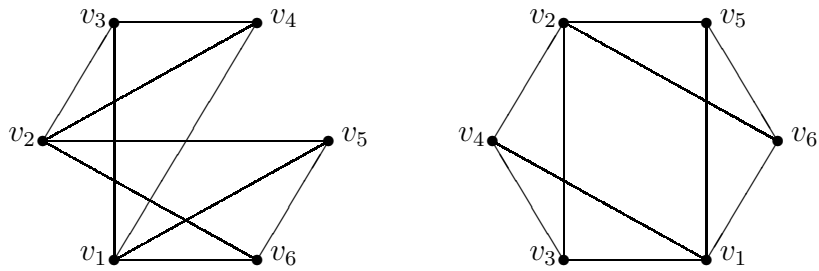
\begin{picture}(100,120)
    \put(70,20){\line(0,1){90}} \put(70,20){\line(1,0){54}}
    \put(70,20){\line(3,5){54}}
    \qbezier(70,20)(151,65)(151,65)
     \put(43,65){\line(3,5){27}}
     \qbezier(43,65)(124,20)(124,20)
     \qbezier(43,65)(124,110)(124,110)
     \put(43,65){\line(1,0){108}}
    \put(70,110){\line(1,0){54}}
    \put(124,20){\line(3,5){27}}


    \put(70,110){\circle*{4}} \put(124,110){\circle*{4}}
    \put(43,65){\circle*{4}} \put(151,65){\circle*{4}}
    \put(70,20){\circle*{4}} \put(124,20){\circle*{4}}
    \put(57,110){$v_3$} \put(129,110){$v_4$}
    \put(30,65){$v_2$}  \put(156,65){$v_5$}
    \put(57,20){$v_1$}  \put(129,20){$v_6$}

    \put(240,20){\line(-3,5){27}}  \put(213,65){\line(3,5){27}}
    \put(240,110){\line(1,0){54}}  \put(294,110){\line(3,-5){27}}
    \put(321,65){\line(-3,-5){27}} \put(240,20){\line(1,0){54}}
      \put(240,20){\line(0,1){90}}   \put(294,20){\line(0,1){90}}
    \qbezier(240,110)(321,65)(321,65)
    \qbezier(294,20)(213,65)(213,65)
    \put(240,110){\circle*{4}}  \put(294,110){\circle*{4}}
    \put(213,65){\circle*{4}}   \put(321,65){\circle*{4}}
    \put(240,20){\circle*{4}}   \put(294,20){\circle*{4}}
    \put(227,110){$v_2$} \put(299,110){$v_5$}
    \put(200,65){$v_4$} \put(326,65){$v_6$}
    \put(227,20){$v_3$} \put(299,20){$v_1$}
\end{picture}
\end{proof}

In the next sections we study the graph hypersurface for $XX_5$ and
the gluings --- the graph obtained by the construction of gluing
motivated by the shape of the graph $XX_5$.
\newpage
\section{$XX_5$}
In this section we compute the middle dimensional cohomology of the
graph hypersurface for the graph $XX_5$ (see Theorem \ref{T3.1.3})
in the similar way as it was done for $ZZ_5$ in Ch.2, Sect.1. We
again consider $H^*(\cdot)$ to be \'etale or Betti's cohomology.

\begin{theorem} Let $X$ be the graph hypersurface for $XX_5$. Then
\begin{equation}
H^8_{prim}(X)\cong \QQ(-3).
\end{equation}
\end{theorem}
\begin{proof}
According to the classification, $XX_5$ is a primitively log
divergent graph with 10 edges and $h_1(XX_5)=5$. We orient and
number edges in the following way:
\bigskip\medskip\smallskip\\
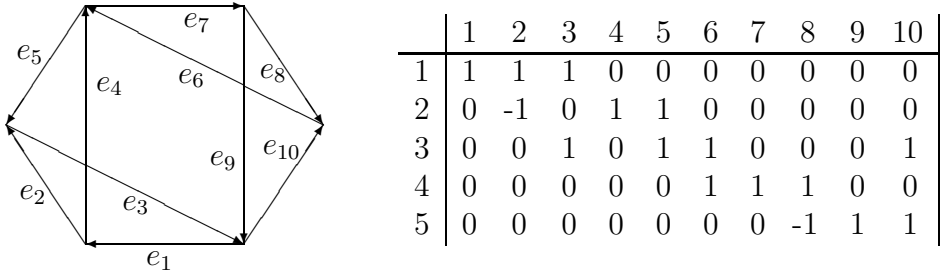
\begin{picture}(0,0)(0,60)
\put(50,20){\vector(-2,3){30}} \put(110,20){\vector(-1,0){60}}
\put(50,110){\vector(-2,-3){30}} \put(50,110){\vector(1,0){60}}
\put(110,110){\vector(2,-3){30}} \put(110,20){\vector(2,3){30}}
\put(20,65){\vector(2,-1){90}} \put(140,65){\vector(-2,1){90}}
\put(50,20){\vector(0,1){90}} \put(110,110){\vector(0,-1){90}}
\put(73,12){$e_1$} \put(25,37){$e_2$} \put(64,33){$e_3$}
\put(24,90){$e_5$} \put(54,78){$e_4$} \put(85,81){$e_6$}
\put(97,50){$e_9$} \put(87,102){$e_7$} \put(116,83){$e_8$}
\put(117,54){$e_{10}$}
\end{picture}
\mathstrut\qquad\qquad\qquad\qquad\qquad\qquad\qquad
\begin{tabular}{c|c c c c c c c c c c|}
      & 1 & 2 & 3 & 4 & 5 & 6 & 7 & 8 & 9 &\!10\\ \hline
    1 & 1 & 1 & 1 & 0 & 0 & 0 & 0 & 0 & 0 & 0\\
    2 & 0 &\!-1& 0 & 1 & 1 & 0 & 0 & 0 & 0 & 0\\
    3 & 0 & 0 & 1 & 0 & 1 & 1 & 0 & 0 & 0 & 1\\
    4 & 0 & 0 & 0 & 0 & 0 & 1 & 1 & 1 & 0 & 0\\
    5 & 0 & 0 & 0 & 0 & 0 & 0 & 0 &\!-1& 1 & 1\\
\end{tabular}
\bigskip\\
By the construction, we get a matrix in variables $T_i$, and
changing coordinates, we obtain the matrix
\begin{equation}
M_{XX_5}(A,B):=
 \left(
  \begin{array}{ccccc}
     B_0  & A_0 & A_4  &  0   &  0   \\
     A_0  & B_1 & A_1  &  0   &  0   \\
     A_4  & A_1 & C_2  & A_2  & A_5  \\
      0   &  0  & A_2  & B_3  & A_3  \\
      0   &  0  & A_5  & A_3  & B_4  \\
  \end{array}
\right)
\end{equation}
with variables $A_0,\ldots,A_5,B_0,B_1,B_3,B_4$ and
$C_2:=A_1+A_2+A_4+A_5$. Define
\begin{equation}
X:=\cV(I_5)\subset\PP^9(A_0:\ldots:A_5:B_0:B_1:B_3:B_4),
\end{equation}
and we compute $H^{mid}_{prim}(X)=H^8_{prim}(X)$ here. The
determinant $I_5$ can be written in the following way:
\begin{equation}
    I_5=I_4B_4-G_4
\end{equation}
with
\begin{equation}
G_4=A_3^2I_3+A_5^2I_2B_3-2A_3A_5I_2A_2.
\end{equation}
Consider the subvariety $\cV(I_5,I_4)^{(9)}\subset X$ and define
$U:=X\backslash\cV(I_5,I_4)^{(9)}$. One has an exact sequence
\begin{equation}
    \rightarrow H^8_c(U) \rightarrow H^8_{prim}(X)\rightarrow
    H^8_{prim}(\cV(I_4,G_4)^{(9)})\rightarrow H^9_c(U) \rightarrow.
    \label{p16}
\end{equation}
\begin{lemma} One has $H^i_c(U)$ for $i<10$.
\end{lemma}
\begin{proof}
The scheme $U$ is defined by the system
\begin{multline}\;\;
    \left\{
        \begin{aligned}
        I_5=B_4I_4-G_4=0\\
        I_4\neq 0.\!\!\\
        \end{aligned}
    \right.\;\;\;\;
\end{multline}
We solve the first equation on $B_4$; projecting from the point
where all the variables but $B_4$ are zero, we get an isomorphism
\begin{equation}
U\cong \PP^8\backslash \cV(I_4)
\end{equation}
with the scheme to the right in $\PP^8$(no $B_4$). The polynomial
$I_4$ is independent of $A_3$. Applying \emph{Theorem B} ($N=8$,
$k=0$, $t=1$) to $\PP^8\backslash \cV(I_4)$, we get
\begin{equation}
H^i(U)=H^i_c(\PP^8\backslash \cV(I_4))=0\quad\text{for}\;\; i<9
\end{equation}
and
\begin{equation}
    H^9_c(U)\cong H^9_c(\PP^8\backslash\cV(I_4))\cong
    H^7_c(\PP^7\backslash\cV(I_4))(-1),
    \label{p20}
\end{equation}
where the scheme to the right lives in $\PP^7$(no $B_4$, $A_3$). One
has an exact sequence
\begin{equation}
\rightarrow H^6_{prim}(\PP^7)\rightarrow
H^6_{prim}(\cV(I_4))\rightarrow H^7_c(\PP^7\backslash\cV(I_4))
\rightarrow H^7(\PP^7)\rightarrow.
\end{equation}
Since the outermost terms vanish, we get an isomorphism
\begin{equation}
    H^7_c(\PP^7\backslash\cV(I_4))\cong H^6_{prim}(\cV(I_4)).
    \label{p22}
\end{equation}
The only one appearance of $A_5$ in $I_4$ is in the sum $C_2$. We
make the linear change of coordinates $C_2:=A_5$ and think of $C_2$
as independent variable. Denote by $I_4'$ and $I_3'$ the images of
the polynomials $I_4$ and $I_3$ under this transformation. One has
an isomorphism
\begin{equation}
    H^6_{prim}(\cV(I_4))\cong H^6_{prim}(\cV(I_4')).
    \label{p23}
\end{equation}
Together with (\ref{p20}) and (\ref{p22}), we get
\begin{equation}
    H^9_c(U)\cong H^6_{prim}(\cV(I_4'))(-1),
    \label{p24}
\end{equation}
where $\cV(I_4')\subset\PP^7$(no $B_4$, $A_3$). We can write
\begin{equation}
    I_4'=B_3I_3'-A_2^2I_2.
    \label{p25}
\end{equation}
Define $\hat{T}:=\cV(I_4',I_3')\subset\cV(I_4')\subset\PP^7$ and
$S:=\cV(I_4')\backslash \hat{T}$. One has an exact sequence
\begin{equation}
    \longrightarrow H^6_c(S)\longrightarrow H^6_{prim}(\cV(I_4'))
    \longrightarrow H^6_{prim}(\hat{T})\longrightarrow.
    \label{p26}
\end{equation}
On $\cV(I_4')\backslash \hat{T}$ we solve (\ref{p25}) on $B_3$;
projecting from the point where all the variables but $B_3$ are
zero, we get an isomorphism
\begin{equation}
    S\cong \PP^6\backslash \cV(I_3').
\end{equation}
The polynomial $I_3'$ is independent of $A_2$, \emph{Theorem B}
($N=6$, $k=0$, $t=1$) implies
\begin{equation}
        H^6_c(S)=0.
\end{equation}
The variety
\begin{equation}
\hat{T}=\cV(I_4',I_3')=\cV(I_3',A_2I_2)\subset\PP^7
\end{equation}
is defined by the polynomials both independent of $B_3$. By
\emph{Theorem A} ($N=7$, $k=2$, $t=1$), we obtain
\begin{equation}
    H^6_{prim}(\hat{T})\cong H^4_{prim}(T)(-1),
    \label{p30}
\end{equation}
where $T:=\cV(I_3',A_2I_2)\subset\PP^6$(no $B_4$, $A_3$, $B_3$). The
sequence (\ref{p26}) simplifies to
\begin{equation}
    0\longrightarrow H^6_{prim}(\cV(I_4'))
    \longrightarrow H^4_{prim}(T)(-1)\rightarrow.
    \label{p31}
\end{equation}
Define $T_1:=\cV(I_3',I_2)\subset T$ and $T_c:=T\backslash T_1$. We
have an exact sequence
\begin{equation}
    \longrightarrow H^3(T_1) \longrightarrow H^4_c(T_c)\longrightarrow H^4_{prim}(T)
    \longrightarrow H^4_{prim}(T_1)\longrightarrow.
    \label{p32}
\end{equation}
Note that $T_1=\cV(I_3',I_2)=\cV(I_2,G_2)\subset\PP^6$ with
$G_2:=I_3'-C_2I_2$, the defining polynomials are independent of
$C_2$ and $A_2$. Thus, \emph{Theorem A} ($N=6$, $k=2$, $t=2$) gives
us $H^i_{prim}(T_1)=0$ for $i<6$. The sequence (\ref{p32}) implies
an isomorphism
\begin{equation}
    H^4_{prim}(T)\cong H^4_c(T_c).
    \label{p33}
\end{equation}
The scheme $T_c\subset\PP^6$ is defined by the sequence
\begin{multline}\;\;
    \left\{
        \begin{aligned}
        I'_3=0\\
        A_2I_2=0\\
        I_2\neq 0\\
        \end{aligned}
    \right.
   \Leftrightarrow \left\{
        \begin{aligned}
        C_2I_2-G_2=0\\
        A_2=0\\
        I_2\neq 0.\!\!\\
        \end{aligned}
    \right. \;\;\;\;
\end{multline}
We solve the first equation on $C_2$. Applying \emph{Theorem B}
($N=4$, $k=0$, $t=2$) to $\PP^4\backslash\cV(I_2)$, we get
\begin{equation}
    H^4_c(T_c)=H^4_c(\PP^4\backslash\cV(I_2))=0.
    \label{p35}
\end{equation}
By (\ref{p24}), (\ref{p31}) and (\ref{p33}), we finally obtain
\begin{equation}
    H^4_{prim}(\cV(T))(-2)=0,\quad H^9_c(U)\cong
    H^6_{prim}(\cV(I'_4))(-1)=0.
\end{equation}
\end{proof}

The lemma gives us the vanishing of the outmost terms of the
sequence (\ref{p16}), thus one gets an isomorphism
\begin{equation}
    H^8_{prim}(X)\cong H^8_{prim}(\cV(G_4,I_4)^{(9)}).
    \label{23}
\end{equation}
The variety to the right is defined by the equations independent of
$B_4$. \emph{Theorem A} implies
\begin{equation}
    H^8_{prim}(X)\cong H^6_{prim}(\cV(G_4,I_4))(-1),
    \label{p38}
\end{equation}
where the variety to the right lives in $\PP^8$(no $B_4$). The
polynomial $G_4$ is defined by
\begin{multline}
    G_4:= A_3^2 \begin{vmatrix}
                        B_0 & A_0 & A_4\\
                        A_0 & B_1 & A_1\\
                        A_4 & A_1 & C_2\\
                \end{vmatrix}
        + A_5^2  \begin{vmatrix}
                        B_0 & A_0 & 0 \\
                        A_0 & B_1 & 0\\
                         0  &  0  & B_3\\
                \end{vmatrix}
         -2A_3A_5 \begin{vmatrix}
                        B_0 & A_0 & A_4 \\
                        A_0 & B_1 & A_1\\
                         0  &  0  & A_2\\
                \end{vmatrix}=\\
    =A_3^2I_3+A_5^2I_2B_3-2A_3A_5I_2A_2.
    \label{p39}
\end{multline}

Define $\hat{V},U_1\subset\cV(G_4,I_4)\subset\PP^8$(no $B_4$) by
\begin{equation}
    \hat{V}:=\cV(G_4,I_4,I_3)
    \label{p40}
\end{equation}
and $U_1:=\cV(G_4,I_4)\backslash\hat{V}$. One can write an exact
sequence
\begin{equation}
    \longrightarrow H^6_c(U_1)
    \longrightarrow H^6_{prim}(\cV(I_4,G_4))\longrightarrow
    H^6_{prim}(\hat{V})\longrightarrow
    H^7_c(U_1) \longrightarrow
    \label{p41}
\end{equation}
The scheme $U_1$ is defined by the system
\begin{equation}
\left\{
        \begin{aligned}
        G_4=0\\
        I_4=0\\
        I_3\neq 0.\!\!\\
        \end{aligned}
    \right.
\end{equation}
Such $U$'s were studied in Chapter 1, Section 1. By Theorem
\ref{T1.7}, it follows that
\begin{equation}
    U_1\cong U_2:=\cV(I_4)\backslash\cV(I_4,I_3)\subset\PP^7(\text{no}\; B_4, A_3)
    \label{p43}
\end{equation}
Using the equality
\begin{equation}
    I_4=B_3I_3-A_2^2I_2
\end{equation}
and projecting from the point where all the coordinates but $B_3$
are zero, we obtain an isomorphism
\begin{equation}
    U_2\cong \PP^6\backslash\cV(I_3).
    \label{p45}
\end{equation}
We change the coordinates $C_2:=A_5$ and denote by $I_3'$ the image
of $I_3$ under this transformation.  The polynomial $I_3'$ is
independent of $A_2$. \emph{Theorem B} ($N=6$, $k=0$, $t=1$) yields
\begin{equation}
\begin{array}{ll}
    H^i(\PP^6\backslash\cV(I_3'))=0\quad\text{for}\;\;i<7,\\
    H^7(\PP^6\backslash\cV(I_3'))\cong H^5(\PP^5\backslash\cV(I_3'))(-1)
    \label{p46}
\end{array}
\end{equation}
with the scheme to the right in $\PP^5$(no $B_4$, $A_3$, $B_3$,
$A_2$). Note that $\cV(I_3')\subset\PP^5$ is exactly the graph
hypersurface for $WS_3$. By Theorem \ref{T3.4.2}, and using the
exact sequence
\begin{equation}
\longrightarrow H^4_{prim}(\PP^5)\longrightarrow
H^4_{prim}(\cV(I_3')) \longrightarrow H^5_c(\PP^5\backslash
\cV(I_3'))\longrightarrow H^5(\PP^5) \longrightarrow,
\end{equation}
we get
\begin{equation}
    H^5_c(\PP^5\backslash\cV(I_3'))\cong
    H^4_{prim}(\cV(I_3'))\cong\QQ(-2).
\end{equation}
Collecting together (\ref{p43}), (\ref{p45}) and (\ref{p46}), we
obtain
\begin{equation}
\begin{array}{ll}
            H^i_c(U_1)=0\quad\text{for}\;\; i<7,\\
            H^7_c(\cV(G_4,I_4)\backslash\hat{V})\cong\QQ(-3).
\end{array}
\end{equation}
The sequence (\ref{p41}) simplifies to
\begin{equation}
 0 \rightarrow H^6_{prim}(\cV(I_4,G_4))\rightarrow
    H^6_{prim}(\hat{V})\rightarrow \QQ(-3) \rightarrow
    \label{p50}
\end{equation}
Now consider $\hat{V}=\cV(G_4,I_4,I_3)\subset\PP^8$(no $B_4$). By
Theorem \ref{T1.5}, the polynomial $G_4$ (see (\ref{p39})) is
independent of $A_3$ on $\hat{V}$. Thus, we can write
\begin{equation}
    \hat{V}=\cV(A_5I_2B_3,I_4,I_3)^{(8)}.
\end{equation}
Now, all the three defining polynomials of $\hat{V}$ are independent
of $A_3$. We apply \emph{Theorem A} ($N=8$, $k=3$, $t=1$) and obtain
\begin{equation}
    H^6_{prim}(\hat{V})\cong H^4_{prim}(V)(-1),
    \label{p52}
\end{equation}
where
\begin{equation}
    V:=\cV(A_5I_2B_3,I_4,I_3)=\cV(A_5I_2B_3,I_3,A_2I_2)\subset
    \PP^7(\text{no}\;B_3, A_3).
\end{equation}
Define $V_1\subset V\subset\PP^7$ by
\begin{equation}
    V_1:=V\cap \cV(I_2)=\cV(I_3,I_2).
\end{equation}
One has an exact sequence
\begin{equation}
    \longrightarrow H^3(V_1) \longrightarrow
    H^4_c(V\backslash V_1)\longrightarrow
    H^4_{prim}(V)\longrightarrow H^4_{prim}(V_1) \longrightarrow
\end{equation}
 We see that the defining polynomials of $V_1$ are independent of $A_2$ and $B_3$.
\emph{Theorem A} ($N=7$, $k=2$, $t=2$) yields
\begin{equation}
H^i(V_1)=0\quad\text{for}\;\;i<7.
\end{equation}
Then the sequence above implies an isomorphism
\begin{equation}
    H^4_{prim}(V)\cong H^4_c(V\backslash V_1).
    \label{p57}
\end{equation}
The scheme $V\backslash V_1\subset\PP^7$ (no $B_4$ or $A_3$) is
defined by the system
\begin{multline}\;\;
    \left\{
        \begin{aligned}
        A_5B_3I_2=0\\
        I_3=0\\
        A_2I_2=0\\
        I_2\neq 0\\
        \end{aligned}
    \right.
   \Leftrightarrow \left\{
        \begin{aligned}
        A_5B_3=0\\
        I_3=0\\
        A_2=0\\
        I_2\neq 0.\!\!\\
        \end{aligned}
    \right. \;\;\;\;
\end{multline}
The variable $B_3$ appears only in the first equation. Set
$V_{02}:=(V\backslash V_1)\cap\cV(A_5)$ and $V_{03}:=(V\backslash
V_1)\cap\cV(B_3)$. We write the Meyer-Vietoris sequence for
$V\backslash V_1$:
\begin{equation}
\begin{aligned}
    \longrightarrow H^3_c(V_{02}) \oplus H^3_c(V_{03})&\longrightarrow H^3_c(V_{02}\cap
    V_{03})\longrightarrow\\
    &H^4_c(V\backslash V_1)\longrightarrow H^4_c(V_{02}) \oplus H^4_c(V_{03}) \longrightarrow
\end{aligned}\label{p59}
\end{equation}
The defining polynomials of $V_{02}$ are all independent of $B_3$.
\emph{Theorem B} ($N=7$, $k=3$, $t=1$) gives us
\begin{equation}
    H^i_c(V_{02})=0\quad\text{for}\;\;i<5.
    \label{p60}
\end{equation}
Now, the variety $V_{03}$ is defined by the system
\begin{equation}
    \left\{
        \begin{aligned}
        B_3=A_2=0\\
        I_3=C_2I_2-G_2=0\\
        I_2\neq 0.\!\!\\
        \end{aligned}
    \right.
\end{equation}
The only one appearance of $A_5$ in the equations is in $C_2$, so we
change the variables $C_2:=A_5$. Thus $V_{03}$ is isomorphic with
the variety $V'_{03}$ defined by the same equations but with $C_2$
independent. We solve the second equation on $C_2$; projecting from
the point where all the variables but $C_2$ are zero, we get
\begin{equation}
V_{03}\cong V'_{03}\cong \PP^4\backslash\cV(I_2),
\end{equation}
where the scheme to the right lives in $\PP^4(B_0:B_1:A_0:A_1:A_4)$.
Since $I_2$ in independent of $A_1$ and $A_4$, \emph{Theorem B}
($N=4$, $k=0$, $t=2$) implies
\begin{equation}
    H^i_c(V_{03})\cong H^i_c(\PP^4\backslash
    I_2)=0\quad\text{for}\;\;i<6.
    \label{p63}
\end{equation}
By (\ref{p60}) and (\ref{p63}), the sequence (\ref{p59}) yields an
isomorphism
\begin{equation}
    H^4_c(V\backslash V_1)\cong H^3_c(V_{02}\cap V_{03}).
    \label{p64}
\end{equation}
On $V_{02}\cap V_{03}$ all the variables $A_5$,$A_2$ and $B_3$ are
zero, thus we obtain
\begin{equation}
    V_{02}\cap V_{03} \cong \cV(I_3)\backslash\cV(I_3,I_2),
    \label{p65}
\end{equation}
where the scheme to the right lives in $\PP^4(B_0:B_1:A_0:A_1:A_4)$.
Together with (\ref{p57}) and (\ref{p65}), one gets an isomorphism
\begin{equation}
    H^4_{prim}(V)\cong
    H^3_c(\cV(I_3)\backslash\cV(I_3,I_2)).
    \label{p66}
\end{equation}
For the scheme to the right, we have an exact sequence
\begin{multline}
    \rightarrow H^2_{prim}(\cV(I_3)) \rightarrow H^2_{prim}(R) \rightarrow
    H^3_c(\cV(I_3)\backslash\cV(I_3,I_2))\rightarrow\\
    \rightarrow H^3(\cV(I_3))\rightarrow H^3(R)
    \rightarrow,
    \label{p67}
\end{multline}
where
\begin{equation}
R:=\cV(I_3,I_2)=\cV(A_1^2B_0+A_4^2B_1-2A_1A_4A_0,I_2)=\cV(G_2,I_2).
\end{equation}
Consider $R_1=R\cap\cV(B_0)=\cV(B_0,A_0,A_4B_1)\subset\PP^4$ and an
exact sequence
\begin{equation}
\begin{aligned}
    \longrightarrow &H^1(R_1) \longrightarrow H^2_c(R\backslash R_1)\rightarrow H^2_{prim}(R)
    \longrightarrow\\ &H^2_{prim}(R_1) \longrightarrow H^3_c(R\backslash R_1)
    \longrightarrow H^3(R) \longrightarrow H^3(R_1) \longrightarrow.
\end{aligned}\label{p69}
\end{equation}
The variety $R_1=\cV(B_0,A_0,A_4B_1)\subset\PP^4$ and isomorphic to
the union of two lines intersected at one point. The sequence above
simplifies to
\begin{equation}
\begin{aligned}
    0\longrightarrow H^2_c(R\backslash R_1)&\longrightarrow H^2_{prim}(R)
    \longrightarrow\\ &\QQ(-1) \longrightarrow H^3_c(R\backslash R_1)
    \longrightarrow H^3(R) \longrightarrow 0
\end{aligned}\label{p70}
\end{equation}
Now the scheme $R\backslash R_1$ is defined by
\begin{multline}\;\;\;
    \left\{
        \begin{aligned}
        G_2=0\\
        I_2=0\\
        B_0\neq 0\\
        \end{aligned}
    \right.
   \Leftrightarrow \left\{
        \begin{aligned}
        A_1^2B_0+A_4^2B_1-2A_1A_4A_0=0\\
        B_0B_1-A_0^2=0\\
        B_0\neq 0.\!\!\\
        \end{aligned}
    \right. \;\;\;
\end{multline}
We apply Theorem \ref{T1.5} to $R\backslash R_1$, and projecting
further from the point where all the variables but $B_1$ are zero,
we obtain
\begin{equation}
    R\backslash
    R_1\cong\cV(I_2)\backslash\cV(I_2,B_0)\cong\PP^2\backslash\cV(B_0)\cong\AAA^2.
\end{equation}
Hence, $H^2_c(R\backslash R_1)= H^3_c(R\backslash R_1)=0$.
Substituting this into the sequence (\ref{p70}), one gets
\begin{equation}
    H^2_{prim}(R)\cong Q(-1)\quad\textrm{and}\quad H^3(R)=0.
    \label{p73}
\end{equation}
Since $\cV(I_3)\subset\PP^4$ is a hypersurface,
$H^2_{prim}(\cV(I_3))=0$; by (\ref{p73}), the sequence (\ref{p67})
simplifies to
\begin{equation}
    0 \longrightarrow \QQ(-1) \longrightarrow
    H^3_c(\cV(I_3)\backslash\cV(I_3,I_2))
    \longrightarrow H^3(\cV(I_3))\longrightarrow 0
    \label{p74}
\end{equation}
We need to compute  $H^3(\cV(I_3))$. Define
$\hat{Y},S_1\subset\cV(I_3)\subset\PP^4(A_0:A_1:A_4:B_0:B_1)$ by
\begin{equation}
\hat{Y}:=\cV(I_3,I_2^1)
\end{equation}
and $S_1:=\cV(I_3)\backslash \hat{Y}$. One has an exact sequence
\begin{equation}
    \longrightarrow H^3_c(S_1) \longrightarrow H^3(\cV(I_3)) \longrightarrow H^3(\hat{Y})
    \longrightarrow H^4_c(S_1) \longrightarrow.
    \label{p76}
\end{equation}
The scheme $S_1$ is defined by
\begin{equation}
    \left\{
        \begin{aligned}
        B_0I_2^1-G_2'=0\\
        I_2^1\neq 0.\!\!\\
        \end{aligned}
    \right.
\end{equation}
We solve the first equation on $B_0$, and projecting from the point
where all the variables but $B_0$ are zero, we get
\begin{equation}
    S_1\cong\PP^3\backslash\cV(I_2^1).
\end{equation}
The polynomial $I_2^1$ is independent of $A_0$. \emph{Theorem B}
($N=3$, $k=0$, $t=1$) applied to the variety to the right implies
\begin{equation}
\begin{array}{ll}
    H^i_c(S_1)=0\quad\text{for}\;\;i<4,\\
    H^4_c(S_1)\cong H^2_c(\PP^2\backslash\cV(I_2^1))(-1).
\end{array}    \label{p79}
\end{equation}
Since the variety
$\cV(I_2^1)\cong\cV(B_1(A_1+A_4)-A_1^2)\subset\PP^2(A_1:A_4:B_1)$ is
isomorphic to a line, we get
\begin{equation}
H^4_c(S_1)\cong H^2_c(\AAA^2)(-1)=0.
\end{equation}
Substituting this into the sequence (\ref{p76}), we get an
isomorphism
\begin{equation}
   H^3(\cV(I_3))\cong H^3(\hat{Y}).
   \label{p81}
\end{equation}
Now, $\hat{Y}:=\cV(I_3,I_2^1)\cong\cV(G'_2,I_2^1)^{(4)}$ with
\begin{equation}
    G'_2:=I_3-B_0I_2^1=A_0^2C_2+A_4^2B_1-2A_0A_4A_1.
\end{equation}
The defining polynomials of $\hat{Y}\subset\PP^4$ are independent of
$B_0$, we apply \emph{Theorem A} and get
\begin{equation}
    H^3(\hat{Y})\cong H^1(Y)(-1),
    \label{p83}
\end{equation}
where $Y:=\cV(G'_2,I_2^1)\subset\PP^3(A_0:A_1:A_4:B_1)$. Define
$Y_1,Y_c\subset Y\subset\PP^3$ by
\begin{equation}
Y_1=Y\cap\cV(B_1)
\end{equation}
and $Y_c:=Y\backslash Y_1$. One has an exact sequence
\begin{equation}
    \longrightarrow H^0_{prim}(Y)\longrightarrow H^0_{prim}(Y_1)\longrightarrow
    H^1_c(Y_c) \longrightarrow H^1(Y)\longrightarrow H^1(Y_1)\longrightarrow
    \label{p85}
\end{equation}
The variety $Y_1$ is defined by the system of equations
\begin{equation}
    \left\{
        \begin{aligned}
        A_0^2C_2+A_4^2B_1-2A_0A_4A_1=0\\
        B_1C_2-A_1^2=0\\
        B_1=0,\!\!\\
        \end{aligned}
    \right.
\end{equation}
where $C_2=A_1+A_4$. It is easy to see that
\begin{equation}
    Y_1=\cV(B_1,A_1,A_0A_4)\subset\PP^3,
\end{equation}
so $Y_1$ is isomorphic to a union of two points. By \emph{Theorem A}
($N=3$, $k=2$, $t=0$), $H^0_{prim}(Y)=0$, thus the sequence
(\ref{p85}) simplifies to
\begin{equation}
  0\longrightarrow\QQ(0)\longrightarrow H^1_c(Y_c) \longrightarrow H^1(Y)
  \longrightarrow 0.
  \label{p88}
\end{equation}
The scheme $Y_c$ is defined by the following system:
\begin{equation}
    \left\{
        \begin{aligned}
        A_0^2C_2+A_4^2B_1-2A_0A_4A_1=0\\
        B_1C_2-A_1^2=0\\
        B_1\neq 0.\!\!\\
        \end{aligned}
    \right.
    \label{p89}
\end{equation}
By Corollary \ref{C1.4}, one has
\begin{equation}
    G'_2B_1\equiv (Li_2)^2\mod I_2,
\end{equation}
where $Li_2:=A_4B_1-A_0A_1$. Hence, the first equation of the system
implies
\begin{equation}
    A_4B_1-A_0A_1=0 \;\;\;\Leftrightarrow\;\;\;
    A_4=\frac{A_0A_1}{B_1}
\end{equation}
while the second equation gives us
\begin{equation}
    C_2=A_1+A_4=\frac{A_1^2}{B_1} \;\;\;\Leftrightarrow\;\;\;
    A_4=\frac{A_1(A_1-B_1)}{B_1}.
\end{equation}
Thus,
\begin{equation}
    A_1(B_1+A_0-A_1)=0
\end{equation}
on $Y_c$. Projecting from the point where all the variables but
$A_4$ are zero, we get an isomorphism
\begin{equation}
    Y_c\cong
    \cV(A_1(B_1+A_0-A_1))\backslash\cV(A_1(B_1+A_0-A_1),B_1)\subset\PP^2.
\end{equation}
For $Y_c$ we can now write an exact sequence
\begin{equation}
\begin{aligned}
    \rightarrow H^0_{prim}(\cV&(A_1(B_1-A_1-A_0)))\rightarrow
    H^0_{prim}(\cV(A_1(B_1-A_1-A_0),B_1))\\
    \rightarrow &H^1_c(Y_c)\rightarrow
    H^1(\cV(A_1(B_1-A_1-A_0)))\rightarrow .
\end{aligned}\label{p95}
\end{equation}
Changing the variables $B_1:=B_1-A_1-A_0$, we see that the variety
$\cV(A_1(B_1-A_1-A_0))$ is isomorphic to a union of two lines
intersected at one point. Similarly, $\cV(A_1(B_1-A_1-A_0),B_1)$ is
isomorphic to a union of two points. Hence, the sequence (\ref{p95})
implies
\begin{equation}
    H^1_c(Y_c)\cong \QQ(0).
\end{equation}
We return to the sequence (\ref{p88}) and obtain the vanishing
\begin{equation}
    H^1(Y)=0.
    \label{p97}
\end{equation}
By (\ref{p81}), (\ref{p83}) and (\ref{p97}), the sequence
(\ref{p74}) gives us an isomorphism
\begin{equation}
    H_c^3(\cV(I_3)\backslash\cV(I_3,I_2))\cong \QQ(-1).
\end{equation}
Using (\ref{p50}), (\ref{p52}) and (\ref{p66}) we obtain an exact
sequence
\begin{equation}
    0 \longrightarrow H^6_{prim}(\cV(I_4,G_4))\longrightarrow
    \QQ(-2)\xrightarrow{\;\;j\;\;} \QQ(-3)\longrightarrow
    \label{p99}
\end{equation}
The map $j$ must be zero both for Hodge structures and for
$\QQ_{\ell}$-modules with $Gal(\overline{\QQ}/\QQ)$ action. Hence,
the sequence yields
\begin{equation}
     H^6_{prim}(\cV(I_4,G_4))\cong \QQ(-2).
\end{equation}
By (\ref{p38}), we finally get
\begin{equation}
    H^8_{prim}(X)\cong H^6_{prim}(\cV(I_4,G_4))(-1)\cong \QQ(-3).
\end{equation}
\end{proof}
Now we want to recall the classical case of $WS_n$ (see \cite{BEK})
and prove two lemmas used in the next section.
\begin{theorem}\label{T3.4.2}
    Let $X\subset\PP^{2n-1}$ be the graph hypersurface associated to $WS_n$, $n\geq 3$. Then
    \begin{equation}
        H^{2n-2}_{prim}(X)\cong \QQ(-2).
    \end{equation}
\end{theorem}
\begin{proof} We only recall several steps of the proof. The proof
itself can be found in \cite{BEK}.

In this case the matrix is a three-diagonal matrix
$3\diag(B_0,\ldots,B_{n-1}$; $A_0,\ldots,A_{n-2})$ plus some extra
term $A_{n-1}$ at the corners.
    \begin{equation}
M_{WS_n}(A,B):=
 \left(
  \begin{array}{cccccc}
     B_0  & A_0 & \vdots &   0   &  0   & A_{n-1}  \\
     A_0  & B_1 & \vdots &   0  &  0   &  0   \\
      \ldots & \ldots & \ddots & \ldots & \ldots & \ldots\\
      0   &  0  & \vdots & B_{n-3}  & A_{n-3}  & 0  \\
      0   &  0  & \vdots & A_{n-3}  & B_{n-2}  &  A_{n-2} \\
    A_{n-1}   &  0  & \vdots &  0  & A_{n-2}  & B_{n-1}  \\
  \end{array}
\right) \label{o1}
\end{equation}
We deal with the hypersurface $X:=\cV(I_n)\subset\PP^{2n-1}(A,B)$.
Projecting from the point where all the variables but $B_{n-1}$
vanish, we get the following isomorphism
\begin{equation}
    H^{2n-2}(X)\cong H^{2n-2}(\cV(I_n,I_{n-1}))\cong
    H^{2n-4}(\cV(I_{n-1},G_{n-1}))(-1),
    \label{o2}
\end{equation}
where $G_{n-1}$ is such that $I_n=B_{n-1}I_{n-1}-G_{n-1}$. The
variety to the right lives in $\PP^{2n-2}$(no $B_{n-1}$). Next, one
can prove that
\begin{equation}
\begin{aligned}
    H^{2n-4}(\cV(I_{n-1},G_{n-1}))\cong
    H^{2n-4}(\cV(I_{n-1},G_{n-1},I_{n-2}))\\\cong
    H^{2n-6}(\cV(I_{n-1},I_{n-2},A_{n-1}I^1_{n-2}))(-1),
    \label{o3}
\end{aligned}
\end{equation}
where the variety on the right hand side lives in $\PP^{2n-3}$(no
$B_{n-1}$, $A_{n-2}$). The next step of the proof is an isomorphism
\begin{equation}
    H^{2n-6}_{prim}(\cV(I_{n-1},I_{n-2},A_{n-1}I^1_{n-2}))\cong
    H^{2n-7}(\cV(I_{n-1},I_{n-2},A_{n-1},I^1_{n-2}))
    \label{o4}
\end{equation}
The last thing to prove is
\begin{equation}
    H^{2n-7}(Z)=\QQ(0),
    \label{o5}
\end{equation}
where $Z:=\cV(I_{n-1},I_{n-2},I^1_{n-2})\subset\PP^{2n-4}$(no
$B_{n-1}$, $A_{n-2}$, $A_{n-1}$). This is exactly the statement of
Theorem 11.9 in \cite{BEK}. In the next section we will slightly
modify this part of the proof, and apply it in the computation of
$WS_n\times WS_3$. Now, (\ref{o5}) implies
\begin{equation}
    H^{2n-4}_{prim}(\cV(I_{n-1},G_{n-1}))\cong\QQ(-1)
    \label{o6}
\end{equation}
and
\begin{equation}
    H^{2n-2}_{prim}(\cV(I_n))\cong\QQ(-2).
\end{equation}
\end{proof}
The other statement is used several times.
\begin{lemma}\label{L3.4.3}
Let $M$ be the three-diagonal matrix
$3\diag(B_0,\ldots,B_{n-1};A_0,\ldots,A_{n-2})$,\newline $n\geq 2$.
Then
\begin{equation}
    H^{i}_{prim}(\cV(I_n))=0
    \label{o8}
\end{equation}
for $i\leq 2n-3$.
\end{lemma}
\begin{proof} The matrix $M$ looks like
\begin{equation}
M:=
 \left(
  \begin{array}{cccccc}
     B_0  & A_0 & \vdots &   0   &  0   & 0   \\
     A_0  & B_1 & \vdots &   0  &  0   &  0   \\
      \ldots & \ldots & \ddots & \ldots & \ldots & \ldots\\
      0   &  0  & \vdots & B_{n-3}  & A_{n-3}  & 0  \\
      0   &  0  & \vdots & A_{n-3}  & B_{n-2}  &  A_{n-2} \\
      0   &  0  & \vdots &  0  & A_{n-2}  & B_{n-1}  \\
  \end{array}
\right)
\end{equation}
The variety $\cV(I_n)$ lives in $\PP^{2n-2}(A,B)$. For $i\leq 2n-4$
the vanishing holds for dimensional reasons (by \emph{Theorem A}).
The statement means that the interesting cohomology of $\cV(I_n)$ is
not the middle dimensional one,
$H^{mid}(\cV(I_n))=H^{2n-3}(\cV(I_n))$, but the cohomology of the
degree one above. Actually, we will prove that
\begin{equation}
    H^{i}_{prim}(\cV(I_n))=H^{i}_{prim}(\cV(I_n,I_{n-1}))=0
\end{equation}
for $i\leq 2n-3$ using induction on $n$. For $n=2$, the variety
$\cV(B_0B_1-A_0^2)$ is isomorphic to $\PP^1$ and
$\cV(I_2,B_0)=\cV(B_0,A_0)$ is a point, thus
\begin{equation}
    H^{1}(\cV(I_2))=H^1(\cV(I_2,I_1))=0.
\end{equation}
Suppose now that for all three-diagonal matrices of dimension
smaller then $n\times n$ the statement holds. Consider the exact
sequence
\begin{equation}
\begin{aligned}
    \longrightarrow H^{2n-3}_c(\cV(I_n)\backslash\cV(I_n&,I_{n-1}))
    \longrightarrow H^{2n-3}(\cV(I_n))\longrightarrow\\
    H^{2n-3}(\cV(I_n,I_{n-1})&) \longrightarrow
    H^{2n-2}_c(\cV(I_n)\backslash\cV(I_n,I_{n-1}))\longrightarrow.
    \label{o10}
\end{aligned}
\end{equation}
Using the formula
\begin{equation}
    I_n=B_{n-1}I_{n-1}-A_{n-2}^2I_{n-2},
\end{equation}
we can solve the equation $I_n=0$ on $B_{n-1}$; projection from the
point where all the variables but $B_{n-1}$ are zero gives us an
isomorphism
\begin{equation}
    \cV(I_n)\backslash\cV(I_n,I_{n-1})\cong\PP^{2n-3}\backslash
    \cV(I_{n-1}).
    \label{o12}
\end{equation}
Now, $I_{n-1}$ is independent of $A_{n-2}$. We project from the
point where all the variables but $A_{n-2}$ are zero, and get
\begin{equation}
    H^*(\PP^{2n-3}\backslash\cV(I_{n-1}))\cong
    H^{*-2}(\PP^{2n-4}\backslash\cV(I_{n-1}))(-1).
    \label{o13}
\end{equation}
The exact sequence
\begin{equation}
\begin{aligned}
   \longrightarrow H^{*-1}_{prim}(\PP^{2n-4})&\longrightarrow H^{*-1}_{prim}(\cV(I_{n-1}))\longrightarrow\\
   &H^*(\PP^{2n-4}\backslash\cV(I_{n-1}))\longrightarrow H^*_{prim}(\PP^{2n-4})\longrightarrow
\end{aligned}
\end{equation}
gives us an isomorphism
\begin{equation}
    H^*(\PP^{2n-4}\backslash\cV(I_{n-1}))\cong
    H^{*-1}_{prim}(\cV(I_{n-1})).
    \label{o15}
\end{equation}
By the induction hypothesis, for $\cV(I_{n-1})\subset\PP^{2n-4}$(no
$B_{n-1}$, $A_{n-2}$) we have
\begin{equation}
    H^i_{prim}(\cV(I_{n-1}))=0\quad i\leq 2n-5.
\end{equation}
By (\ref{o12}), (\ref{o13}) and (\ref{o15}), we get
\begin{equation}
    H^i_c(\cV(I_n)\backslash\cV(I_n,I_{n-1}))\cong
    H^i_c(\PP^{2n-3}\backslash\cV(I_{n-1}))=0
\end{equation}
for $i\leq (2n-5)+2+1=2n-2$. Thus, the sequence (\ref{o10}) implies
an isomorphism
\begin{equation}
    H^{2n-3}(\cV(I_n))\cong H^{2n-3}(\cV(I_n,I_{n-1})).
\end{equation}
We write $\cV(I_n,I_{n-1})=\cV(I_{n-1},A_{n-2}I_{n-2})^{(2n-2)}$,
and projecting from the point where all variables but $B_{n-1}$ are
zero, we get
\begin{equation}
    H^{2n-3}(\cV(I_n,I_{n-1}))\cong H^{2n-5}(\cV(I_{n-1},A_{n-2}I_{n-2}))(-1)
\end{equation}
with $\cV(I_{n-1},A_{n-2}I_{n-2})\subset\PP^{2n-3}$(no $B_{n-1}$).
One has an exact sequence
\begin{equation}
\begin{aligned}
    H^{2n-6}_{prim}(\cV(&I_{n-1},A_{n-2},I_{n-2}))\longrightarrow
    H^{2n-5}(\cV(I_{n-1},A_{n-2}I_{n-2}))\longrightarrow\\
    &H^{2n-5}(\cV(I_{n-1},A_{n-2}))\oplus
    H^{2n-5}(\cV(I_{n-1},I_{n-2})^{(2n-3)})\longrightarrow.
\end{aligned}
\end{equation}
By the induction assumption, both the term to the left and the sum
to the right vanish, and we get
\begin{equation}
\begin{aligned}
    H^{2n-3}(\cV(I_n))\cong H&^{2n-3}(\cV(I_n,I_{n-1}))\cong\\
    &H^{2n-5}(\cV(I_{n-1},A_{n-2}I_{n-2}))(-1)=0.
\end{aligned}
\end{equation}
\end{proof}
Another lemma with the similar statement will be used in the next
section.
\begin{lemma}\label{L3.4.4}
Let $M$ be the three-diagonal matrix
$3\diag(B_0,\ldots,B_{n-2},C_{n-1};A_0,\ldots,A_{n-2})$, $n\geq 3$
with $C_{n-1}=A_{n-2}$. Then
\begin{equation}
    H^{i}_{prim}(\cV(I_n))=0
    \label{o22}
\end{equation}
for $i\leq 2n-4$.
\end{lemma}
\begin{proof} We work with the matrix
\begin{equation}
M:=
 \left(
  \begin{array}{cccccc}
     B_0  & A_0 & \vdots &   0   &  0   & 0   \\
     A_0  & B_1 & \vdots &   0  &  0   &  0   \\
      \ldots & \ldots & \ddots & \ldots & \ldots & \ldots\\
      0   &  0  & \vdots & B_{n-3}  & A_{n-3}  & 0  \\
      0   &  0  & \vdots & A_{n-3}  & B_{n-2}  &  A_{n-2} \\
      0   &  0  & \vdots &  0  & A_{n-2}  & C_{n-1}  \\
  \end{array}
\right)
\end{equation}
The variety $\cV(I_n(M))$ is taken in $\PP^{2n-3}(A,B)$.
\emph{Theorem A} ($N=2n-3$, $k=1$, $t=0$) implies $H^i(\cV(I_n))=0$
for $i<2n-4$. Consider $H^{2n-4}(\cV(I_n))$.  We can write
\begin{equation}
    I_n=C_{n-1}I_{n-1}-A_{n-2}^2I_{n-2}=A_{n-2}(I_{n-1}-A_{n-2}I_{n-2}).
\end{equation}
Define $S,T\subset\cV(I_n)$ by
\begin{equation}
    T:=\cV(I_{n-1}-A_{n-2}I_{n-2})
\end{equation}
and $S:=\cV(A_{n-2})^{(2n-3)}$. Then $\cV(I_n)=T\cup S$, and one has
the Mayer-Vietoris sequence
\begin{equation}
\begin{aligned}
    \longrightarrow H^{2n-5}_{prim}(S\cap T)\longrightarrow H^{2n-4}_{prim}&(\cV(I_n))
    \longrightarrow\\ &H^{2n-4}_{prim}(S)\oplus H^{2n-4}_{prim}(T)
    \longrightarrow.
    \label{o26}
\end{aligned}
\end{equation}
Since $S\cap T=\cV(A_{n-2},I_{n-1})$, Lemma \ref{L3.4.3} implies
$H^{2n-5}_{prim}(S\cap T)=0$. The variety
$S=\cV(A_{n-2})\subset\PP^{2n-3}$ is isomorphic to $\PP^{2n-4}$,
thus $H^{2n-4}_{prim}(S)=0$. For $T$, we can write
\begin{equation}
    I_{n-1}-A_{n-2}I_{n-2}=(B_{n-2}-A_{n-2})I_{n-2}-A_{n-3}^2I_{n-3}.
\end{equation}
We see that this polynomial does not depend on $A_{n-2}$ or
$B_{n-2}$ but only on $B_{n-2}-A_{n-2}$. We can change the variables
$B_{n-2}:=B_{n-2}-A_{n-2}$, then \emph{Theorem A}($N=2n-3$, $k=1$,
$t=1$) implies $H^{2n-4}(T)=0$. Finally, the sequence (\ref{o26})
implies
\begin{equation}
    H^{2n-4}_{prim}(\cV(I_n))=0.
\end{equation}
\end{proof}

\newpage
\section{Gluings of WS's}
In general, it is not easy to verify whether the graph is
primitively log divergent or not. Nevertheless, we can construct new
primitively log divergent graphs from the existing one's by the
operation of gluing.
\begin{definition}
Let $\Gamma$ and $\Gamma'$ be two graphs, choose two edges $(u,v)\in
E(\Gamma)$ and $(u',v')\in E(\Gamma')$. We define the graph
$\Gamma\times\Gamma'$ as follows. We drop the edges $(u,v)$ and
$(u',v')$, and identify vertices $u$ with $u'$ and $v$ with $v'$. We
say also that $\Gamma\times\Gamma'$ is the \emph{gluing} of $\Gamma$
and $\Gamma'$ along edges $(u,v)$ and $(u',v')$.
\end{definition}
\begin{example}
    The graph $XX$ considered in the previous section is isomorphic to
    \hbox{$WS_3\times WS_3$}.
\end{example}
\begin{theorem}\label{T3.2.3}
The gluing $\Gamma\times\Gamma'$ of two primitively log divergent
graphs $\Gamma$ and $\Gamma'$ (along edges $(u,v)$ and $(u',v')$ )
is again primitively log divergent.
\end{theorem}
\begin{proof}
Suppose that $\Gamma$ and $\Gamma'$ have $2n$ and $2m$ edges
respectively, then $h_1(\Gamma)$ and $h_1(\Gamma')=m$. We can
chose a basis $\{\gamma_1,\ldots, \gamma_n\}$ of $H_1(\Gamma,\ZZ)$
such that the edge $(u,v)$ only appears in $\gamma_n$. Indeed, we
take any basis $\{\gamma_1,\ldots, \gamma_{n-1}\}$ of
$H_1(\Gamma\backslash\{(u,v)\},\ZZ)$ and define $\gamma_n$ to be any
loop containing $(u,v)$, then $\{\gamma_1,\ldots, \gamma_n\}$ form a
basis of $H_1(\Gamma,\ZZ)$. Similarly, we choose a basis
$\delta_1,\ldots,\delta_m$ such that the only appearance of
$(u',v')$ is in $\delta_m$. It follows that the loops
$\{\gamma_1,\ldots, \gamma_{n-1}, \delta_1,\ldots, \delta_{m-1},
\gamma_n\times\delta_m\}$ form a basis of
$H_1(\Gamma\times\Gamma',\ZZ)$. Thus,
$|E(\Gamma\times\Gamma')|=2n+2m-2=2h_1(\Gamma\times\Gamma')$ and
$\Gamma\times\Gamma'$ is logarithmically divergent.

To prove that $\Gamma\times\Gamma'$ is primitively log divergent, we
consider a proper subgraph $\Gamma_0\subset \Gamma\times\Gamma'$ and
define $\Gamma_1$ (respectively $\Gamma_2$) to be the graph
$\Gamma_0\cap\Gamma\cup\{(u,v)\}$ (respectively
$\Gamma_0\cap\Gamma'\cup\{(u',v')\}$). Because the graphs $\Gamma$
and $\Gamma'$ is primitively log divergent, for the subgraphs
$\Gamma_1\subset\Gamma$ and $\Gamma_2\subset\Gamma'$ the
inequalities
\begin{equation}
    |E(\Gamma_1)|\leq 2 h_1(\Gamma_1)\quad\text{and}\quad |E(\Gamma_2)|\leq
    2h_1(\Gamma_2)
\end{equation}
hold, and the inequalities become strict if subgraphs are proper.
Since $\Gamma_0$ is the proper subgraph, at least one of the
subgraphs $\Gamma_1$, $\Gamma_2$ is proper. Thus we get
\begin{equation}
    |E(\Gamma_1)|+|E(\Gamma_2)|<2(h_1(\Gamma_1)+h_1(\Gamma_2))
    \label{m12}
\end{equation}
The number of edges of $\Gamma_0$ equals
\begin{equation}
    |E(\Gamma_0)|=|E(\Gamma_1)|+|E(\Gamma_2)|-2,
\end{equation}
and one has an inequality
\begin{equation}
    h_1(\Gamma_1)+h_1(\Gamma_2)-1\leq h_1(\Gamma_0)
\end{equation}
which becomes an equality if the operation of adding $(u,v)$ to
$\Gamma_0\cap\Gamma$ (or that of $(u',v')$ to $\Gamma_0\cap\Gamma'$)
increases the Betti number. The inequality \ref{m12} implies
\begin{equation}
    |E(\Gamma_0)|<2h_1(\Gamma_0).
\end{equation}
Thus, every subgraph of $\Gamma\times\Gamma'$ is convergent and
$\Gamma\times\Gamma'$ is primitively log divergent.
\end{proof}
\begin{corollary} Every gluing $\Gamma$ of finitely many $GZZ$
graphs (along any pair of edges) is primitively log divergent.
\end{corollary}
\begin{proof} This follows from the fact that a $GZZ$ graph is
primitively log divergent, see Theorem \ref{T2.2.5}.
\end{proof}
Throughout this section we only deal with gluings of $WS$'s graphs.
Our goal here is to analyse the middle dimensional (Betti)
cohomology of hypersurfaces associated to graphs $WS_n\times WS_3$
for $n\geq 4$. The gluing for $WS_n\times WS_3$ goes along some two
$b$-edges (not spokes).
\begin{theorem}
    Let $X$ be the graph hypersurface for the graph $WS_n\times WS_3$, $n\geq 4$.
    For the middle dimensional cohomology $H^{mid}(X)$, one
    has
    \begin{equation}
        gr_6^W(H^{mid}_{prim}(X))=\QQ(-3)\quad\text{and}\;\;
    gr_8^W(H^{mid}_{prim}(X))=\QQ(-4)^{\oplus d},
    \end{equation}
    where $d=0$, 1 or 2, and all other $\gr_i^{W}=0$.
\end{theorem}
\begin{proof}
Fix $n\geq 4$ and consider the graph $WS_n$. We orient the spokes
($a$-edges) $(v_0,v_i)$ as exiting the center $v_0$ and label them
with $e_1$ through $e_n$. The boundary edges $(v_i,v_{i+1})$ (modulo
$n$) are denoted by $e_{n+i}$ and are oriented exiting $v_i$. Now we
rename the last edge $e_{2n}=:e$, play the same game with the graph
$WS_3$, shifting the numeration of edges by $2n-1$, and glue $WS_n$
with $WS_3$ along $e$ and $e_{2n+5}$. Denote the resulting graph by
$\Gamma$. To show the way of constructing the tables and the
matrices associated to this gluing,
we restrict to the case $WS_4\times WS_3$.\bigskip\\
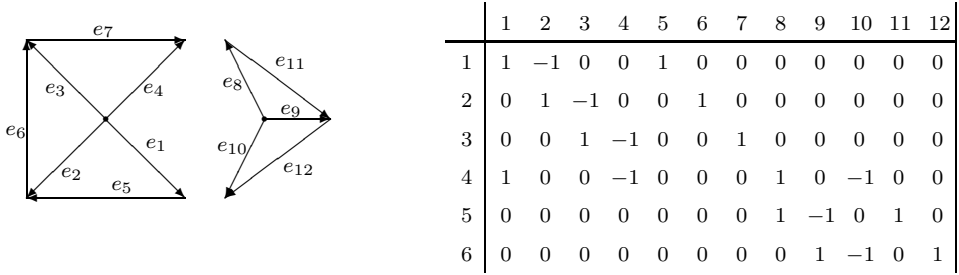
\begin{picture}(0,0)(0,40)
\put(10,20){\vector(0,1){60}} \put(70,20){\vector(-1,0){60}}
\put(10,80){\vector(1,0){60}} \put(40,50){\vector(-1,-1){30}}
\put(40,50){\vector(1,-1){30}} \put(40,50){\vector(-1,1){30}}
\put(40,50){\vector(1,1){30}}
\put(100,50){\vector(-1,-2){15}} \put(100,50){\vector(-1,2){15}}
\put(125,50){\vector(-4,-3){40}} \put(85,80){\vector(4,-3){40}}
\put(100,50){\vector(1,0){25}} %
\put(100,50){\circle*{2}} \put(40,50){\circle*{2}}%
\put(55,39){$\s{e_1}$} \put(23,28){$\s{e_2}$} \put(17,60){$\s{e_3}$}
\put(54,60){$\s{e_4}$} \put(42,23){$\s{e_5}$} \put(2,44){$\s{e_6}$}
\put(35,82){$\s{e_7}$} \put(84,62){$\s{e_8}$}
\put(106,52){$\s{e_9}$} \put(82,38){$\s{e_{10}}$}
\put(104,70){$\s{e_{11}}$} \put(107,30){$\s{e_{12}}$}
\end{picture}
\mathstrut\qquad\qquad\qquad\qquad\qquad\qquad\qquad
\begin{tabular}{p{0.1cm}|p{0.1cm}p{0.1cm}p{0.1cm}p{0.1cm}p{0.1cm}
p{0.1cm}p{0.1cm}p{0.1cm}p{0.1cm}p{0.1cm}p{0.1cm}p{0.1cm}|}
      &\s{1} &\s{2} &\s{3} &\s{4} &\s{5} &\s{6} &\s{7} &\s{8} &\s{9}
      &\s{\!10}&\s{\!11}&\s{\!12}\\\hline
 \s{1}&\sj   &\sk   &\so   &\so   &\sj   &\so   &\so   &\so   &\so   &\so   &\so &\so\\
 \s{2}&\so   &\sj   &\sk   &\so   &\so   &\sj   &\so   &\so   &\so   &\so   &\so &\so\\
 \s{3}&\so   &\so   &\sj   &\sk   &\so   &\so   &\sj   &\so   &\so   &\so   &\so &\so\\
 \s{4}&\sj   &\so   &\so   &\sk   &\so   &\so   &\so   &\sj   &\so   &\sk   &\so &\so\\
 \s{5}&\so   &\so   &\so   &\so   &\so   &\so   &\so   &\sj   &\sk   &\so   &\sj &\so\\
 \s{6}&\so   &\so   &\so   &\so   &\so   &\so   &\so   &\so   &\sj   &\sk   &\so &\sj\\
\end{tabular}
\smallskip\\

\noindent The matrix $M_{\Gamma}$ has two "blocks" coming from the
matrices of $WS_n$ and $WS_3$ intersected by one element which
becomes dependent.
\begin{equation}
M_{\Gamma}(A,B)=
 \left(
  \begin{array}{ccccccccc}
     \s{B_0}& \s{A_0} & \so & \s{\ldots}   & \so &\so  &\s{\!A_{n+1}\!}&\so & \so \\
     \s{A_0}& \s{B_1} &\s{A_1} &  \s{\ldots}  &\so & \so  & \so  &\so &\so \\
     \so   & \s{A_1} &  \s{B_2} &  \s{\ldots}  &\so & \so  & \so &\so &\so \\
     \hdotsfor{9}\\
      \so   &  \so  & \so & \s{\ldots}  &\s{\!B_{n-3}\!} & \s{\!A_{n-3}\!} & \so  &\so &\so\\
      \so   &  \so  & \so & \s{\ldots}  & \s{\!A_{n-3}\!} &\s{B_{n-2}}  &\s{A_{n-2}}  &\so &\so\\
    \s{\!A_{n+1}\!}&\so & \so  &\s{\ldots} & \so &\s{A_{n-2}}&\s{C_{n-1}}& \s{\!A_{n-1}\!} & \s{\!A_{n+2}\!}\\
     \so   &  \so & \so & \s{\ldots} & \so &\so  & \s{A_{n-1}}  & \s{\!B_n\!} & \s{\!A_n\!}\\
     \so   & \so  & \so & \s{\ldots} & \so &\so  & \s{A_{n+2}} & \s{\!A_n\!} & \s{\!B_{n+1}\!}\\
  \end{array}
\right) \label{m17}
\end{equation}
We deal with polynomials in variables $A=\{A_0,A_1,\ldots,
A_{n+2}\}$ and \linebreak $B=\{B_0,\ldots, B_{n-2}, B_n, B_{n+1}\}$,
the element $C_{n-1}$ is equal to the sum
\begin{equation}
    C_{n-1}:=A_{n+1}+A_{n-2}+A_{n-1}+A_{n+2}.
\end{equation}

The hypersurface $X\subset\PP^{2n+3}(A,B)$ is defined by the
vanishing of\linebreak $I_{n+2}(M)=\det M_{\Gamma}$. The middle
dimensional cohomology to compute is
\begin{equation}
H^{mid}(X)=H^{2n+2}(X).
\end{equation}
One has
\begin{equation}
    I_{n+2}=B_{n+1}I_{n+1}-G_{n+1}.
    \label{m20}
\end{equation}
We have an exact sequence
\begin{multline}
 \longrightarrow H^{2n+2}_c(U)\longrightarrow
H^{2n+2}(X)\longrightarrow\\
H^{2n+2}(\cV(I_{n+2},I_{n+1}))\longrightarrow
H^{2n+3}_c(U)\longrightarrow, \label{m21}
\end{multline}
where $U:=X\backslash\cV(I_{n+2},I_{n+1})\subset\PP^{2n+3}(A,B)$.
\begin{lemma}\label{L3.2.6}
One has $H^i_c(U)=0$ for $i\leq 2n+3$.
\end{lemma}
\begin{proof}
Denote by $P_1$ the point where all the variables but $B_{n+1}$ are
zero. The natural projection from the point $P_1$,
$\pi_1:\PP^{2n+3}\backslash P_1\rightarrow \PP^{2n+2}$(no
$B_{n+1}$), induces an isomorphism
\begin{equation}
    U \cong \PP^{2n+2}\backslash\cV(I_{n+1}).
    \label{m22}
\end{equation}
Note that $I_{n+1}$ is independent of $A_n$. Thus, \emph{Theorem
B}($N=2n+2$, $k=0$, $t=1$) already gives us $H^i_c(U)=0$ for $i\leq
2n+2$. For $i=2n+3$ we need to stratify further. Let
$P_2\in\PP^{2n+2}$ be the point where all the variables but $A_n$
are zero. The natural projection $\pi_1:\PP^{2n+2}\backslash
P_2\rightarrow \PP^{2n+1}$(no $B_{n+1}$) gives us an
$\AAA^1$-fibration over $\PP^{2n+1}\backslash\cV(I_{n+1})$, thus
\begin{equation}
    H^{2n+3}_c(U)\cong H^{2n+3}_c(\PP^{2n+2}\backslash\cV(I_{n+1}))
    \cong H^{2n+1}_c(\PP^{2n+1}\backslash\cV(I_{n+1}))(-1).
    \label{m23}
\end{equation}
Next, we make a change of variables: $C_{n-1}:=A_{n+2}$ and denote
by $I_i'$ the image of $I_i$ under this transformation, $i\leq n+1$.
We get an isomorphism
\begin{equation}
    \PP^{2n+1}\backslash\cV(I_{n+1})\cong\PP^{2n+1}\backslash\cV(I_{n+1}').
    \label{m24}
\end{equation}
For the right hand side scheme, we have an exact sequence
\begin{multline}
    \longrightarrow H^{2n}_{prim}(\PP^{2n+1}) \longrightarrow
    H^{2n}_{prim}(\cV(I_{n+1}'))\longrightarrow\\
    H^{2n+1}_c(\PP^{2n+1}\backslash\cV(I_{n+1}'))\longrightarrow
    H^{2n+1}(\PP^{2n+1}) \longrightarrow.
\end{multline}
It follows that
\begin{equation}
    H^{2n+1}_c(\PP^{2n+1}\backslash\cV(I_{n+1}'))\cong
    H^{2n}_{prim}(\cV(I_{n+1}')).
    \label{m26}
\end{equation}
Define $\hat{T}, T_0\subset\PP^{2n+1}$ (no $B_{n+1}$, $A_n$)  by
$\hat{T}:=\cV(I_{n+1}',I_n')$ and $T_0:=\cV(I_{n+1}')\backslash
\hat{T}$. We need to analyse the exact sequence
\begin{equation}
    \longrightarrow H^{2n}_c(T_0) \longrightarrow
    H^{2n}_{prim}(\cV(I_{n+1}'))\longrightarrow H^{2n}_{prim}(\hat{T})\longrightarrow
    \label{m27}
\end{equation}
The open scheme $T_0$ is defined by the system
\begin{equation}
\left\{
        \begin{array}{ll}
         I_{n+1}'=B_nI_n'-A_{n-1}^2I_{n-1}=0 \\
         I_n'\neq 0,
        \end{array}
    \right.
    \label{m28}
\end{equation}
The projection from the point where all the variables but $B_n$
vanish induces an isomorphism
\begin{equation}
    T_0\cong \PP^{2n}\backslash\cV(I_n').
\end{equation}
Since $I_n'$ is independent of $A_{n-1}$, \emph{Theorem B}($N=2n$,
$k=0$, $t=1$), applied to $\PP^{2n}\backslash\cV(I_n')$, implies
\begin{equation}
    H^{2n}_c(T_0)=0.
\end{equation}
Now,
\begin{equation}
\hat{T}=\cV(I_{n+1}',I_n')=\cV(I_n',A_{n-1}I_{n-1})\subset\PP^{2n+1}.
\end{equation}
Both polynomials $I_n'$ and $A_{n-1}I_{n-1}$ are independent of
$B_n$, and projecting from the point where all variables but $B_n$
are zero, we get
\begin{equation}
    H^{2n}_{prim}(\hat{T})\cong H^{2n-2}_{prim}(T)(-1)
    \label{m32}
\end{equation}
with
\begin{equation}
    T:=\cV(I_n',A_{n-1}I_{n-1})\subset\PP^{2n}(\text{no}\; B_{n+1},A_n,B_n).
\end{equation}
Consider $T_1\subset T$ defined by
\begin{equation}
    T_1:=T\cap\cV(I_{n-1}),
\end{equation}
and set $T_{00}=T\backslash T_1$. One has an exact sequence
\begin{equation}
    \longrightarrow H^{2n-2}_c(T_{00}) \longrightarrow
    H^{2n-2}_{prim}(T)\longrightarrow H^{2n-2}_{prim}(T_1)\longrightarrow
     \label{m35}
\end{equation}
The variety
\begin{equation}
T_1=\cV(I_n',I_{n-1})
\end{equation}
is defined by two polynomials both independent of $A_{n-1}$. We
apply \emph{Theorem A} ($N=2n$, $k=2$, $t=2$) to $T_1$  and get
\begin{equation}
    H^{2n-2}_{prim}(T_1)=0.
     \label{m37}
\end{equation}
The open scheme $T_{00}$ is defined by the system
\begin{multline}\;\;
    \left\{
        \begin{aligned}
        I_n'=0\\
        A_{n-1}I_{n-1}=0\\
        I_{n-1}\neq 0\\
        \end{aligned}
    \right.
   \Leftrightarrow \left\{
        \begin{aligned}
        I_n'=0\\
        A_{n-1}=0\\
        I_{n-1}\neq 0\\
        \end{aligned}.
    \right. \;\;\;\;
    \label{m38}
\end{multline}
Write
\begin{equation}
I_n'=C_{n-1}I_{n-1}-G_{n-1},
\end{equation}
where $G_{n-1}$ is independent of $C_{n-1}=A_{n+2}$. We can express
$C_{n-1}$ from the system, and the projection from the point where
all the variables but $C_{n-1}$ are zero induces an isomorphism
\begin{equation}
    T_{00}\cong \PP^{2n-1}\backslash\cV(I_{n-1}).
\end{equation}
The polynomial $I_{n-1}$ is independent of $A_{n-2}$ and $A_{n+1}$.
We apply \emph{Theorem B} ($N=2n-1$, $k=0$, $t=2$) and get
\begin{equation}
    H^{2n-2}_c(T_{00})=0.
     \label{m41}
\end{equation}
By (\ref{m37}) and (\ref{m41}), the exact sequence (\ref{m35})
simplifies to
\begin{equation}
    H^{2n+2}_{prim}(T)=0.
\end{equation}
The sequence (\ref{m27}) together with (\ref{m32}) gives us
\begin{equation}
    H^{2n}_{prim}(\cV(I_{n+1}'))=0.
\end{equation}
The vanishing of $H^{2n+3}(U)$ now follows from (\ref{m23}) and
(\ref{m26}) .
\end{proof}
We return to the sequence (\ref{m21}). Lemma \ref{L3.2.6} yields an
isomorphism
\begin{equation}
    H^{2n+2}(X)\cong H^{2n+2}(\cV(I_{n+2},I_{n+1})).
    \label{m44}
\end{equation}
One has
\begin{equation}
    I_{n+2}=B_{n+1}I_{n+2}-G_{n+1},
\end{equation}
thus
\begin{equation}
    H^{2n+2}(X)\cong H^{2n+2}(\cV(I_{n+1},G_{n+1})^{(2n+3)}).
\end{equation}
Both $I_{n+1}$ and $G_{n+1}$ are independent of $B_{n+1}$, we can
project from the point $P_1$ (see Lemma \ref{L3.2.6}) and get
\begin{equation}
     H^{2n+2}(X)\cong H^{2n}(\cV(I_{n+1},G_{n+1}))(-1)
     \label{m47}
\end{equation}
with the variety on the right hand side living in $\PP^{2n+2}$(no
$B_{n+1}$). Now define
\begin{equation}
\hat{V}:=\cV(I_{n+1},G_{n+1},I_n)\subset\cV(I_{n+1},G_{n+1})\subset\PP^{2n+2}.
\end{equation}
We write an exact sequence
\begin{equation}
\begin{aligned}
\longrightarrow H^{2n}_c(U_1)\longrightarrow
H^{2n}(\cV(&I_{n+1},G_{n+1}))\longrightarrow\\
&H^{2n}(\hat{V})\longrightarrow H^{2n+1}_c(U_1)\longrightarrow,
\label{m49}
\end{aligned}
\end{equation}
where $U_1:=\cV(I_{n+1},G_{n+1})\backslash\hat{V}$. This $U_1$ can
be defined by the system
\begin{equation}
    \left\{
        \begin{array}{ll}
         I_{n+1}=G_{n+1}=0 \\
         I_n\neq 0,
        \end{array}
    \right.
    \label{m50}
\end{equation}
where
\begin{equation}
    G_{n+1}:=A_n^2I_n+A_{n+2}^2B_nI_{n-1}-2A_nA_{n+1}I_{n+1}(n;n+1).
    \label{m51}
\end{equation}
Such $U_1$ were studied in section 1 of chapter 1 (see (\ref{b46})),
and Theorem \ref{T1.7} claims that
\begin{equation}
    U_1\cong U_2:=\cV(I_{n+1})\backslash\cV(I_{n+1},I_n)
    \label{m52}
\end{equation}
with $U_2\subset\PP^{2n+1}$(no $B_{n+1}$, $A_n$). We use the
equality
\begin{equation}
    I_{n+1}=B_nI_n-A_{n-1}^2I_{n-1},
\end{equation}
and we project further from the point $P_3$ where all the
coordinates but $B_n$ vanish. One gets
\begin{equation}
    U_2\cong \PP^{2n}\backslash\cV(I_n).
    \label{m54}
\end{equation}
Now, the only appearance of $A_{n+2}$ in $I_n$ is inside the sum
$C_{n-1}$. We again change the variables as in the lemma above and
come to the polynomial $I_n'$ which has $C_{n-1}:=A_{n+2}$ and does
not depend on $A_{n-1}$. \emph{Theorem B} ($N=2n$, $k=0$, $t=1$)
implies that $H^i(\PP^{2n}\backslash\cV(I_n'))=0$ for $i\leq 2n$ and
\begin{equation}
H^{2n+1}_c(\PP^{2n}\backslash\cV(I_n'))\cong
H^{2n-1}_c(\PP^{2n-1}\backslash\cV(I_n'))(-1). \label{m55}
\end{equation}
Note that $\cV(I_n')$ is exactly the graph hypersurface for $WS_n$.
The exact sequence
\begin{multline}
   \longrightarrow H^{2n-2}_{prim}(\PP^{2n-1})\longrightarrow
   H^{2n-2}_{prim}(\cV(I_n'))\longrightarrow\\ H^{2n-1}_c(\PP^{2n-1}\backslash\cV(I_n'))
    \longrightarrow H^{2n-1}(\PP^{2n-1})\longrightarrow
\end{multline}
implies
\begin{equation}
H^{2n-1}_c(\PP^{2n-1}\backslash\cV(I_n'))\cong
H^{2n-2}_{prim}(\cV(I_n'))\cong \QQ(-2). \label{m57}
\end{equation}
Collect (\ref{m52}), (\ref{m54}),(\ref{m55}) and (\ref{m57})
together; the sequence (\ref{m49}) simplifies to
\begin{equation}
0\longrightarrow H^{2n}(\cV(I_{n+1},G_{n+1}))\longrightarrow
H^{2n}(\hat{V})\longrightarrow \QQ(-3)\longrightarrow. \label{m58}
\end{equation}
We can simplify the polynomial $G_{n+1}$ on
$\hat{V}\subset\PP^{2n+2}$(no $B_{n+1}$). Indeed, Theorem \ref{T1.5}
in the first chapter asserts that when the coefficient $I_n$ of
$A_n^2$ vanishes (see (\ref{m51})), then the rightmost summand in
(\ref{m51}) vanishes as well. Thus we can rewrite
$\hat{V}=\cV(I_n,I_{n+1},A_{n+2}B_nI_{n-1})$. The defining equations
of $\hat{V}$ are independent of $A_n$. We can project from the point
$P_2$ where all the variables but $A_n$ vanish and get
\begin{equation}
    H^{2n}(\hat{V})\cong H^{2n-2}(V)(-1),
    \label{m59}
\end{equation}
where
\begin{equation}
V:=\cV(I_n,I_{n+1},A_{n+2}B_nI_{n-1})\subset\PP^{2n+1}(\text{no}\;
B_{n+1}, A_n).
\end{equation}
By (\ref{m47}), (\ref{m58}) and (\ref{m59}), one has the following
exact sequence
\begin{equation}
0\longrightarrow H^{2n+2}(X)\longrightarrow
H^{2n-2}(V)(-2)\longrightarrow \QQ(-4)\longrightarrow, \label{m61}
\end{equation}
One can avoid polarization and rewrite the exact sequence
\begin{equation}
0\longrightarrow H^{2n+2}_{prim}(X)\longrightarrow
H^{2n-2}_{prim}(V)(-2)\longrightarrow \QQ(-4)\longrightarrow.
 \label{m62}
\end{equation}
Now we attack $V$. Using the equality
\begin{equation}
    I_{n+1}=B_nI_n-A_{n-1}^2I_{n-1},
\end{equation}
we can write
\begin{equation}
    V=\cV(I_n,A_{n+2}B_nI_{n-1},A_{n-1}I_{n-1}).
\end{equation}
Define the subvarieties $V_1,V_2\subset V\subset\PP^{2n+1}$(no
$B_{n+1}$,$A_n$) by
\begin{equation}
\begin{array}{ll}
V_1:=\cV(I_n,B_n,A_{n-1}I_{n-1})\\
V_2:=\cV(I_n,A_{n+2}I_{n-1},A_{n-1}I_{n-1}).
\end{array}
\end{equation}
One has an exact sequence
\begin{equation}
\begin{aligned}
\longrightarrow H^{2n-3}(V_1)&\oplus H^{2n-3}(V_2)\longrightarrow
H^{2n-3}(V_3)\longrightarrow\\ &H^{2n-2}_{prim}(V)\longrightarrow
H^{2n-2}_{prim}(V_1)\oplus H^{2n-2}_{prim}(V_2)\longrightarrow
\end{aligned}
    \label{m66}
\end{equation}
with
\begin{equation}
V_3:=V_1\cap V_2 = \cV(I_n,B_n,A_{n+2}I_{n-1},A_{n-1}I_{n-1}).
\end{equation}
Note that the defining polynomials of $V_2$ are independent of
$B_n$. \emph{Theorem A} ($N=2n+1$, $k=3$, $t=1$) implies
\begin{equation}
 H^i_{prim}(V_2)=0\quad\text{for}\;\; i\leq 2n-2.
 \label{m68}
\end{equation}
\emph{Theorem A} ($N=2n+1$, $k=3$, $t=0$) also gives us
\begin{equation}
    H^i_{prim}(V_1)=0\quad\text{for}\;\; i\leq 2n-3.
     \label{m69}
\end{equation}
We show that $H^{2n-2}_{prim}(V_1)$ vanishes as well. Define
\begin{equation}
    V_{11}:=\cV(I_n,B_n,I_{n-1})\subset V_1\subset\PP^{2n+1}.
\end{equation}
and denote by $V_{10}$ the complement $V_1\backslash V_{11}$. One
has an exact sequence
\begin{equation}
    \longrightarrow H^{2n-2}_c(V_{10}) \longrightarrow H^{2n-2}_{prim}(V_1)
    \longrightarrow H^{2n-2}_{prim}(V_{11})\longrightarrow
    \label{m71}
\end{equation}
The equations of $V_{11}$ do not depend on $A_{n-1}$ or $A_{n+2}$
but only on the sum $A_{n-1}+A_{n+2}$ in $C_{n-1}$. After the change
of variables $C_{n-1}:=A_{n+2}$, we can apply \emph{Theorem A}
($N=2n+1$, $k=3$, $t=1$) and get
\begin{equation}
    H^{2n-2}_{prim}(V_{11})=0.
    \label{m72}
\end{equation}
The open subscheme $V_{10}$ is defined by the system
\begin{equation}
    \left\{
        \begin{aligned}
        B_n=I_n=0\\
        A_{n-1}I_{n-1}=0\\
        I_{n-1}\neq 0\\
        \end{aligned}
    \right.
   \quad \Leftrightarrow \quad\left\{
        \begin{aligned}
        C_{n-1}I_{n-1}-G_{n-1}=0\\
        B_n=A_{n-1}=0\\
        I_{n-1}\neq 0\\
        \end{aligned}.
    \right.
\end{equation}
Using the same change of the variables and expressing
$C_{n-1}:=A_{n+2}$ from the system, we obtain that the projection
from the point $P_4$ where all the variables but $A_{n+2}$ vanish,
induces an isomorphism
\begin{equation}
    V_{10}\cong\PP^{2n-2}\backslash\cV(I_{n-1}).
    \label{m74}
\end{equation}
We have forgotten variables $B_n$, $A_{n-1}$ identifying
$\PP^{2n-2}$(no $B_{n+1}$, $A_n$, $B_n$, $A_{n+2}$, $A_{n-1}$) with
$\cV(B_n,A_{n-1})\subset\PP^{2n}$(no $B_{n+1}$, $A_n$, $A_{n+2}$).
Because the polynomial $I_{n-1}$ is independent of $A_{n-2}$ and
$A_{n+1}$, \emph{Theorem B} ($N=2n-2$, $k=0$, $t=2$) implies
\begin{equation}
    H^i_c(V_{10})=0\quad\text{for}\;\: i\leq 2n-1.
\end{equation}
By (\ref{m71}) and (\ref{m72}), one gets
\begin{equation}
    H^{2n-2}_{prim}(V_1)=0.
    \label{m76}
\end{equation}
We return to the sequence (\ref{m66}). By (\ref{m68}), (\ref{m69})
and (\ref{m76}), we get an isomorphism
\begin{equation}
    H^{2n-2}_{prim}(V)\cong H^{2n-3}(V_3),
    \label{m77}
\end{equation}
where
\begin{equation}
    V_3:=\cV(I_n,B_n,A_{n+2}I_{n-1},A_{n-1}I_{n-1}).
\end{equation}
Consider $V_{31}\subset V_3\subset\PP^{2n+1}$(no $B_{n+1}$, $A_n$)
defined by
\begin{equation}
    V_{31}:=V_3\cap\cV(I_{n-1})=\cV(I_n,B_n,I_{n-1}).
\end{equation}
One has an exact sequence
\begin{equation}
\begin{aligned}
    \longrightarrow H^{2n-4}_{prim}(V_{31})\longrightarrow
    H&^{2n-3}_c(V_3\backslash V_{31})\longrightarrow\\ &H^{2n-3}(V_3)
    \longrightarrow H^{2n-3}(V_{31})\longrightarrow.
\end{aligned}\label{m80}
\end{equation}
\emph{Theorem A} ($N=2n+1$, $k=3$, $t=0$) implies
\begin{equation}
H^i_{prim}(V_{31})=0\quad\text{for}\;\: i\leq 2n-3.
\end{equation}
Thus, the sequence (\ref{m80}) yields an isomorphism
\begin{equation}
    H^{2n-3}(V_3)\cong H^{2n-3}_c(V_3\backslash V_{31}).
    \label{m82}
\end{equation}
The subscheme $V_3\backslash V_{31}$ is defined by the system
\begin{equation}
     \left\{
        \begin{aligned}
        A_{n-1}I_{n-1}=I_n=0\\
        B_n=A_{n+2}I_{n-1}=0\\
        I_{n-1}\neq 0\\
        \end{aligned}
    \right.
   \quad \Leftrightarrow \quad\left\{
        \begin{aligned}
        A_{n-1}=I_n=0\\
        B_n=A_{n+2}=0\\
        I_{n-1}\neq 0\\
        \end{aligned}.
    \right.
\end{equation}
Now we consider $V_3\backslash V_{31}$ as being in $\PP^{2n-2}$(no
$DV_5$) and define $Y,S\subset\PP^{2n-2}$ by
\begin{equation}
\begin{array}{ll}
    Y=\cV(I_n),\\
    S=\cV(I_n,I_{n-1}),
\end{array}
\end{equation}
where by $DV_5$ the set of the dropped variables $\{B_{n+1}, A_{n},
B_n, A_{n+2}, A_{n-1}\}$ is denoted. This gives us an exact sequence
\begin{equation}
    0\longrightarrow H^{2n-4}_{prim}(S)\longrightarrow
    H^{2n-3}(V_3\backslash V_{31}) \longrightarrow
    H^{2n-3}(Y)\longrightarrow
    \label{m85}
\end{equation}
After rewriting
\begin{equation}
    S=\cV(I_{n-1},C_{n-1}I_{n-1}-G_{n-1})=\cV(I_{n-1},G_{n-1}),
\end{equation}
we note that $S$ is exactly the variety which appears in the first
reduction step of the case of $WS_n$ (see Theorem \ref{T3.4.2}), and
we know that
\begin{equation}
    H^{2n-4}(S)\cong\QQ(-1).
    \label{m87}
\end{equation}
The computation of $H^{2n-3}(Y)$ is less easy. The polynomial $I_n$
is similar to the polynomial associated to $WS_n$ with the only
difference that $C_{n-1}$ is not independent and is equal
$A_{n+1}+A_{n-2}$. We start from the upper left corner of the matrix
and write
\begin{equation}
    I_n=B_0I^1_{n-1}-\widetilde{G}_{n-1},
    \label{m88}
\end{equation}
where
\begin{equation}
    \widetilde{G}_{n-1}=A_0^2I^2_{n-2}+A_{n+1}^2I^1_{n-2}+(-1)^{n-1}A_0A_{n+1}S_{n-2}.
    \label{m89}
\end{equation}
Consider $\hat{Y}_1\subset Y\subset\PP^{2n-2}$(no $DV_5$) defined by
\begin{equation}
    \hat{Y}_1:=\cV(I_n,I^1_{n-1}).
\end{equation}
One has
\begin{equation}
\rightarrow H^{2n-3}_c(Y\backslash\hat{Y}_1)\rightarrow H^{2n-3}(Y)
\rightarrow H^{2n-3}(\hat{Y}_1) \rightarrow
H^{2n-2}_c(Y\backslash\hat{Y}_1) \rightarrow.
 \label{m91}
\end{equation}
Using the projection from the point $P_4\subset\PP^{2n-2}$ where all
the variables but $B_0$ vanish, we get an isomorphism
\begin{equation}
    Y\backslash\hat{Y}_1\cong\PP^{2n-3}\backslash\cV(I^1_{n-1}).
\end{equation}
Because $I^1_{n-1}$ is independent of $A_0$, \emph{Theorem
B}($N=2n-3$, $k=0$, $t=1$) implies
\begin{equation}
\begin{aligned}
    &H^{2n-3}_c(Y\backslash\hat{Y}_1)=0\quad\text{and}\\
    H^{2n-2}_c&(Y\backslash\hat{Y}_1)\cong
    H^{2n-4}_c(\PP^{2n-4}\backslash\cV(I^1_{n-1})(-1).
    \label{m93}
\end{aligned}
\end{equation}
One has an exact sequence
\begin{equation}
\begin{aligned}
    \longrightarrow H^{2n-5}(\PP^{2n-4}) &\longrightarrow H^{2n-5}(\cV(I^1_{n-1}))
    \longrightarrow\\ &H^{2n-4}_c(\PP^{2n-4}\backslash\cV(I^1_{n-1}))\longrightarrow
    H^{2n-4}_{prim}(\PP^{2n-4})\longrightarrow
\end{aligned}
\end{equation}
and gets
\begin{equation}
    H^{2n-2}_c(Y\backslash\hat{Y}_1)\cong H^{2n-5}(\cV(I^1_{n-1})).
    \label{m95}
\end{equation}
We can make a change of variables $C_{n-1}:=A_{n+1}$ and, for the
corresponding $I'^1_{n-1}$ (by Lemma \ref{L3.4.3}), we obtain
\begin{equation}
H^{2n-5}(\cV(I'^1_{n-1}))\cong H^{mid}(\cV(I'^1_{n-1}))=0.
\end{equation}
By (\ref{m93}) and (\ref{m95}), the sequence (\ref{m91}) simplifies
to an isomorphism
\begin{equation}
    H^{2n-3}(Y)\cong H^{2n-3}(\hat{Y}_1).
    \label{m97}
\end{equation}
Using the equality (\ref{m88}), we can write
\begin{equation}
    Y_1:=\cV(I_n,I_{n-1})=\cV(I_{n-1}, \widetilde{G}_{n-1})^{(2n+2)}.
\end{equation}
The polynomials to the right are independent of $B_0$. We project
from the point where all the variables but $B_0$ vanish and come to
$Y_1\subset\PP^{2n-3}$(no $DV_5$, $B_0$) defined by
\begin{equation}
    Y_1:=\cV(I_{n-1},G_{n-1}),
\end{equation}
and together with (\ref{m97}) this implies
\begin{equation}
    H^{2n-3}(Y)\cong H^{2n-5}(Y_1)(-1).
    \label{m100}
\end{equation}
Define
\begin{equation}
    \hat{Y}_2:=Y_1\cap\cV(I^2_{n-2})=\cV(I_{n-1},\widetilde{G}_{n-1},I^2_{n-2})
\end{equation}
and consider an exact sequence
\begin{equation}
\rightarrow H^{2n-5}_c(Y_1\backslash \hat{Y}_2)\rightarrow
H^{2n-5}(Y_1) \rightarrow H^{2n-5}(\hat{Y}_2) \rightarrow
H^{2n-4}_c(Y_1\backslash \hat{Y}_2)\rightarrow.
 \label{m102}
\end{equation}
The open subscheme $Y_1\backslash \hat{Y}_2\subset
Y_1\subset\PP^{2n-3}$(no $DV_5$, $B_0$) is defined by the system
\begin{equation}
    \left\{
        \begin{aligned}
        I^1_{n-1}=0\\
        \widetilde{G}_{n-1}=0\\
        I^2_{n-2}\neq 0\\
        \end{aligned}
    \right.
\end{equation}
with $\widetilde{G}_{n-1}$ as in (\ref{m89}). This $G_{n-1}$ has the
same shape as that one we have studied at Section 1 of Chapter 1
with the only difference that the "variables" are chosen from the
zero row and column. We apply Theorem \ref{T1.7} and get an
isomorphism
\begin{equation}
    Y_1\backslash \hat{Y}_2\cong U_3:=\cV(I^1_{n-1})\backslash\cV(I^1_{n-1},
    I^2_{n-2}).
\end{equation}
with $U_3\subset\PP^{2n-4}$(no $DV_5$, $B_0$, $A_0$). Using the
equation
\begin{equation}
    I^1_{n-1}=B_1I^2_{n-2}-A_1^2I^3_{n-3}
    \label{m105}
\end{equation}
and projection from the point where all the variables but $B_1$ are
zero, we get an isomorphism
\begin{equation}
    U_3\cong\PP^{2n-5}\backslash\cV(I^2_{n-2})\cong\PP^{2n-5}\backslash
    \cV(I'^2_{n-2}),
\end{equation}
where $I'^2_{n-2}$ is the image of $I^2_{n-2}$ under the change of
the variables $C_{n-1}:=A_{n+1}$. Note that $I'^2_{n-2}$ is
independent of $A_1$, thus \emph{Theorem B}($N=2n-5$, $k=0$, $t=1$)
implies
\begin{equation}
    H^i_c(Y_1\backslash \hat{Y}_2)\cong H^i_c(U_3)\cong H^i_c(\PP^{2n-5}\backslash
    \cV(I'^2_{n-2}))=0,\;\; i\leq 2n-5.
    \label{m107}
\end{equation}
and
\begin{equation}
    H^{2n-4}_c(\PP^{2n-5}\backslash\cV(I'^2_{n-2}))\cong
    H^{2n-6}_c(\PP^{2n-6}\backslash\cV(I'^2_{n-2}))(-1)
\end{equation}
The localization sequence for $\cV(I'^2_{n-2})\subset \PP^{2n-6}$(no
$DV_5$, $B_0$, $A_0$, $B_1$, $A_1$) gives us
\begin{equation}
    H^{2n-6}_c(\PP^{2n-5}\backslash\cV(I'^2_{n-2}))\cong
    H^{2n-7}(\cV(I'^2_{n-2})).
\end{equation}
We know (see Theorem \ref{L3.4.3}) that
\begin{equation}
    H^{mid}_{prim}(\cV(I'^2_{n-2}))=0.
\end{equation}
Thus,
\begin{equation}
    H^{2n-4}_c(Y_1\backslash \hat{Y}_2)\cong H^{2n-4}_c(U_3)=
    H^{2n-6}_c(\PP^{2n-6}\backslash\cV(I'^2_{n-2}))(-1)=0.
    \label{m111}
\end{equation}
Using (\ref{m107}), (\ref{m111}) and the sequence (\ref{m102}), we
get an isomorphism
\begin{equation}
    H^{2n-5}(Y_1)\cong H^{2n-5}(\hat{Y}_2).
\end{equation}
Now recall that in the case when $I^1_{n-1}=I^2_{n-2}=0$, the
polynomial $\widetilde{G}_{n-1}$ does not depend on $A_0$ (see
Theorem \ref{T1.5}). Thus, by (\ref{m89}),
\begin{equation}
    \hat{Y}_2:=\cV(I_{n-1}^1,\widetilde{G}_{n-1},I^2_{n-2})\cong
    \cV(I^1_{n-1},A_{n+1}I^1_{n-2},I^2_{n-2})^{(2n-3)}
\end{equation}
We project from the point where all the variables but $A_0$ vanish,
and get
\begin{equation}
    H^{2n-5}(Y_1)\cong H^{2n-5}(\hat{Y}_2)\cong H^{2n-7}(Y_2)(-1),
    \label{m114}
\end{equation}
where $Y_2\subset\PP^{2n-4}$(no $DV_5$, $B_0$, $A_0$) is defined by
\begin{equation}
    Y_2:=\cV(I^1_{n-1},A_{n+1}I^1_{n-2},I^2_{n-2}).
\end{equation}
Define the subvarieties $Y_{21},Y_{22}\subset
Y_2\subset\PP^{2n-4}$(no $DV_5$, $B_0$, $A_0$):
\begin{equation}
\begin{array}{ll}
    Y_{21}:=\cV(I^1_{n-1},I^1_{n-2},I^2_{n-2}),\\
    Y_{22}:=\cV(I^1_{n-1},A_{n+1},I^2_{n-2})
\end{array}
\end{equation}
with $Y_{3}:=Y_{21}\cap Y_{22}$. One has an exact sequence
\begin{multline}
    \longrightarrow H^{2n-8}_{prim}(Y_{21})\oplus H^{2n-8}_{prim}(Y_{22})
    \longrightarrow H^{2n-8}_{prim}(Y_3)\longrightarrow\\ H^{2n-7}(Y_2)
    \longrightarrow H^{2n-7}(Y_{21})\oplus H^{2n-7}(Y_{22})
    \longrightarrow.
    \label{m117}
\end{multline}
\emph{Theorem A}($N=2n-4$, $k=3$, $t=0$) implies the vanishing of
the two leftmost summands. Moreover, by (\ref{m105}),
\begin{equation}
    Y_{22}:=\cV(I^1_{n-1},A_{n+1},I^2_{n-2})=\cV(A_{n+1},I^2_{n-2},A_1I^3_{n-3}),
\end{equation}
and the defining equations of $Y_{22}$ are independent of $B_1$.
Thus, by \emph{Theorem A}($N=2n-4$, $k=3$, $t=1$), we obtain
\begin{equation}
    H^{2n-7}(Y_{22})=0.
\end{equation}
Now, we change the variables $C_{n-1}:=A_{n+1}$ and note that the
variety $Y'_{21}$, the image of $Y_{21}$ under this transformation
is exactly the variety appeared in the proof of the $WS_n$ case (was
called $Z_{n-1}$, see Theorem \ref{T3.4.2}). Thus
\begin{equation}
    H^{2n-7}(Y_{21})\cong\QQ(0).
\end{equation}
The sequence (\ref{m117}) simplifies to
\begin{equation}
    0\longrightarrow H^{2n-8}_{prim}(Y_3)\longrightarrow
    H^{2n-7}(Y_2)\longrightarrow\QQ(0)\longrightarrow,
    \label{m121}
\end{equation}
where
\begin{equation}
    Y_3:=\cV(A_{n+1},I^1_{n-1},I^1_{n-2},I^2_{n-2}).
    \label{m122}
\end{equation}
We change the notation and consider $Z\subset\PP^{2n-5}$(no $DV_8$)
defined by
\begin{equation}
    Z:=\cV(I^1_{n-1},I^1_{n-2},I^2_{n-2}),
    \label{m123}
\end{equation}
where $DV_8:=DV_5\cup\{B_0, A_0, A_{n+1}\}$. To abuse the notation,
we write $I^i_j$ for $I^i_j$ after setting $A_{n+1}=0$ (so,
$C_{n-1}=A_{n-2}$). We are interested in
\begin{equation}
    H^{2n-8}_{prim}(Z).
\end{equation}
Define $Z_1,Z_2\subset\PP^{2n-5}$(no $DV_8$) by
\begin{equation}
\begin{array}{ll}
    Z_1:=\cV(I^1_{n-1},I^1_{n-2}),\\
    Z_2:=\cV(I^1_{n-1},I^2_{n-2}),
\end{array}
\end{equation}
then $Z=Z_1\cap Z_2$. We write an exact sequence
\begin{multline}
    \longrightarrow H^{2n-8}_{prim}(Z_1)\oplus H^{2n-8}_{prim}(Z_2)
    \longrightarrow H^{2n-8}_{prim}(Z)\longrightarrow\\ H^{2n-7}(\bar{Z})
    \longrightarrow H^{2n-7}(Z_1)\oplus H^{2n-7}(Z_2)
    \longrightarrow,
    \label{m126}
\end{multline}
where $\bar{Z}:=Z_1\cup Z_2$. By Lemma \ref{L3.4.3}, this sequence
gives us an isomorphism
\begin{equation}
    H^{2n-8}_{prim}(Z)\cong H^{2n-7}(\bar{Z}).
    \label{m127}
\end{equation}
Using Corollary \ref{C1.2}, we obtain
\begin{equation}
    \bar{Z}:=\cV(I^1_{n-1},I^1_{n-2}I^2_{n-2})=\cV(I^1_{n-1},S_{n-2}).
\end{equation}
One can easily compute $S_{n-2}=A_1A_2\ldots A_{n-2}$. Define
$Z_3,Z_4\subset\bar{Z}\subset\PP^{2n-5}$ by
\begin{equation}
\begin{array}{ll}
    Z_3:=\cV(I^1_{n-2},A_{n-2}),\\
    Z_4:=\cV(I^1_{n-2},S_{n-3})
\end{array}
\end{equation}
and $Z_5:=Z_3\cap Z_4$. One has an exact sequence
\begin{multline}
    \longrightarrow
    H^{2n-8}(Z_5)\longrightarrow
    H^{2n-7}(\bar{Z})\longrightarrow\\ H^{2n-7}(Z_3)\oplus
    H^{2n-7}(Z_4) \longrightarrow H^{2n-7}(Z_5)\longrightarrow
    \label{m130}
\end{multline}
Since
\begin{equation}
    I^1_{n-2}=C_{n-1}I^1_{n-2}-A_{n-2}I^1_{n-3}
\end{equation}
with $C_{n-1}=A_{n-2}$,
\begin{equation}
    Z_5:=Z_3\cap Z_4= \cV(I^1_{n-2},A_{n-2},S_{n-3})=\cV(A_{n-2},S_{n-3}).
\end{equation}
The defining polynomials of $Z_5$ are independent of $B_1$ and
$B_2$, \emph{Theorem A}($N=2n-5$, $k=2$, $t=2$) implies
\begin{equation}
    H^i_{prim}(Z_5)=0\quad\text{for}\;\;i\leq 2n-6.
\end{equation}
Similarly,
\begin{equation}
    Z_3:=\cV(I^1_{n-1},A_{n-2})=\cV(A_{n-2})\quad\text{and}\;\;H^{2n-7}(Z_3)=0.
\end{equation}
The sequence (\ref{m130}) now yields
\begin{equation}
    H^{2n-7}(\bar{Z})\cong H^{2n-7}(Z_4)\cong H^{2n-7}(\cV(S_{n-3},I^1_{n-1})).
    \label{m135}
\end{equation}
From now the proof is very similar to that of Theorem 11.9 in
\cite{BEK}.
  We will
analyse the spectral sequence
\begin{equation}
    E^{p,q}_1=\bigoplus_{i_0<\ldots<i_p}
    H^q\big(\cV(A_{i_0},\ldots,A_{i_p},I^1_{n-1})\big)\Rightarrow
    H^{p+q}\big(\cV(S_{n-3},I^1_{n-1})\big).
\end{equation}
First, we have to compute
$H^q(\cV(A_{i_0},\ldots,A_{i_p},I^1_{n-1}))$. For each $j$,\newline
$0\leq j\leq p+1$, define $n_j:=i_j-i_{j-1}$, where $i_{-1}:=0$ and
$i_{p+2}:-1$. We have the partition $n-1=\sum^{p+1}_{0} n_j$.
Computing modulo the ideal $\mathcal{J}$ generated by
$A_{i_0},\ldots,A_{i_p}$, we can factor
\begin{equation}
    I^1_{n-1}\equiv I^1_{n_0}I^{i_0+1}_{n_1}\ldots
    I^{i_{p-1}+1}_{n_p}I^{i_p+1}_{n_{p+1}} \mod \mathcal{J}.
\end{equation}
Each $I^{i_{j-1}+1}_{n_j}$ is a homogeneous function of
$\PP^{2n_j-2}$ for $j<p+1$ and that of $\PP^{2n_j-3}$ for $j=p+1$.
If $n_j=1$, $I_1^k=B_k$ is a homogeneous function on $\PP^0$.

Define linear spaces
\begin{equation}
    L_j\subset\PP^{2n-p-6}(A_1,\ldots,\widehat{A}_{i_0},\ldots
    \widehat{A}_{i_p},\ldots,A_{n-2},B_1,\ldots,B_{n-2})
\end{equation}
and cone maps $\pi:\PP^{2n-p-6}\backslash L_j\rightarrow
\PP^{2n_j-2}$ for $0\leq j\leq p$ and $\pi:\PP^{2n-p-6}\backslash
L_{p+1}\rightarrow \PP^{2n_j-3}$. Then
\begin{equation}
    \cV(A_{i_0},\ldots,A_{t_{p+1}},I^1_{n-1})=\bigcup\pi^{-1}_j(\cV(I^{i_{j-1}+1}_{n_j}))
\end{equation}
In the case $n_j=1$, $j\leq p$, we have simply
$\cV(I^{i_{j-1}+1}_{n_j})\cong L_j$. Set
\begin{equation}
    U_j:=\PP^{2n_j-2}\backslash\cV(I^{i_{j-1}+1}_{n_j})
\end{equation}
for $0\leq j\leq p$ and
$U_{p+1}:\PP^{2n_j-3}\backslash\cV(I^{i_p+1}_{n_{p+1}})$. Define
\begin{equation}
    U:=\PP^{2n-p-6}\big\backslash\bigcup^{p+1}_{j=0}\pi^{-1}(\cV(I^{i_{j-1}+1}_{n_j})).
\end{equation}
It is easy to see that the map $\prod \pi_j:U\rightarrow \prod U_j$
is a $\GG_m^{p+1}$-bundle. Using K\"unneth formula, we get
\begin{equation}
\begin{aligned}
    H^*_c\Big(\PP^{2n-p-6}\backslash\cV(A_{i_0},&\ldots,A_{i_p},I^1_{n-1})\Big)=\\
    &H^*_c(U)\cong H^*_c(\GG^{p+1}_m)\otimes \bigotimes^{p+1}_{j=0} H^*_c(U_j).
\end{aligned}
\end{equation}
Consider the case when $n_j\geq 1$ for some $j$, $1\leq j\leq p$, or
$n_{p+1}>2$. Then, the cohomology groups $H^*_c(U)$ vanish in
degrees less than or equal to
\begin{equation}
    p+1+\sum^p_{j=0}(2n_j-2)+2n_{p+1}-3=2n-p-6.
    \label{ms1}
\end{equation}
This implies
\begin{equation}
    H^i_{prim}(\cV(A_{i_0},\ldots,A_{i_p},I^1_{n-1}))=0\quad\text{for}\;\:
    i\leq 2n-p-7.
\end{equation}
The exceptional case is when $n_j=1$, $j\leq p$ and $n_{p+1}=2$. In
this case $p-4$, $U\cong \GG^{p+1}_m$ and $H^{n-2}(U)\neq 0$
contrary to (\ref{ms1}). We get
\begin{equation}
    E^{n-4,q}_1=H^q(\cV(A_1,\ldots,A_{n-3},I^1_{n-1}))=
    H^q(\cV(\prod^{n-3}_{j=1}B_jB_{n-2}A_{n-2}))
\end{equation}
(here, for $I^{n-3}_2$, we used the change of variables
$B_{n-2}:=B_{n-2}-A_{n-2}$). Stratifying $\cV(B_1\ldots
B_{n-2}A_{n-2})$, using Mayer-Vietoris's sequence and induction, it
is easy to compute $(E^{n-4,n-3}_1)_{prim}=\QQ(0)$, and
$E^{p,q}_2=0$ for $p+q=2n-7$ and $1\leq p\leq n-5$. Moreover,
\begin{equation}
\begin{aligned}
E^{0,2n-7}_2=\ker\Big(\bigoplus^{n-3}_{i=1}H^{2n-7}(&\cV(A_i,I^1_{n-1}))\longrightarrow\\
&\bigoplus_{i_1,i_2} H^{2n-7}(\cV(A_{i_1},A_{i_2},I^1_{n-1}))\Big)
\end{aligned}
\end{equation}
By (\ref{ms1}), $E^{0,2n-7}_2=0$ as well. Now, consider the
sequences
\begin{equation}
    E^{p-r,q+r-1}_r\longrightarrow E^{p,q}_r\longrightarrow
    E^{p+r,q-r+1}_r
\end{equation}
for $r\geq 2$ and $p+q=2n-7$. The group to the left vanishes by
(\ref{ms1}), the group in the middle vanishes for $p\neq n-4$. For
$p-4$ the group to the right vanishes because we have only $n-3$
components. Thus, $E^{p,q}_{r+1}=E^{p,q}_r=E^{p,q}_{\infty}$.
Finally, we get
\begin{equation}
    H^{2n-7}(Z_4)=\QQ(0).
\end{equation}
By (\ref{m122}), (\ref{m123}), (\ref{m127}) and (\ref{m130}), the
sequence (\ref{m121}) takes the form
\begin{equation}
    0\longrightarrow \QQ(0)\longrightarrow
    H^{2n-7}(Y_2)\longrightarrow\QQ(0)\longrightarrow,
    \label{m149}
\end{equation}
Together with (\ref{m100}) and (\ref{m114}), this gives us the exact
sequence
\begin{equation}
    0\longrightarrow \QQ(-2)\longrightarrow
    H^{2n-3}(Y)\longrightarrow\QQ(-2)\longrightarrow.
\end{equation}
Consequently,
\begin{equation}
    H^{2n-3}(Y)\cong \QQ(-2)^{\oplus i}
\end{equation}
for $i=1$ or $i=2$. We return to the sequence (\ref{m85}), and using
(\ref{m87}), we get
\begin{equation}
    0\longrightarrow\QQ(-1)\longrightarrow
    H^{2n-3}(V_3\backslash V_{31}) \longrightarrow
    H^{2n-3}(Y)\longrightarrow.
\end{equation}
By (\ref{m77}) and (\ref{m82}), we can rewrite this sequence
\begin{equation}
    0\longrightarrow\QQ(-1)\longrightarrow
    H^{2n-2}_{prim}(V) \longrightarrow
    H^{2n-3}(Y)\longrightarrow.
\end{equation}
From this, one can describe $H^{2n-2}_{prim}(V)$:
\begin{equation}
    \gr_2^W(H^{2n-2}_{prim}(V))=\QQ(-1),\;\;
    \gr_4^W(H^{2n-2}_{prim}(V))=\QQ(-2)^{\oplus j}
\end{equation}
and all other $gr^i_W$ are zero. Here $0\leq j\leq i$, thus $j$
equals 0, 1 or 2. Now, using the exact sequence (\ref{m62}) we get
finally
\begin{equation}
    \gr_6^W(H^{2n-2}_{prim}(X))=\QQ(-3)\quad\text{and}\;\;
    \gr_8^W(H^{2n-2}_{prim}(X))=\QQ(-4)^{\oplus d},
\end{equation}
where $d=0$, 1 or 2, and all other $gr_i^W=0$.
\end{proof}
We are almost sure that we always have $\gr^W_{min}(X)=\QQ(-3)$ for
$X$ being a graph hypersurface of $WS_n\times WS_m$, $n,m\geq 4$,
but we are not sure that this can be done with our technique even in
the case of $WS_4\times WS_4$.

It is very interesting to understand, what can the (first nontrivial
weight piece of) middle dimensional cohomology for gluings of
primitively log divergent graphs be in general. Consider the very
simple series of examples of such gluings, namely the $WS$'s glued
"in a strip". Fix some $m>2$ and choose $m$ graphs
$\Delta_1,\ldots,\Delta_m$, where $\Delta_i\cong WS_{n_i}$ for
$n_i\geq 3$, $1\leq i\leq m$. As in Theorem \ref{T2.2.5} , $a$-edges
are defined to be the "spokes" while the $b$-edges are the other
(boundary) edges. Define $\Gamma_1:=\Delta_1$ and
$\Gamma_{k+1}:=\Gamma_k\times \Delta_{k+1}$, the gluing of the
graphs $\Gamma_k$ and $\Delta_{k+1}$ along a $b$-edge of $\Gamma_k$
that belongs to $\Delta_k$ and some $b$-edge of $\Delta_{k+1}$,
$k\leq m$. For any graph $\Gamma$ from this series we hope that the
minimal weight piece is still of Tate type
\begin{equation}
    \gr^W_{min}(X_{\Gamma})\cong \QQ(-m-1)^{\oplus d}
\end{equation}
for some $d=d(\Gamma)\geq 1$.


\end{document}